\documentclass[12pt,onecolumn]{IEEEtran}
\usepackage{amssymb}
\usepackage{graphicx, epstopdf}
\usepackage{epstopdf}
\usepackage[tbtags]{amsmath}
\usepackage{amsfonts}
\usepackage{mathrsfs}
\usepackage{cite,url}
\usepackage{upquote}
\usepackage{pifont}
\usepackage{amsmath}
\usepackage{mathtools}

\usepackage{units}

\newtheorem{thm}{Theorem}
\newtheorem{theorem}[thm]{Theorem}
\newtheorem{lemma}[thm]{Lemma}
\newtheorem{definition}[thm]{Definition}
\newtheorem{proposition}[thm]{Proposition}
\newtheorem{corollary}[thm]{Corollary}

\newtheorem{example}{Example}

\newtheorem{remark}{Remark}
\newtheorem{proof}{Proof}

\begin{document}
	\title {Projection Theorems and Estimating Equations for Power-Law Models}
	\author{
		\IEEEauthorblockN{Atin Gayen and M. Ashok Kumar}\\
		\IEEEauthorblockA{\small{
				Discipline of Mathematics\\
				Indian Institute of Technology Palakkad \\
				Kerala 678557, India\\
				Email: atinfordst@gmail.com; ashokm@iitpkd.ac.in}}
	}
	
	\maketitle

\begin{abstract}
We extend projection theorems concerning Hellinger and Jones et al. divergences to the continuous case. These projection theorems reduce certain estimation problems on generalized exponential models to linear problems. We introduce the notion of regularity for generalized exponential models and show that the projection theorems in this case are similar to the ones in discrete and canonical case. We also apply these ideas to solve certain estimation problems concerning Student and Cauchy distributions.
\end{abstract}


\section{Introduction}
\label{1sec:introduction}
Divergence is a non-negative extended real-valued function $D$ defined for any pair of probability distributions $(p,q)$ satisfying $D(p,q) = 0$ if and only if $p = q$. Minimum divergence (or distance) method is popular in statistical inference because of its many desirable properties including robustness and efficiency \cite{BasuSP11B,Pardo06B}. Minimization of information divergence ($I$-divergence) or relative entropy is closely related to the maximum likelihood estimation (MLE) \cite[Lem.~3.1]{CsiszarS04B}. MLE is not a preferred method when the data is contaminated by outliers. However, $I$-divergence can be extended by replacing the logarithmic function by some power function to produce divergences that are robust to outliers \cite{BasuHHJ98J, JonesHHB01J, CichockiA10J}. In this paper, we consider three such families of divergences that are well-known in the context of robust statistics. They are defined as follows.

Let $p$ and $q$ be probability distributions having a common 
support $\mathbb{S}\subseteq \mathbb{R}^d$.
Let $\alpha>0, \alpha\neq 1$.

\begin{enumerate}

	\item[(a)] The Hellinger divergence $D_\alpha$ (also known as 
	Cressie-Read power divergence \cite{CressieR84J} or  power divergence \cite{PatraMPB13J} and, up-to a monotone function, same as R\'enyi divergence \cite{Renyi61J}):
	\begin{eqnarray}
	\label{1defn:D_alpha_divergence}
	D_\alpha(p, q) := 
	\frac{1}{\alpha-1}\left(\int
	p(\textbf{{x}})^\alpha q(\textbf{{x}})^{1-\alpha} 
	d\textbf{{x}} - 1\right).
	\end{eqnarray}
	
	\item[(b)] The Basu et al. divergence $B_\alpha$ (also known as power pseudo-distance \cite{BroniatowskiTV12J, BroniatowskiV12J}, density power divergence \cite{BasuHHJ98J, Naudts04J, Kanamori14J}, $\beta$-divergence \cite{MinamiE02J}):
	\begin{equation}
	\label{1defn:B_alpha_divergence}
	B_{\alpha}(p, q) := \frac{\alpha}{1 - \alpha}\int
	p (\textbf{{x}})q(\textbf{{x}})^{\alpha -1}d\textbf{{x}}
	- \frac{1}{1 - \alpha}\int p(\textbf{{x}})^{\alpha}d\textbf{{x}} 
	+ \int q(\textbf{{x}})^{\alpha}d\textbf{{x}}.
	\end{equation}

	\item[(c)] The Jones et al. divergence $\mathscr{I}_\alpha$ \cite{Sundaresan02ISIT, LutwakYZ05J, Sundaresan07J, FujisawaE08J} (also known as relative $\alpha$-entropy \cite{KumarS15J1, KumarS15J2},  R\'enyi pseudo-distance \cite{BroniatowskiV12J, BroniatowskiTV12J}, logarithmic density power divergence \cite{MajiGB16J},
	projective power divergence \cite{EguchiKK11J}, $\gamma$-divergence \cite{FujisawaE08J, CichockiA10J}):
	\begin{equation}
	\label{1defn:I_alpha_divergence}
	\mathscr{I}_\alpha(p, q):= 
	\frac{\alpha}{1 - \alpha}\ln\int 
	p (\textbf{{x}})q(\textbf{{x}})^{\alpha -1}d\textbf{{x}}
	- \frac{1}{1 - \alpha}\ln\int p(\textbf{{x}})^{\alpha}d\textbf{{x}} + \ln\int  q(\textbf{{x}})^{\alpha}d\textbf{{x}}.
	\end{equation}
	
\end{enumerate}
Throughout the paper we assume that all the integrals are well defined over $\mathbb{S}$. The integrals are with respect to the Lebesgue measure on 
$\mathbb{R}^d$ in the continuous case and with respect to the counting measure in the discrete case. Many well-known divergences fall in the above classes of divergences. For example, Chi-square divergence,
Bhattacharyya distance
\cite{Bhattacharyya43J} and Hellinger distance \cite{Beran77J} fall in the $D_\alpha$-divergence class; Cauchy-Schwarz divergence \cite[Eq. (2.90)]{Principe10B} falls in the  $\mathscr{I}_\alpha$-divergence class; squared Euclidean distance falls in the 
$B_\alpha$-divergence class \cite{BasuHHJ98J}.
All three classes of divergences coincide with the $I$-divergence as $\alpha\rightarrow 1$ \cite{CichockiA10J}, where
\begin{align}
\label{1eqn:kullback-leibler}
I(p, q) :=\int p(\textbf{{x}})\ln\frac{p(\textbf{{x}})}
{q(\textbf{{x}})}d\textbf{{x}}.
\end{align}
In this sense, each of these three classes of divergences can be regarded as a generalization of $I$-divergence.

$D_\alpha$-divergences also arise as generalized cut-off rates in information theory \cite{Csiszar95J1}. $B_\alpha$-divergences belong to the Bregman class which is characterized by transitive projection rules \cite[Eq.~(3.2), Theorem 3]{Csiszar91J}, \cite[Example 3]{Kanamori14J}. $\mathscr{I}_\alpha$-divergences (for $\alpha < 1$) arise in information theory as a redundancy measure in the mismatched cases of guessing \cite{Sundaresan02ISIT}, source coding \cite{KumarS15J2} and encoding of tasks \cite{BunteL14J}. The three classes of  divergences are closely related to robust estimation, for $\alpha>1$ in case of $B_\alpha$ and $\mathscr{I}_\alpha$, and $\alpha<1$ in case of $D_\alpha$, as we shall see now.

Let $\textbf{X}_1,\ldots, \textbf{X}_n$ be an independent and identically distributed (i.i.d.) sample drawn from an unknown distribution $p$. Let us suppose that $p$ is a member of a parametric family of probability
distributions $\Pi = \{p_\theta : \theta\in\Theta\}$, where $\Theta$ is an open subset of $\mathbb{R}^k$ and all $p_\theta$ have a common support $\mathbb{S}\subseteq \mathbb{R}^d$. MLE picks the distribution
$p_{\theta^*}\in \Pi$ that would have most likely caused the sample.
MLE solves the so-called score equation or estimating equation for $\theta$, given by
\begin{eqnarray}
\label{1eqn:score_equation_mle_in_terms_of_sample}
\frac{1}{n}\sum\limits_{j=1}^n s(\textbf{X}_j; \theta) = 0,
\end{eqnarray}
where $s(\textbf{{x}};\theta) := \nabla\ln p_\theta(\textbf{{x}})$, called the score function and $\nabla$ stands for gradient with respect to $\theta$. 
In the discrete case, the above equation can be re-written as
\begin{equation}
\label{1eqn:score_equation_mle_discrete_case}
\sum\limits_{\textbf{{x}}\in \mathbb{S}} p_n(\textbf{{x}})
s(\textbf{{x}};\theta) = 0,
\end{equation}
where $p_n$ is the empirical measure of the sample $\textbf{X}_1,\dots, \textbf{X}_n$.

Let us now suppose that the sample $\textbf{X}_1,\ldots, \textbf{X}_n$
is from a mixture distribution of the
form $p_\epsilon = (1-\epsilon)p + \epsilon\delta$, $\epsilon\in [0,1)$, 
where $p$ is supposed to be a member of $\Pi = \{p_\theta : \theta\in\Theta\}$; $p$ is regarded as
the distribution of ``true'' samples and $\delta$, that of outliers. Assume that support of $\delta$ is a subset of $\mathbb{S}$.
While the usual MLE tries to fit a distribution for $p_\epsilon$, robust estimation tries to fit for $p_\theta$. Throughout the paper, the above will be the setup in all the estimation problems, unless otherwise stated. Thus for robust estimation, one needs to modify
the estimating equation so that the effect of outliers is 
down-weighted. The following modified estimating equation, referred as generalized Hellinger estimating equation, was proposed in \cite{BasuBC97J}, where the score function was weighted by $p_n(\textbf{x})^\alpha p_\theta
(\textbf{x})^{1-\alpha}$ instead of $p_n(\textbf{x})$ in
(\ref{1eqn:score_equation_mle_discrete_case}):
\begin{eqnarray}
\label{1eqn:score_equation_for_D_alpha_discrete_case}
\sum\limits_{\textbf{{x}}\in \mathbb{S}} p_n(\textbf{{x}})^\alpha p_\theta (\textbf{{x}})^{1-\alpha}
s(\textbf{{x}};\theta) = 0,
\end{eqnarray}
where $\alpha\in (0,1)$.
This was proposed based on the following intuition. If $\textbf{{x}}$ is an outlier, then $p_n(\textbf{x})^\alpha p_\theta(\textbf{x})^{1-\alpha}$ will be smaller than $p_n(\textbf{x})$
for sufficiently smaller values of $\alpha$.
Hence the terms corresponding to outliers 
in (\ref{1eqn:score_equation_for_D_alpha_discrete_case}) are down-weighted (c.f. \cite[Section 4.3]{BasuSP11B} and the references therein).

Notice that (\ref{1eqn:score_equation_for_D_alpha_discrete_case}) does not extend to continuous case due
to the appearance of $p_n^\alpha$. However in literature,
to avoid this technical difficulty, some smoothing techniques
such as kernel density estimation \cite[Section 3]{Beran77J},
\cite[Section 3.1,~3.2.1]{BasuSP11B}, Basu-Lindsay
approach \cite[Section 3.5]{BasuSP11B}, Cao et al. modified approach \cite{CaoCF95J} and so on are
used for a continuous estimate of $p_n$. The resulting estimating equation is of the form
\begin{equation}
\label{1eqn:score_equation_for_D_alpha}
\int \widetilde{p}_n (\textbf{{x}})^\alpha p_\theta(\textbf{{x}})^{1-\alpha} s(\textbf{{x}};\theta)
d\textbf{{x}} = 0,
\end{equation}
where $\widetilde{p}_n$ is some continuous estimate of $p_n$. To avoid this smoothing, Broniatowski et al. derived a duality technique where one first finds a dual representation for the Hellinger distance and then minimizes the empirical estimate of this dual representation to find the estimator. The empirical estimate of this dual representation does not require any smoothing.
See \cite{BroniatowskiK09J, TomaB11J, BroniatowskiV12J, BroniatowskiTV12J, Broniatowski14J, Mohamad18J} for details.

The following estimating equation, where the score function is weighted by power of model density
and equated to its hypothetical one, was proposed by Basu et al. \cite{BasuHHJ98J}:
\begin{eqnarray}
\label{1eqn:score_equation_B_alpha}
\frac{1}{n}\sum\limits_{j=1}^n p_{\theta}(\textbf{X}_j)^{\alpha-1} 
s(\textbf{X}_j,\theta) = \int p_{\theta}(\textbf{{x}})^{\alpha} s(\textbf{{x}},\theta)d\textbf{{x}},
\end{eqnarray}
where $\alpha>1$. Motivated by the works of Field and Smith \cite{FieldS94J} and Windham \cite{Windham95J},
an alternate estimating equation, 
where
the weights are further normalized,
was proposed by Jones et al. \cite{JonesHHB01J}:
\begin{eqnarray}
\label{1eqn:score_equation_I_alpha}
\dfrac{\frac{1}{n}\sum\limits_{j=1}^n  p_\theta(\textbf{X}_j)^{\alpha-1}
	s(\textbf{X}_j;\theta)}
{\frac{1}{n}\sum\limits_{j=1}^n  p_\theta(\textbf{X}_j)^{\alpha-1}}
= \dfrac{\int p_\theta(\textbf{{x}})^\alpha s(\textbf{{x}};\theta)d\textbf{{x}}}
{\int p_\theta(\textbf{{x}})^\alpha d\textbf{{x}}},
\end{eqnarray} 
where $\alpha>1$. Notice that (\ref{1eqn:score_equation_B_alpha})
and (\ref{1eqn:score_equation_I_alpha}) do not require the use of empirical
distribution. Hence no smoothing is required in these cases.
The estimators of (\ref{1eqn:score_equation_for_D_alpha}),
(\ref{1eqn:score_equation_B_alpha}) and
(\ref{1eqn:score_equation_I_alpha}) are
consistent and asymptotically normal \cite[Theorem 2]{BasuHHJ98J}, 
\cite[Section 3]{JonesHHB01J}, \cite[Theorem 3]{Beran77J}.
They also satisfy two invariance properties,
one when the underlying model is re-parameterized by a one-one
function of the parameter 
\cite[Section 3.4]{BasuHHJ98J}, and the other when the samples 
are replaced by some of their linear transformation \cite[Theorem 3.1]{TamuraB86J}, 
\cite[Section 3.4]{BasuHHJ98J}.
They
coincide with the ML-estimating equation (\ref{1eqn:score_equation_mle_in_terms_of_sample}) when $\alpha=1$ under the condition that
$\int p_\theta(\textbf{{x}})s(\textbf{{x}};\theta)d\textbf{{x}} = 0$.
The estimating equations (\ref{1eqn:score_equation_mle_in_terms_of_sample}), (\ref{1eqn:score_equation_for_D_alpha}), (\ref{1eqn:score_equation_B_alpha})
and (\ref{1eqn:score_equation_I_alpha}) are, respectively, associated with the divergences in (\ref{1eqn:kullback-leibler}), (\ref{1defn:D_alpha_divergence}), (\ref{1defn:B_alpha_divergence}), and (\ref{1defn:I_alpha_divergence}) in a sense that will be made clear in the following.

Observe that the estimating equations (\ref{1eqn:score_equation_mle_in_terms_of_sample}), (\ref{1eqn:score_equation_for_D_alpha}), (\ref{1eqn:score_equation_B_alpha}),
and (\ref{1eqn:score_equation_I_alpha}) are implications of the first order optimality condition of maximizing,
respectively, the usual log-likelihood function 
\begin{eqnarray}
\label{1eqn:log_likelihood_function}
L(\theta) := \frac{1}{n}\sum\limits_{j=1}^n\ln p_\theta(\textbf{X}_j),
\end{eqnarray}
and the following generalized
likelihood functions
\begin{align}
\label{1eqn:likelihood_function_for_D_alpha}
L_1^{(\alpha)}(\theta) &:= \dfrac{1}{1 - \alpha}
\int \widetilde{p}_n(\textbf{{x}})^\alpha p_\theta (\textbf{{x}})^{1-\alpha} d\textbf{{x}},\\
\label{1eqn:likelihood_function_for_B_alpha}
L_2^{(\alpha)}(\theta) &:= \dfrac{1}{n}\sum\limits_{j=1}^n\left[\dfrac{\alpha p_\theta(\textbf{X}_j)
	^{\alpha-1}-1}{\alpha-1}\right] - \int p_\theta(\textbf{{x}})^\alpha d\textbf{{x}},\\
\label{1eqn:likelihood_function_for_I_alpha}
L_3^{(\alpha)}(\theta) &:= \dfrac{\alpha}{\alpha-1}\ln\left[\dfrac{1}{n}\sum\limits_{j=1}^n p_\theta(\textbf{X}_j)^{\alpha-1}\right] - \ln\left[\int p_\theta(\textbf{{x}})^\alpha 
d\textbf{{x}}\right].
\end{align}
The above likelihood functions (\ref{1eqn:likelihood_function_for_D_alpha}), (\ref{1eqn:likelihood_function_for_B_alpha}) and (\ref{1eqn:likelihood_function_for_I_alpha}) are not defined for $\alpha=1$. However it can be shown that they all coincide with $L(\theta)$ 
as $\alpha\to 1$.

It is easy to see that the probability  distribution  $p_\theta$ that maximizes (\ref{1eqn:likelihood_function_for_D_alpha}),
(\ref{1eqn:log_likelihood_function}), (\ref{1eqn:likelihood_function_for_B_alpha}) or (\ref{1eqn:likelihood_function_for_I_alpha})  is same as,
respectively, the one that minimizes
$D_\alpha(\widetilde{p}_n,p_\theta)$ or the empirical estimates of 
$I(p, p_\theta)$, $B_\alpha(p_\epsilon, p_\theta)$ or $\mathscr{I}_\alpha(p_\epsilon, p_\theta)$.
Thus for MLE or ``robustified MLE," one needs to solve
\begin{eqnarray}
\label{1eqn:reverse_projection}
\inf_{p_\theta\in \Pi} D(\bar{p}_n,p_\theta),
\end{eqnarray}   
where $D$ is either $I$, $D_\alpha$, $B_\alpha$ or $\mathscr{I}_\alpha$;
$\bar{p}_n = p_n$ when $D$ is $I,B_\alpha$ or $\mathscr{I}_\alpha$ and
$\bar{p}_n = \widetilde{p}_n$ when $D$ is $D_\alpha$.
Notice that (\ref{1eqn:score_equation_for_D_alpha}) for $\alpha>1$,
(\ref{1eqn:score_equation_B_alpha}) and
(\ref{1eqn:score_equation_I_alpha}) for $\alpha<1$, do not make
sense in terms of robustness. However, they still serve as first
order optimality condition for the divergence minimization problem
(\ref{1eqn:reverse_projection}). A probability distribution that attains the infimum is known as a reverse $D$-projection of $\bar{p}_n$ on $\Pi$.

A ``dual'' minimization problem is the so-called forward projection problem, where the minimization is over the first argument of the divergence function. Given  a set $\mathbb{C}$ of probability 
distributions with support $\mathbb{S}$ and a probability distribution
$q$ with the same support,
any $p^*\in \mathbb{C}$ that attains 
\begin{eqnarray}
\label{1eqn:forward_projection}
\inf_{p\in \mathbb{C}} D(p,q)
\end{eqnarray}
is called a forward $D$-projection of $q$ on $\mathbb{C}$. Forward projection is usually on a convex set or on an $\alpha$-convex set of probability distributions. Forward projection on a convex set is motivated by the well-known maximum entropy principle of statistical physics \cite{Jaynes83J}. Motivation for forward projection on $\alpha$-convex set comes from the so-called  non-extensive statistical physics \cite{Tsallis88J, TsallisMP98J, Tsallis09B, KumarS15J1}. Forward $I$-projection on convex set was extensively studied by  Csisz\'ar \cite{Csiszar75J, Csiszar84J, Csiszar95J2},  Csisz\'ar and Mat\'u\v{s} \cite{CsiszarM12J, CsiszarM03J},  Csisz\'ar and Shields \cite{CsiszarS04B}, and Csisz\'ar and Tusn\'ady
\cite{CsiszarT84J}.

The forward projections of either of the divergences
in (\ref{1defn:D_alpha_divergence})-(\ref{1eqn:kullback-leibler}) on convex (or $\alpha$-convex) sets of probability
distributions yield a parametric family of probability distributions. A reverse projection on this parametric family turns into a 
forward projection
on the convex (or $\alpha$-convex) set, which further reduces to solving
a system of linear equations. We call such a result a 
projection theorem of the divergence. 
These projection theorems were mainly due to an 
``orthogonal'' relationship between the convex (or the $\alpha$-convex) family and the
associated parametric family. The Pythagorean theorem of the
associated divergence plays a key role in this context.

Projection theorem of the $I$-divergence is due to Csisz\'ar and Shields
\cite[pp. 24]{CsiszarS04B}
where the convex family is a linear family and the associated parametric family is an exponential family.
Projection theorem for $\mathscr{I}_\alpha$-divergence was established by Kumar and Sundaresan \cite[Theorem 18 and Theorem 21]{KumarS15J2}, where the so-called $\alpha$-power-law family ($\mathbb{M}^{(\alpha)}$-family) plays the role of the exponential family. Projection theorem for $D_\alpha$-divergence was established by Kumar and Sason \cite[Theorem 6]{KumarS16J}, where a variant of the $\alpha$-power-law family, called $\alpha$-exponential family ($\mathscr{E}^{(\alpha)}$-family), plays the role of the exponential family and the so-called $\alpha$-linear family plays the role of the linear family. Projection theorem for more general class of Bregman divergences, in which $B_\alpha$ is a subclass, was established by  Csisz\'ar and Mat\'u\v{s} \cite{CsiszarM12J} using techniques from convex analysis. (See also \cite{OharaW10J}.) We observe that the parametric family associated with the projection theorem of $B_\alpha$-divergence is closely related to the $\alpha$-power-law family, which we call a $\mathbb{B}^{(\alpha)}$-family. 

Thus projection theorems enable us to find the estimator (MLE or any of the generalized estimators) as a forward projection if the estimation is done under a specific parametric family. While for MLE the required family is exponential, for the generalized estimations, it is one of the power-law families.
\vspace*{0.1cm}

Our main contributions in this paper are the following.
\begin{itemize}
		\item[1.] The projection theorem for $\mathscr{I}_\alpha$-divergence is known in the literature only for the discrete, canonical case. We first define the associated power-law family $\mathbb{M}^{(\alpha)}$ in a more general setup and establish projection theorem for $\mathscr{I}_\alpha$ on $\mathbb{M}^{(\alpha)}$.
    \item[2.]  We derive the projection theorem for $D_\alpha$-divergence on $\mathscr{E}^{(\alpha)}$-family in more generality by establishing a one-to-one correspondence between this problem and the projection problem concerning $\mathscr{I}_\alpha$-divergence on $\mathbb{M}^{(\alpha)}$-family.
	\item[3.] We introduce the concept of regularity (full-rank family) for the power-law families $\mathbb{B}^{(\alpha)}$, $\mathbb{M}^{(\alpha)}$ and $\mathscr{E}^{(\alpha)}$. We also establish a close relationship among them.
	\item[4.] We show that the Cauchy distributions (also known as $q$-Gaussian distributions \cite{PratoT99J, VignatP07J, OharaW10J, GhoshdastidarDB12ISIT, Matsuzoe17J}) are the escort distributions of the Student distributions \cite{JohnsonV07J}, \cite{EguchiKK11J}.  
	Also Cauchy and Student distributions, respectively, form regular $\mathscr{E}^{(\alpha)}$ and regular $\mathbb{M}^{(\alpha)}$ (and $\mathbb{B}^{(\alpha)}$) families. 
	\item[5.] We find some generalized estimators for the location and scale parameters of the Student and Cauchy distributions using the projection theorems of the Jones et al. and Hellinger divergences. We also observe that these projection theorems can not be applied when the distributions are compactly supported. In this case the estimators should be found on a case by case basis. We find estimators in one such a case and compare it with MLE.
\end{itemize}

Rest of the paper is organized as follows.  In Section \ref{1sec:the_general_form_of_the_alpha_families}, we first generalize the power-law families to the continuous case and show that the Student and Cauchy distributions belong to this class. We also introduce the notion of regularity to these power-law families and establish the relationship among them in this section. In Section \ref{1sec:estimating_equation_for_the_general_family}, we establish projection theorems for the general power-law families. In Section
{\ref{1sec:generalized_estimation_on_student_t_and_cauchy}}, we apply projection theorems to Student and Cauchy distributions to find generalized estimators for their parameters. We also perform some simulations to analyze the efficacy of such estimators. We end the paper with a summary and concluding remarks in Section \ref{1sec:summary_remarks}. In the Appendix, we establish projection theorem of $B_\alpha$-divergence in the discrete case using elementary tools and identify the parametric family associated with this divergence.
\section{The power-law families: definition and examples}
\label{1sec:the_general_form_of_the_alpha_families}
In this section, we define the power-law families associated with the projection theorems of the divergences $B_\alpha$, $\mathscr{I}_\alpha$ and $D_\alpha$ in a more general set-up than they are studied in the literature. We also introduce the concept of regularity for these families. In the literature such a notion for exponential family has been studied, which sometimes is referred as full-rank family (see \cite{LehmannC98B, HoggCM13B}). We then make a comparison among these families. We also show that the well-known Student and Cauchy distributions can be expressed as regular power-law families.

\subsection{\bf{The $\mathbb{B}^{(\alpha)}$-family}}
\label{1subsec:B_alpha_family}
Motivation for $\mathbb{B}^{(\alpha)}$-family comes from the forward projection of $B_\alpha$-divergence on a linear family (See (\ref{1eqn:linear_family})). 
Csisz\'ar and Mat\'u\v{s} \cite{CsiszarM12J} studied a more general form of this family in connection with the projection problems of Bregman divergences. 
\vspace{0.2cm}

\begin{definition}
	\label{1defn:general_B_alpha_family}
\emph{	Consider a family of probability distributions $\lbrace p_\theta : \theta\in\Theta \rbrace$ on $\mathbb{R}^d$, where $\Theta$ is an open subset of $\mathbb{R}^k$. Let $\mathbb{S}$ be the support of $p_\theta$ (which may depend on $\theta$). Let $w = [w_1,\ldots,w_s]^\top$ and $f =$ $[f_1,\ldots,$ $f_s]^\top$, where $w_i: \Theta \to \mathbb{R}$ is differentiable for $i\in\{1,\ldots,s\}$, $f_i: \mathbb{R}^d\to \mathbb{R}$ for $i\in\{1,\ldots,s\}$ and $h:\mathbb{R}^d\to \mathbb{R}$. The family is said to form a $k$-parameter $\mathbb{B}^{(\alpha)}$-family characterized by $h,w,f,\Theta$ and $\mathbb{S}$ if
	\begin{eqnarray}
	\label{1eqn:form_of_general_B_alpha_family}
	{p_\theta({\bf{x}})} = \left\{
	\begin{array}{ll}
	{\big[ h({\bf{x}}) +  F(\theta) + 
		w(\theta)^\top f({\bf{x}}) \big]^{\frac{1}{\alpha - 1}},} &\hbox{~} 
	{\bf{x}}\in \mathbb{S},\\
	{0,} &\hbox{~otherwise},
	\end{array}
	\right.
	\end{eqnarray}
	for some differentiable function $F:\Theta \to \mathbb{R}$. Here $F(\theta)$ is the normalizing factor that can be determined from $\int_{\mathbb{S}} [ h({\bf{x}}) +  F(\theta) + w(\theta)^\top f({\bf{x}})]^{1/(\alpha - 1)}d{\bf{x}} = 1$.}
\end{definition}
\vspace{0.2cm}

The family is said to be regular if, in addition, the following conditions are satisfied.
\begin{enumerate}
	\label{1defn:regular_exponential_family}
	\item[(i)] support $\mathbb{S}$ does not depend on the parameter $\theta$,
	\item[(ii)] number of $\theta_i$'s equals the number of $w_i$'s, that is,
	$s=k$,
	\item[(iii)] the functions 1, $w_1, \ldots, w_s$ are linearly
	independent on $\Theta$,
	\item[(iv)] the
	functions $1,f_1,\ldots,f_s$ are linearly independent on $\mathbb{S}$.
\end{enumerate} 
Further, it is said to be in canonical form if $w_i(\theta) = \theta_i$ for $i\in\{1,\ldots,k\}$. 
The natural parameter space in this case is given by the set of all $\theta\in\mathbb{R}^k$ such that $[h({\bf{x}}) +  F(\theta) + w(\theta)^\top f({\bf{x}})]^{1/(\alpha-1)}>0$ on $\mathbb{S}$ and $\int_{\mathbb{S}} [ h({\bf{x}}) +  F(\theta) + w(\theta)^\top f({\bf{x}})]^{1/(\alpha - 1)}d{\bf{x}} = 1$. 

Observe that $\mathbb{B}^{(\alpha)}$-family is a special case of the family $\mathcal{F}_{[\beta h]}$ in \cite[Eq.~(28)]{CsiszarM12J} with
$h= q$ and $\beta(\cdot, t) = \frac{1}{\alpha -1}[t^{\alpha}-\alpha t +\alpha -t]$.
Bashkirov \cite[Eq. (15)]{Bashkirov04J} derived maximum R\'enyi entropy distribution subject to linear constraints on underlying probability distribution, as in (\ref{1eqn:form_of_general_B_alpha_family}), and called it to be in S-form. Naudts \cite[Ex. 4]{Naudts04J} derived the canonical $\mathbb{B}^{(\alpha)}$-family with $h\equiv 1$ as the `free energy' minimizing distributions with respect to Tsallis entropy.
We shall now see some examples of $\mathbb{B}^{(\alpha)}$-family.

\begin{example}[\textbf{Student distributions}]
	\label{1expl:example_of_B_alpha}
\emph{	Let ${\boldsymbol{\mu}}:= [\mu_1,\ldots,\mu_d]^\top\in\mathbb{R}^d$, $\boldsymbol{\Sigma}:=(\sigma_{ij})$ be a symmetric, positive-definite matrix of order $d$ and $\nu\in\mathbb{R}\setminus \{0\}$. The $d$-dimensional Student distribution with location parameter ${\boldsymbol{\mu}}$, scale parameter $\boldsymbol{\Sigma}$ and degrees of freedom parameter $\nu$, with $\nu\notin[2-d,0]$ when $d\geq 3$, is given by 
	\begin{equation}
	\label{1eqn:student_distribution_in_nu}
	p_{{\boldsymbol{\mu}},\boldsymbol{\Sigma}} ({\bf{x}}) = N_{\boldsymbol{\Sigma},\nu} \Big[ 1 +  \frac{1}{\nu} ({\bf{x}} - {\boldsymbol{\mu}})^\top \boldsymbol{\Sigma} ^{-1} ({\bf{x}} - {\boldsymbol{\mu}}) \Big]_+ ^{-\frac{\nu + d}{2}},
	\end{equation}
	where for a real number $r$, $[r]_+:= \max\{r,0\}$.	
	The support of this distribution is given by
	\begin{displaymath}
	\mathbb{S} = \left\{
	\begin{array}{ll}
	{\big\lbrace {\bf{x}} : ({\bf{x}}
		- \boldsymbol{\mu})^\top \boldsymbol{\Sigma}^{-1} ({\bf{x}} - \boldsymbol{\mu}) < -\nu
		\big\rbrace,} &\hbox{~if~}
	\nu\in (-\infty, \min\{0,2-d\}),\\
	{\mathbb{R}^d}, &\hbox{~if~} 
	\nu \in (0,\infty),
	\end{array}
	\right.
	\end{displaymath}
	and the normalizing factor
	\begin{displaymath}
	N_{\boldsymbol{\Sigma},\nu} := \left\{
	\begin{array}{ll}
	{\frac{\Gamma (1-[\nu/2])}{\Gamma (1- [\nu +d]/2) (-\nu\pi) ^{d/
				2} |\boldsymbol{\Sigma}|^{1/2}},} &\hbox{~if~}
	\nu\in (-\infty, \min\{0,2-d\}),\\\\
	{\frac{\Gamma ([\nu +d]/2)}{\Gamma (\nu/2) (\nu\pi )^{d/2}
			|\boldsymbol{\Sigma}|^{1/2}}, } &\hbox{~if~} 
	\nu \in (0,\infty).
	\end{array}
	\right.
	\end{displaymath}
	It should be noted that Student distributions are not defined for $\nu\in[2-d,0]$ when $d\geq 3$ as (\ref{1eqn:student_distribution_in_nu}) is not integrable in this case. While these distributions do not have finite mean for $\nu\in[0,1]$, they do not have finite variance for $\nu\in[0,2]$. For all other values of $\nu$, the mean and covariance matrix of these distributions are given by $\boldsymbol{\mu}$ and $[\nu/(\nu-2)]\cdot \boldsymbol{\Sigma}$ respectively. Further, (\ref{1eqn:student_distribution_in_nu}) coincides with a normal distribution when $|\nu|\to \infty$.\\
	Let $\alpha := 1-\frac{2}{\nu+d}$. Then $\nu\to +\infty$ and $\nu\to -\infty$ correspond to $\alpha\to 1$ from the left and the right respectively. Let $\theta = [\mu_i,\sigma_{ij}]^\top_{i,j\in\{1,\ldots,d\}, i\le j}$. Then
	(\ref{1eqn:student_distribution_in_nu}) can be re-written as
	\begin{equation}
	\label{1eqn:student_distribution}
	p_{\theta} ({\bf{x}}) = N_{\theta,\alpha} \big[ 1 +  b_\alpha ({\bf{x}}
	- \boldsymbol{\mu})^\top \boldsymbol{\Sigma}^{-1} ({\bf{x}} - \boldsymbol{\mu}) \big]_+ ^{\frac{1}{\alpha - 1}},
	\end{equation}
	where $\alpha\in (-\infty,\min \{0, (d-2)/d\})\cup ((d-2)/d,1)\cup (1,\infty)$, $b_\alpha = 1/\nu =(1-\alpha)/[2 - d (1 - \alpha)]$ and 
	$N_{\theta,\alpha}=N_{\boldsymbol{\Sigma},\nu}$, the normalizing factor. Notice that the Student distribution with $\nu=-d$ is not considered in (\ref{1eqn:student_distribution}) as $\nu=-d$ corresponds to an infinite value of $\alpha$. For a matrix $A=(a_{ij})_{d\times d}$, we use the following notations. 
	\begin{align*}
	\begin{array}{ll}
	{\rm Tr}(A) :=  \sum\limits_{i=1}^d a_{ii},\quad
	\rm{vec} (A) := [a_{11},\dots,a_{1d},a_{21},\ldots,a_{2d},\ldots,a_{d1},\ldots,a_{dd}]^\top,
	\end{array}
	\end{align*}
	that is, $\rm{vec}(A)$ is a column vector of dimension $d^2$ where its $[(i-1)d + j]$-th element is $a_{ij}$ for
	$i,j\in\{1,\dots,d\}$. With these notations
	(\ref{1eqn:student_distribution}) can be re-written,
	for ${\bf{x}}\in \mathbb{S}$, as
	\begin{align}
	\label{1eqn:student_t_as_B_alpha}
	p_{\theta} ({\bf{x}}) &= N_{\theta,\alpha} \big[ 1 + b_\alpha \lbrace
	{\bf{x}}^\top \boldsymbol{\Sigma} ^{-1} {\bf{x}} - 2 \boldsymbol{\mu}^\top \boldsymbol{\Sigma} ^{-1} {\bf{x}}
	+\boldsymbol{\mu}^\top \boldsymbol{\Sigma} ^{-1}\boldsymbol{\mu}\rbrace\big]^{\frac{1}{\alpha -1}}\nonumber\\
	& \stackrel{(a)}{=} \big[ N_{\theta,\alpha}^{\alpha -1} + b_\alpha
	N_{\theta,\alpha}^{\alpha -1} \lbrace \rm{Tr}(
	{\bf{x}}^\top \boldsymbol{\Sigma}^{-1} {\bf{x}}) - 2 \boldsymbol{\mu}^\top\boldsymbol{\Sigma} ^{-1} {\bf{x}}
	+\boldsymbol{\mu}^\top\boldsymbol{\Sigma} ^{-1}\boldsymbol{\mu}\rbrace\big]^{\frac{1}{\alpha -1}}\nonumber\\
	& \stackrel{(b)}{=} \big[ N_{\theta,\alpha}^{\alpha -1} + b_\alpha 
	N_{\theta,\alpha}^{\alpha -1}\lbrace \rm{Tr}(
	\boldsymbol{\Sigma} ^{-1} {\bf{x}}{\bf{x}}^\top) - 2 \boldsymbol{\mu}^\top \boldsymbol{\Sigma} ^{-1} {\bf{x}}
	+\boldsymbol{\mu}^\top\boldsymbol{\Sigma} ^{-1}\boldsymbol{\mu}\rbrace\big]^{\frac{1}{\alpha -1}}\nonumber\\
	& \stackrel{(c)}{=} \big[ N_{\theta,\alpha}^{\alpha -1} + b_\alpha
	N_{\theta,\alpha}^{\alpha -1} \lbrace 
	\rm{vec} (\boldsymbol{\Sigma} ^{-1})^\top \rm{vec}({\bf{x}}{\bf{x}}^\top) - 2 \boldsymbol{\mu}^\top \boldsymbol{\Sigma} ^{-1} {\bf{x}} +\boldsymbol{\mu}^\top\boldsymbol{\Sigma} ^{-1}\boldsymbol{\mu}\rbrace\big]^{\frac{1}{\alpha -1}}\nonumber\\
	& = \big[ 1 + (N^{\alpha - 1}_{\theta,\alpha} +
	b_\alpha N^{\alpha - 1}_{\theta,\alpha}\boldsymbol{\mu}^\top \boldsymbol{\Sigma} ^{-1}\boldsymbol{\mu} -1)
	- 2b_\alpha N^{\alpha - 1}_{\theta,\alpha} (\boldsymbol{\Sigma} ^{-1}\boldsymbol{\mu})^\top {\bf{x}} + b_\alpha N^{\alpha - 1}_{\theta,\alpha} \rm{vec}(\boldsymbol{\Sigma} ^{-1})^\top \rm{vec}({\bf{x}}{\bf{x}}^\top) 
	\big]^{\frac{1}{\alpha -1}},
	\end{align}
	where equality (a) follows because ${\bf{x}}^\top \boldsymbol{\Sigma}^{-1} {\bf{x}}$ is a scalar, (b) follows because $\rm{Tr}(AB) = \rm{Tr}(BA)$, and (c) follows because $\rm{Tr}(AB) = \rm{vec}(A)^\top\rm{vec}(B^\top)$.
	Comparing (\ref{1eqn:student_t_as_B_alpha}) with
	(\ref{1eqn:form_of_general_B_alpha_family}), we conclude that the Student distributions form a $d(d+3)/2$-parameter $\mathbb{B}^{(\alpha)}$-family with
	\begin{align*}
	\begin{array}{ll}
	{\theta = [\mu_i,\sigma_{ij}]^\top_{i,j\in\{1,\ldots,d\}, i\le j},\quad
		F(\theta) = N^{\alpha - 1}_{\theta,\alpha} +
		b_\alpha N^{\alpha - 1}_{\theta,\alpha}\boldsymbol{\mu}^\top \boldsymbol{\Sigma} ^{-1}\boldsymbol{\mu} -1,}\\
	{h({\bf{x}})\equiv 1,\quad w(\theta)=\big[w^{(1)}(\theta),w^{(2)}(\theta)\big]^\top,\quad
		f({\bf{x}}) = \big[f^{(1)}({\bf{x}}),
		f^{(2)}({\bf{x}})\big]^\top,}
	\end{array}
	\end{align*}
	where
	\begin{eqnarray}
	\label{1eqn:parameter_in_student_as_B_alpha}
	\begin{array}{ll}
	{w^{(1)}(\theta) = - 2b_\alpha N^{\alpha - 1}_{\theta,\alpha} {\boldsymbol{\Sigma}} ^{-1}{\boldsymbol{\mu}},\quad
		f^{(1)}({\bf{x}}) = {\bf{x}},} \quad
	{w^{(2)}(\theta) = b_\alpha N^{\alpha - 1}_{\theta,\alpha} \rm{vec} ({\boldsymbol{\Sigma}} ^{-1}),\quad
		f^{(2)}({\bf{x}}) = \rm{vec}({\bf{x}}{\bf{x}}^\top)}.
	\end{array}
	\end{eqnarray}
	}
\end{example} 

The distributions in (\ref{1eqn:student_distribution_in_nu}) for $\nu\in (-\infty,-d)\cup (2,\infty)$ were studied by Johnson and Vignat \cite[Definition 1]{JohnsonV07J} as the maximizer of R\'enyi entropy under covariance constraint, where they classified them as Student-t when $\nu >2$ and Student-r when $\nu <-d$ (see also \cite{Bashkirov04J}). For simplicity we just call them Student distributions. Observe that (\ref{1eqn:student_distribution_in_nu})
for $\nu >0$ is the usual $d$-dimensional $t$-distribution.
\vspace{0.2cm}

\begin{theorem}
	\label{1thm:student_dist_regular_B_alpha}
	The Student distributions for $\nu>0$ (that is, $\alpha\in({(d-2)}/{d},1)$) form a {\em regular} $\mathbb{B}^{(\alpha)}$-family. 
\end{theorem}
\vspace{0.2cm}

\begin{proof}
	Let $\boldsymbol{\Sigma}^{-1}:= (\sigma^{ij})_{d\times d}$ be the inverse of $\boldsymbol{\Sigma}$. The characterizing functions $w^{(i)}$'s and $f^{(i)}$'s in (\ref{1eqn:parameter_in_student_as_B_alpha})
	are given by
	\begin{align*}
	w^{(1)}(\theta)=[w_1(\theta),\ldots,w_d(\theta)]^\top,\quad
	f^{(1)}({\bf{x}})=[f_1({\bf{x}}),\ldots,f_d({\bf{x}})]^\top,
	\end{align*}
	such that
	\begin{align*} 
	w_i(\theta) = -2b_\alpha N_{\theta,\alpha}^{\alpha-1} \sum\limits_{j=1}
	^d \sigma^{ij}\mu_j, \quad f_i({\bf{x}}) =x_i,\quad\text{for}~ i\in\{1,\ldots,d\},
	\end{align*}
	and
	\begin{align*}
	w^{(2)}(\theta)=[w_{ij}(\theta)]^\top_{i,j\in\{1,\ldots,d\},~i\leq j},\quad
	f^{(2)}({\bf{x}})=[f_{ij}({\bf{x}})]^\top_{i,j\in\{1,\ldots,d\},~i\leq j},
	\end{align*}
	where 
	\begin{align*}
	w_{ij}(\theta) = b_\alpha N^{\alpha - 1}_{\theta,\alpha}\sigma^{ij},\quad
	i,j\in\{1,\ldots,d\},~i\leq j,\quad
	f_{ii}({\bf{x}})=x_i^2,~f_{ij}({\bf{x}})=2x_ix_j,\quad
	i,j\in\{1,\ldots,d\},~i<j.
	\end{align*}
	Note that the number of $w_i$'s and $w_{ij}$'s = $d+d+(d-1)+(d-2)+\cdots+1= d(d+3)/2$, which is same as the number of unknown parameters
	$\theta_i$'s. Also $1$, $f_i$'s and $f_{ij}$'s are linearly independent
	on $\mathbb{S}$. Hence it
	remains to show only that $1$, $w_i$'s and $w_{ij}$'s are linearly independent
	on $\Theta$. Suppose that
	\begin{equation*}
	c. 1 +\sum\limits_{i=1}^{d} c_i w_i(\theta)
	+\sum\limits_{i=1}^d\sum\limits_{j=i}^d c_{ij}w_{ij}(\theta)  =0~ \text{for some}~c,c_i,c_{ij}\in\mathbb{R}.
	\end{equation*}
	Dividing both sides by $b_\alpha N_{\theta,\alpha}^{\alpha -1}$,
	\begin{equation}
	\label{1eqn:linearly_independent_condition}
	c b_\alpha^{-1} N_{\theta,\alpha}^{1-\alpha} -2\sum\limits_{i=1}^{d} c_i 
	\Big[ \sum\limits_{j=1}^d \sigma^{ij}\mu_j\Big]
	+  \sum\limits_{i=1}^d\sum\limits_{j=i}^d c_{ij}\sigma^{ij} =0.
	\end{equation}
	Taking partial derivative with respect to $\boldsymbol{\mu}$ in (\ref{1eqn:linearly_independent_condition}),
	\begin{eqnarray}
	\label{1eqn:linearly_independent_condition_4}
	[c_1,\ldots,c_d]\Sigma^{-1} ={\bf{0}}^\top,
	\end{eqnarray}
	where $\bf{0}$ is the zero vector in $\mathbb{R}^d$.
	Since $|\Sigma^{-1}| \neq 0$, from (\ref{1eqn:linearly_independent_condition_4})
	we must have $c_1=\cdots=c_d=0.$
	Thus
	(\ref{1eqn:linearly_independent_condition}) becomes
	\begin{equation}
	\label{1eqn:linearly_independent_condition_3}
	c b_\alpha^{-1}N^{1-\alpha}_{\theta,\alpha} +  \sum\limits_{i=1}^d\sum\limits_{j=i}^d c_{ij}\sigma^{ij} =0.
	\end{equation}
	For
	$i,j\in\{1,\ldots,d\}$, $i\leq j$,
	\begin{equation*}
	\partial_{\sigma^{ij}} (cb_\alpha^{-1}N_{\theta,\alpha}^{1-\alpha})
	= cb_\alpha^{-1}(1-\alpha) N_{\theta,\alpha}^{-\alpha} \partial_{\sigma^{ij}} (N_{\theta,\alpha})
	= \big( [cb_\alpha^{-1}(1-\alpha) N_{\theta,\alpha}^{1-\alpha}]\big/
	2|\boldsymbol{\Sigma}^{-1}|\big)\partial_{\sigma^{ij}} (|\boldsymbol{\Sigma}^{-1}|)
	 =-k_\theta\partial_{\sigma^{ij}} (|\boldsymbol{\Sigma}^{-1}|),
	\end{equation*} 
	where $k_\theta := [cb_\alpha^{-1}(\alpha-1) N_{\theta,\alpha}^{1-\alpha}]\big/ [2|\boldsymbol{\Sigma}^{-1}|]$ and $\partial_{\sigma^{ij}}$ denotes partial derivative with respect to $\sigma^{ij}$.
	Thus differentiating
	(\ref{1eqn:linearly_independent_condition_3})
	with respect to $\sigma^{ij}$, for $i,j\in\{1,\ldots,d\}$, $i\leq j$,
	$c_{ij}= k_\theta\partial_{\sigma^{ij}} (|\boldsymbol{\Sigma}^{-1}|)$.
	Using
	these values in (\ref{1eqn:linearly_independent_condition_3}),
	\begin{eqnarray}
	\label{1eqn:condition_on_sigma}
	cb_\alpha^{-1} N_{\theta,\alpha}^{1 - \alpha}\Big[1+ \tfrac{(\alpha-1)}{2|\boldsymbol{\Sigma}^{-1}|} \sum\limits_{i=1}^d\sum\limits_{j=i}^d \sigma ^{ij}\partial_{\sigma^{ij}}(|\boldsymbol{\Sigma}^{-1}|) \Big] =0.
	\end{eqnarray}
	Since $\boldsymbol{\Sigma}^{-1}$ is symmetric, 
	\begin{equation*}
	\sum\limits_{i=1}^d\sum\limits_{j=i}^d \sigma ^{ij}\partial_{\sigma^{ij}}(|\boldsymbol{\Sigma}^{-1}|) =
	\sum\limits_{i=1}^d \sum\limits_{j=1}^d\sigma ^{ij}
	(\text{cofactor of}~\sigma^{ij}~\text{in}~\boldsymbol{\Sigma}^{-1})= d |\boldsymbol{\Sigma}^{-1}|.
	\end{equation*}
	Using this in (\ref{1eqn:condition_on_sigma}),
	\begin{eqnarray*}
	c\Big[\tfrac{(\alpha-1)}{2|\boldsymbol{\Sigma}^{-1}|}\cdot  d\cdot |\boldsymbol{\Sigma}^{-1}|  + 1\Big] =0.
	\end{eqnarray*}
	Since $\alpha> {(d-2)}/{d}$, then
	$c=0$. This implies $k_\theta=0$ and thus
	$c_{ij}=0$ for all $i,j\in\{1,\ldots,d\}$, $i\leq j$.
	Hence $1$, $w_i$'s and $w_{ij}$'s are linearly independent. This completes the proof.
\end{proof}
\vspace{0.2cm}

\begin{remark}
\emph{	Student distributions for $\nu <0$ do not form a regular $\mathbb{B}^{(\alpha)}$-family as their support, in this case, depends on the unknown parameters.}
\end{remark}
\vspace{0.2cm}

\begin{example}
\emph{	Wigner semi-circle distributions \cite{Wigner58J} form a $\mathbb{B}^{(\alpha)}$-family.}
\end{example}

\subsection{\bf{The $\mathbb{M}^{(\alpha)}$-family}}
\label{1subsec:M_alpha_family}
We now define the parametric family $\mathbb{M}^{(\alpha)}$ associated with the projection theorem of $\mathscr{I}_\alpha$. Kumar and Sundaresan \cite{KumarS15J2} studied this family in the discrete case. \vspace{0.2cm}

\begin{definition}
	\label{1defn:general_M_alpha_family}
\emph{	Let $h,w,f,\Theta$ and $\mathbb{S}$ be as in Definition \ref{1defn:general_B_alpha_family}. The family of probability distributions $\lbrace p_\theta : \theta\in\Theta \rbrace$ is said to form a  $k$-parameter $\alpha$-power-law family or an $\mathbb{M}^{(\alpha)}$-family characterized by $h,w,f,\Theta$ and $\mathbb{S}$ if 
	\begin{eqnarray}
	\label{1eqn:form_of_general_M_alpha_family}
	{p_\theta({\bf{x}})} = \left\{
	\begin{array}{ll}
	{Z(\theta)\big[ h({\bf{x}}) +  
		w(\theta)^\top f({\bf{x}}) \big]^{\frac{1}{\alpha - 1}},} & 
	{\bf{x}}\in \mathbb{S},\\
	{0,} &\text{otherwise}
	,
	\end{array}
	\right.
	\end{eqnarray}
	for some differentiable function $Z:\Theta\to\mathbb{R}$. Here $Z(\theta)$ is the normalizing factor which is given by $Z(\theta) = 1/ \int_{\mathbb{S}}[h({\bf{x}}) +  w(\theta)^\top f({\bf{x}})]^{1/(\alpha-1)} d{\bf{x}}$.
	}
\end{definition}
\vspace{0.2cm}

Bashkirov \cite{Bashkirov04J} derived a specific form of (\ref{1eqn:form_of_general_M_alpha_family}) in connection with R\'enyi entropy maximization and called it to be in Z-form.

The family is said to be regular if, along with (i)-(iii) of Definition \ref{1defn:general_B_alpha_family}, also the functions $f_1,\dots,f_s$ are linearly independent on $\mathbb{S}$. Further, it is said to be canonical if $w_i(\theta)=\theta_i$ for $i\in\{1,\ldots,k\}$. The natural parameter space of this family is the set of all $\theta\in\mathbb{R}^k$ such that $[h({\bf{x}}) +  w(\theta)^\top f({\bf{x}})]^{1/(\alpha-1)}>0$ on $\mathbb{S}$ and $\int_{\mathbb{S}}[h({\bf{x}}) +  w(\theta)^\top f({\bf{x}})]^{1/(\alpha-1)} d{\bf{x}}<\infty$.
\vspace{0.2cm}

\begin{example}
	\label{1expl:example_of_M_alpha_family}
\emph{	The Student distributions in (\ref{1eqn:student_distribution}) can be re-written as
	\begin{align}
	\label{1eqn:student_distribution_as_M_alpha}
	p_\theta ({\bf{x}}) &= N_{\theta,\alpha} \big[ 1 + b_\alpha \lbrace
	{\bf{x}}^\top \boldsymbol{\Sigma} ^{-1} {\bf{x}} - 2 \boldsymbol{\mu}^\top\boldsymbol{\Sigma} ^{-1} {\bf{x}}
	+\boldsymbol{\mu}^\top \boldsymbol{\Sigma} ^{-1}\boldsymbol{\mu}\rbrace\big]_+^{\frac{1}{\alpha -1}}\nonumber\\
	&= N_{\theta,\alpha} \big[ 1 + b_\alpha \lbrace 
	\rm{vec}^\top (\boldsymbol{\Sigma} ^{-1}) \rm{vec}({\bf{x}}{\bf{x}}^\top) - 2 (\boldsymbol{\Sigma} ^{-1}\boldsymbol{\mu})^\top {\bf{x}}
	+\boldsymbol{\mu}^\top \boldsymbol{\Sigma} ^{-1}\boldsymbol{\mu}\rbrace\big]_+^{\frac{1}{\alpha -1}}.
	\end{align}
	Let $S(\theta):= 1+b_\alpha \boldsymbol{\mu}^\top\boldsymbol{\Sigma}^{-1}\boldsymbol{\mu}$. Note that
	$S(\theta)>0$ if $\alpha\in((d-2)/d,1)$. However, when $\alpha\notin((d-2)/d,1)$, we consider the restricted parameter space such that $S(\theta)>0$. Thus (\ref{1eqn:student_distribution_as_M_alpha}) can be re-written,
	for $\bf{x}\in\mathbb{S}$, as
	\begin{equation}
	\label{1eqn:student_distribution_as_M_alpha_1}
	p_\theta ({\bf{x}}) =  
	S(\theta)^{\frac{1}{\alpha -1}} N_{\theta,\alpha}
	\big[ 1 + b_\alpha S(\theta)^{-1}\lbrace 
	\rm{vec}^\top (\boldsymbol{\Sigma} ^{-1}) \rm{vec}({\bf{x}}{\bf{x}}^\top) - 2 (\boldsymbol{\Sigma} ^{-1}\boldsymbol{\mu})^\top {\bf{x}}
	\rbrace\big]^{\frac{1}{\alpha -1}}.
	\end{equation}
	Comparing (\ref{1eqn:student_distribution_as_M_alpha_1})
	and (\ref{1eqn:form_of_general_M_alpha_family}), we see that
	Student distributions form a $d(d+3)/2$-parameter 
	$\mathbb{M}^{(\alpha)}$-family with
	\begin{align*}
	\theta = [\mu_i,\sigma_{ij}]^\top_{i,j\in\{1,\ldots,d\}, i\le j},\quad 
	Z(\theta) = S(\theta)^{\frac{1}{\alpha -1}} N_{\theta,\alpha},\quad h({\bf{x}})\equiv 1,\quad
	w(\theta)=\big[w^{(1)}(\theta), w^{(2)}(\theta)\big]^\top,\quad
	f({\bf{x}}) = \big[ f^{(1)}({\bf{x}}),
	f^{(2)}({\bf{x}})\big]^\top,
	\end{align*}
	where
	\begin{align}
	\label{1eqn:charecterizing_entities_of_student_as_M}
	w^{(1)}(\theta) = - 2b_\alpha S(\theta)^{-1} \boldsymbol{\Sigma} ^{-1}\boldsymbol{\mu}, \quad 
	f^{(1)}({\bf{x}}) = {\bf{x}},\quad
	w^{(2)}(\theta) = b_\alpha S(\theta)^{-1} \rm{vec} (\boldsymbol{\Sigma}^{-1}),\quad 
	f^{(2)}({\bf{x}}) =
	\rm{vec}({\bf{x}}{\bf{x}}^\top).
	\end{align}
	}
\end{example}
This suggests a close relationship between $\mathbb{M}^{(\alpha)}$ and $\mathbb{B}^{(\alpha)}$ families. In the following, we elucidate this fact in more details.
\vspace{0.2cm}

\begin{remark}
	\begin{itemize}
		\item [(a)] $\mathbb{M}^{(\alpha)}$ can be expressed as a $\mathbb{B}^{(\alpha)}$: \emph{Any $p_\theta\in \mathbb{M}^{(\alpha)}$ as in
		(\ref{1eqn:form_of_general_M_alpha_family})
		can be re-written, for $\bf{x}\in\mathbb{S}$, as
		\begin{eqnarray}
		\label{1eqn:M_alpha_to_B_alpha}
		p_\theta({\bf{x}})= [1+ F(\theta) +\widetilde{w}(\theta)^\top\widetilde{f}({\bf{x}})]^{\frac{1}{\alpha -1}}
		\end{eqnarray}
		with $F(\theta)\equiv-1$, $\widetilde{w}(\theta)=\big[Z(\theta)^{\alpha-1},
		Z(\theta)^{\alpha-1}w_1(\theta),\ldots,Z(\theta)^{\alpha-1}
		w_s(\theta)\big]^\top$ and $\widetilde{f}({\bf{x}})$ =
		$\big[h({\bf{x}})$, $f_1({\bf{x}}),\ldots,$ $f_s({\bf{x}})\big]^\top$. This
		implies that these $p_\theta$ also form a $k$-parameter $\mathbb{B}^{(\alpha)}$-family but characterized by $1,\widetilde{f}$ and $\widetilde{w}$.
		\item [(b)] {\em $\mathbb{B}^{(\alpha)}$ can be expressed as an  $\mathbb{M}^{(\alpha)}$}: Any $p_\theta\in\mathbb{B}^{(\alpha)}$ as in
		(\ref{1eqn:form_of_general_B_alpha_family}) can be re-written, 
		for ${\bf{x}}\in \mathbb{S}$, as
		\begin{equation}
		\label{1eqn:B_alpha_to_M_alpha_1}
		p_\theta({\bf{x}}) = Z(\theta)[h({\bf{x}})
		+ \widetilde{w}(\theta)^\top\widetilde{f}({\bf{x}})]^{\frac{1}{\alpha -1}} 
		\end{equation}
		with $Z(\theta)\equiv 1$, $\widetilde{w}(\theta)= [F(\theta), w_1(\theta),\ldots,
		w_s(\theta)]^\top$ and $\widetilde{f}({\bf{x}})$ $=$  $[1,f_1({\bf{x}}),
		\ldots,$ $f_s({\bf{x}})]^\top$, or
		\begin{equation}
		\label{1eqn:B_alpha_to_M_alpha_2}
		p_\theta({\bf{x}}) = Z(\theta)[1
		+ \widetilde{w}(\theta)^\top\widetilde{f}({\bf{x}})]^{\frac{1}{\alpha -1}} 
		\end{equation}
		with $Z(\theta)= F(\theta)^{\frac{1}{\alpha -1}}$, $\widetilde{w}(\theta)= \big[1/F(\theta), w_1(\theta)/F(\theta),\ldots,
		w_s(\theta)/F(\theta)\big]^\top$ and $\widetilde{f}({\bf{x}})=$ 
		$[h({\bf{x}})$, $f_1({\bf{x}})$,
		$\ldots, f_s({\bf{x}})]^\top$, provided $F(\theta)> 0$.
		This implies that
		$p_\theta$ forms a $k$-parameter $\mathbb{M}^{(\alpha)}$-family
		as in (\ref{1eqn:B_alpha_to_M_alpha_1}) or in
		(\ref{1eqn:B_alpha_to_M_alpha_2}). However, as before, the characterizing entities when we view it as 
		a member of $\mathbb{M}^{(\alpha)}$ 
		are not the same as we view it as $\mathbb{B}^{(\alpha)}$.
		\item [(c)] {\em A regular $\mathbb{B}^{(\alpha)}$ may not be a regular  $\mathbb{M}^{(\alpha)}$}: Notice that
		the number of $w_i$'s (and $f_i$'s) is increased when we expressed any member of a $\mathbb{B}^{(\alpha)}$ as an $\mathbb{M}^{(\alpha)}$.
		Thus in general, (\ref{1eqn:B_alpha_to_M_alpha_1})
		or (\ref{1eqn:B_alpha_to_M_alpha_2}) need not define a
		{\em regular} $\mathbb{M}^{(\alpha)}$-family
		even if (\ref{1eqn:form_of_general_B_alpha_family}) defines a {\em
			regular}
		$\mathbb{B}^{(\alpha)}$-family. This can be seen in the following
		example. Consider the 1-dimensional Student distributions with unit variance and $1/3<\alpha<1$:
		\begin{equation}
		\label{1eqn:student-t_one_dim_sigma_fixed}
		p_\mu (x) = N_\alpha [1 +{b}_\alpha (x-\mu)^2]^{\frac{1}{\alpha -1}},
		\end{equation}
		where $N_\alpha$ is the normalizing factor which is independent of
		the unknown parameter $\mu$. This can be viewed as a regular $\mathbb{B}^{(\alpha)}$-family as
		\begin{equation*}
		p_\mu (x) = [(N_\alpha ^{\alpha -1} +N_\alpha ^{\alpha -1}{b}_\alpha x^2) + N_\alpha ^{\alpha -1}{b}_\alpha\mu^2 +
		(-2N_\alpha ^{\alpha -1}{b}_\alpha\mu) x]^{\frac{1}{\alpha -1}}
		\end{equation*}
		with $h(x)=N_\alpha ^{\alpha -1} +N_\alpha ^{\alpha -1}{b}_\alpha x^2$, $F(\mu) = N_\alpha ^{\alpha -1}{b}_\alpha\mu^2$,
		$w_1(\mu)= -2N_\alpha ^{\alpha -1}{b}_\alpha\mu$ and $f_1(x)=x$. Observe that (\ref{1eqn:student-t_one_dim_sigma_fixed}) can be re-written as
		an $\mathbb{M}^{(\alpha)}$-family as
		\begin{equation*}
		p_\mu (x) = N_\alpha[( 1 +{b}_\alpha x^2) + b_\alpha\mu^2 + (-2{b}_\alpha\mu) x]^{\frac{1}{\alpha -1}}
		\end{equation*}
		with $h(x)=1 +{b}_\alpha x
		^2$, $Z(\mu) = N_\alpha $, $w_1(\mu) = {b}_\alpha \mu^2$, $f_1(x) = 1$, $w_2(\mu) = -2{b}_\alpha\mu$ and $f_2(x) = x$. However, this does not define a {\em regular} $\mathbb{M}^{(\alpha)}$ as number of $w_i$'s
		(which is two) is not equal to the number of unknown parameters (which is one).
		\item [(d)] {\em The normalizing factor in a $\mathbb{B}^{(\alpha)}$-family may take negative values}: Unlike the normalizing factor $Z(\theta)$ in $\mathbb{M}^{(\alpha)}$, $F(\theta)$ in an $\mathbb{B}^{(\alpha)}$ may take negative values for some $\theta$ (see Example \ref{1expl:B_alpha_does_not_intersect_L} and \cite[Ex.~3]{KumarS15J2} for a comparison).
		}
		\end{itemize}
\end{remark}
\vspace{0.2cm}

In the following, we find conditions under which a regular $\mathbb{B}^{(\alpha)}$ can be expressed as a regular $\mathbb{M}^{(\alpha)}$.
\vspace{0.2cm}

\begin{proposition}
	\label{1pro:regular_B_alpha_and_regular_M_alpha}
	A regular $\mathbb{B}^{(\alpha)}$-family as in Definition \ref{1defn:general_B_alpha_family} with $h$ being a non-zero constant also forms a regular $\mathbb{M}^{(\alpha)}$-family
	characterized by the same functions $h$ and $f_i$'s, if $1+[F(\theta)/h]>0$ for $\theta\in\Theta$ and one of the following conditions holds.
	\begin{itemize}
		\item[(a)] $F(\theta)$ is identically a constant, or
		\item[(b)] $1$, $F(\theta)$, $w_1(\theta),\dots,w_k(\theta)$ are linearly independent.
	\end{itemize}
\vspace{0.2cm}
	
	\begin{proof}
		Consider a regular
		$\mathbb{B}^{(\alpha)}$-family with $h$ being identically a constant. Then
		from Definition \ref{1defn:general_B_alpha_family}, for $\textbf{x}\in \mathbb{S}$, we have
		\begin{equation}
		\label{1eqn:B_alpha_q_1}
		p_\theta({\bf{x}}) = \big[ h +  F(\theta) + 
		w(\theta)^\top f({\bf{x}}) \big]^{\frac{1}{\alpha - 1}}.
		\end{equation}
		(\ref{1eqn:B_alpha_q_1})
		can be re-written as
		\begin{equation}
		\label{1eqn:B_alpha_to_M_alpha_q_1}
		p_\theta({\bf{x}}) =
		S(\theta)^{\frac{1}{\alpha - 1}} [h +[w(\theta)/S(\theta)]^\top
		f({\bf{x}})]^{\frac{1}{\alpha -1}},
		\end{equation}
		where $S(\theta): = 1 + [F(\theta)\big/h]$. Comparing (\ref{1eqn:B_alpha_to_M_alpha_q_1})
		with (\ref{1eqn:form_of_general_M_alpha_family}) we see that
		$p_\theta$'s form an $\mathbb{M}^{(\alpha)}$-family characterized by $h$, $f_1,\ldots,f_k$.
		This family is regular if $1,w_1(\theta)/S(\theta),\ldots,
		w_k(\theta)/S(\theta)$ are linearly independent.
		Let
		\begin{equation*}
		c_0 +c_1[w_1(\theta)/S(\theta)] +\cdots + c_k[w_k(\theta)/S(\theta)] =0,
		\end{equation*}
		for some scalars $c_i$, $i\in\{0,\ldots,k\}$. Using the value of
		$S(\theta)$, we get
		\begin{equation*}
		c_0 [h+F(\theta)]+hc_1w_1(\theta) +\cdots + hc_kw_k(\theta) =0.
		\end{equation*}
		If $F(\theta)$ is identically a constant then $c_0=c_1=\cdots=c_k=0$,
		since $1,w_1,\dots,w_k$ are linearly independent. Otherwise also $c_0=c_1=\cdots=c_k=0$, if $1$, $F(\theta), w_1(\theta),\ldots,w_k(\theta)$ are
		linearly independent.
	\end{proof}
\end{proposition}

In the view of above proposition, we now show that Student distributions also form a regular $\mathbb{M}^{(\alpha)}$-family.
\vspace{0.2cm}

\begin{corollary}
	\label{1cor:Student_regular_M_alpha}
	Student distributions for $\nu >0$ (that is, $\alpha\in({(d-2)}/d,1)$) form a regular $\mathbb{M}^{(\alpha)}$-family.
\end{corollary}
\vspace{0.2cm}

\begin{proof}
	Recall that, for $\alpha\in({(d-2)}/d,1)$, Student distributions form a regular $\mathbb{B}^{(\alpha)}$-family with $h({\bf{x}})
	\equiv 1$ (Theorem \ref{1thm:student_dist_regular_B_alpha}). Hence, in
	view of Proposition \ref{1pro:regular_B_alpha_and_regular_M_alpha},
	these also form a regular $\mathbb{M}^{(\alpha)}$-family if 
	$1$, $F(\theta)$, $w_i(\theta)$'s and 
	$w_{ij}(\theta)$'s as described in Example
	\ref{1expl:example_of_B_alpha} are linearly independent. 
	To see this, let
	\begin{align}
	\label{1eqn:linearly_independent_condition_M_alpha}
	c. 1+ c_0 F(\theta)  +\sum\limits_{i=1}^{d} c_i w_i(\theta)
	+\sum\limits_{i=1}^d\sum\limits_{j=i}^d c_{ij}w_{ij}(\theta)  =0
	\end{align}
	for some $c,c_i$ and $c_{ij}\in\mathbb{R}$, where
	$F(\theta) = N^{\alpha - 1}_{\theta,\alpha} +
	b_\alpha N^{\alpha - 1}_{\theta,\alpha}\boldsymbol{\mu}^\top \boldsymbol{\Sigma} ^{-1}\boldsymbol{\mu}-1$ and
	$w_i$'s and $w_{ij}$'s are as defined in Theorem \ref{1thm:student_dist_regular_B_alpha}.
	Note that $\partial_{\boldsymbol{\mu}}[\boldsymbol{\mu}^\top \boldsymbol{\Sigma} ^{-1}\boldsymbol{\mu}] = 2(\boldsymbol{\Sigma}^{-1}\boldsymbol{\mu})$
	and $\partial_{\boldsymbol{\mu}}[(\boldsymbol{\Sigma}^{-1}\boldsymbol{\mu})^\top] = \boldsymbol{\Sigma}^{-1}$. Hence
	taking partial derivative with respect to $\boldsymbol{\mu}$ in
	(\ref{1eqn:linearly_independent_condition_M_alpha}), we get	
	\begin{equation}
	\label{1eqn:linearly_independent_condition_M_alpha1}
	2b_\alpha N_{\theta,\alpha}^{\alpha -1}[c_0 \boldsymbol{\Sigma}^{-1}\boldsymbol{\mu}-\boldsymbol{\Sigma}^{-1}\tilde{c}]
	= \bf{0},
	\end{equation}	
	where $\tilde{c}= [c_1,\ldots,c_d]^\top$. Since $|\boldsymbol{\Sigma}^{-1}|\neq 0$,
	from (\ref{1eqn:linearly_independent_condition_M_alpha1}), we have $c_0 \boldsymbol{\mu}-\tilde{c}=\bf{0}$, which, further upon taking partial derivative
	with respect to $\boldsymbol{\mu}$, implies $c_0=c_1=\cdots=c_d=0$. 
	Thus (\ref{1eqn:linearly_independent_condition_M_alpha})
	reduces to (\ref{1eqn:linearly_independent_condition_3})
	of Theorem \ref{1thm:student_dist_regular_B_alpha}. Hence
	proceeding as in Theorem \ref{1thm:student_dist_regular_B_alpha},
	we get $c_{ij}=0$ for $i,j\in\{1,\ldots,d\}$, $i\leq j$. This completes the proof.
\end{proof}
\vspace{0.2cm}

\begin{remark}
\emph{	Student distributions for $\nu \notin (0,\infty)$ do not form a regular
	$\mathbb{M}^{(\alpha)}$ as their support depends on the unknown parameters in this case.}
\end{remark}
\vspace{0.2cm}

\begin{example}
\emph{	Wigner semi-circle distributions also form an $\mathbb{M}^{(\alpha)}$-family.}
\end{example}
\subsection{\bf{The $\mathscr{E}^{(\alpha)}$-family}}
\label{1subsec:the E_alpha_family}
Next we define the parametric family $\mathscr{E}^{(\alpha)}$. This is motivated by the work of Kumar and Sason \cite{KumarS16J} in connection with the forward $D_\alpha$-projections on $\alpha$-linear families where they dealt only the discrete distributions. 
\vspace{0.2cm}

\begin{definition}
	\label{1defn:general_E_alpha_family}
\emph{	Let $h,w,f,\Theta$ and $\mathbb{S}$ be as in Definition \ref{1defn:general_B_alpha_family}. The family of probability distributions $\lbrace p_\theta : \theta\in\Theta \rbrace$ is said to form a  $k$-parameter $\alpha$-exponential family or an $\mathscr{E}^{(\alpha)}$-family
	characterized by $h,w,f,\Theta$ and $\mathbb{S}$ if 
	\begin{eqnarray}
	\label{1eqn:form_of_general_E_alpha_family}
	{p_\theta({\bf{x}})} = \left\{
	\begin{array}{ll}
	{Z(\theta) \big[ h({\bf{x}}) +  
		w(\theta)^\top f({\bf{x}}) \big]^{\frac{1}{1 - \alpha}},} &\hbox{~} 
	{\bf{x}}\in \mathbb{S},\\
	{0,} &\hbox{~otherwise},
	\end{array}
	\right.
	\end{eqnarray}
	for some differentiable function $Z:\Theta\to\mathbb{R}$. Here $Z(\theta)$ is the normalizing factor given by $Z(\theta) = 1/\int_{\mathbb{S}}[ h({\bf{x}}) +  w(\theta)^\top f({\bf{x}}) \big]^{1/{(1-\alpha)}} d{\bf{x}}$.
	}
\end{definition}
\vspace{0.2cm}

The family is said to be regular if, along with (i)-(iii) of Definition \ref{1defn:general_B_alpha_family}, also the functions $f_1,\ldots,f_s$ are linearly independent on $\mathbb{S}$. Further, it is said to be canonical if $w_i(\theta)=\theta_i$ for $i\in\{1,\ldots,k\}$. The natural parameter space in this case is given by the set of all $\theta\in\mathbb{R}^k$ such that $[ h({\bf{x}}) +  w(\theta)^\top f({\bf{x}}) \big]^{1/{(1-\alpha)}}>0$ on $\mathbb{S}$ and $\int_{\mathbb{S}}[ h({\bf{x}}) +  w(\theta)^\top f({\bf{x}}) \big]^{1/{(1-\alpha)}} d{\bf{x}}<\infty$.

Observe that, (\ref{1eqn:form_of_general_E_alpha_family}) with $h\equiv 1$ forms a $\phi$-exponential family for $\phi(x) = x^\alpha$ studied in \cite{Naudts04J, OharaW10J, Matsuzoe17J} (see also the references therein). However, if $h$ is not identically a constant, these two families are not the same.
\vspace{0.2cm}

\begin{remark}
	\label{1rem:M_alpha_E_alpha}
	\emph{Connection between $\mathscr{E}^{(\alpha)}$ and $\mathbb{M}^{(\alpha)}$ families:}
	\begin{itemize}
     \em{ \item [(a)] Observe that (\ref{1eqn:form_of_general_E_alpha_family})
		can be re-written, for $\bf{x}\in\mathbb{S}$, as
		\begin{eqnarray}
		\label{1eqn:E_alpha_as_M_alpha}
		p_\theta({\bf{x}}) = Z(\theta) [h({\bf{x}}) +w(\theta)^\top f({\bf{x}})]^{\frac{1}{(2-\alpha)-1}}.
		\end{eqnarray}
		Let $\alpha'=2-\alpha$. Thus an $\mathscr{E}^{(\alpha)}$-family can be expressed as an $\mathbb{M}^{(\alpha')}$-family characterized by the same entities, and vice-versa. This is also discussed in \cite{Tsallis09B} for the specific case $h\equiv 1$.
		 \item [(b)] $\mathbb{M}^{(\alpha)}$ and $\mathscr{E}^{(\alpha)}$ families are related through an escort transformation. When $h\equiv 1$, such escort transformations are studied in the context of non-extensive statistical physics \cite{TsallisMP98J, Naudts04J}. Karthik and Sundaresan \cite[Theorem 2]{KarthikS17J} derived this connection for discrete, canonical families. We now extend this to the more general $\mathbb{M}^{(\alpha)}$ and $\mathscr{E}^{(\alpha)}$ families as in (\ref{1eqn:form_of_general_M_alpha_family}) and (\ref{1eqn:form_of_general_E_alpha_family}).} 
\end{itemize}
\end{remark}
\vspace{0.2cm}

\begin{lemma}
			\label{1lem:equivalence_of_E_alpha_and_I_alpha}
			Let $\alpha\neq 0$. The map $p\mapsto p^{(\alpha)}$ establishes a one-to-one correspondence between an $\mathscr{E}^{(\alpha)}$-family characterized by $h,f,w,\Theta$ and $\mathbb{S}$, and the $\mathbb{M}^{(1/\alpha)}$-family characterized by the same entities, where
			$p^{(\alpha)}({\bf{x}})=
			{p({\bf{x}})^\alpha}\big/{\int p({\bf{y}})^\alpha d{\bf{y}}}$ is the $\alpha$-scaled
				measure (or the escort measure) associated with $p$.
		\end{lemma}
		\vspace{0.2cm}
		
		\begin{proof}
			For any $p_\theta\in \mathscr{E}^{(\alpha)}$ characterized by
			$h,f$ and $w$, from (\ref{1eqn:form_of_general_E_alpha_family}) we have, for ${\bf{x}}\in \mathbb{S}$,
			\begin{align*}
			\label{1eqn:E_alpha_to_M_alpha}
			p^{(\alpha)}_\theta({\bf{x}}) 
			= \Big({Z(\theta)^{\alpha}}\big/~{\|p_\theta\|^\alpha}\Big)~ 
			\Big[h({\bf{x}}) + w(\theta)^\top f({\bf{x}})\Big]^{\frac{\alpha}{1 - \alpha}}
			=Z'(\theta)\Big[ h({\bf{x}}) + w(\theta)^\top f({\bf{x}})\Big]^{\frac{1}{\frac{1}{\alpha}-1}},
			\end{align*}
			where $\|p_\theta\|^\alpha = \int p({\bf{x}})^\alpha d{\bf{x}}$
			and $Z'(\theta) = Z(\theta)^\alpha\big/ \|p_\theta\|^\alpha$.
			Hence $p_\theta^{(\alpha)}\in \mathbb{M}^{(1/\alpha)}$ characterized
			by the same functions $h,f$ and $w$. So, the mapping is well-defined. The map is one-one, since it is easy to see that if $p_{\theta}^{(\alpha)} = p_{\eta}^{(\alpha)}$ for some $\theta, \eta\in\Theta$ then $p_{\theta} = p_{\eta}$.
			To verify it is onto, let $p\in \mathbb{M}^{(1/\alpha)}$
			be arbitrary. Then, for ${\bf{x}}\in\mathbb{S}$,
			\begin{align*}
			p({\bf{x}}) 
			& = Z(\theta)\Big[h({\bf{x}}) + w(\theta)^\top f({\bf{x}})\Big]^{\frac{1}{\frac{1}{\alpha}-1}},
			\end{align*}
			which implies
			\begin{equation*}
			p({\bf{x}})^{1/\alpha} =Z(\theta)^{1/\alpha} \big[h({\bf{x}}) + 
			w(\theta)^\top f({\bf{x}})\big]^{\frac{1}{1-\alpha}}
			\end{equation*}
			and hence
			\begin{align*}
			p^{(1/\alpha)}({\bf{x}})
			& = \frac{Z(\theta)^{1/\alpha}}{\int p({\bf{y}})^{1/\alpha}
				d{\bf{y}}} \big[h({\bf{x}}) + 
			w(\theta)^\top f({\bf{x}})\big]^{\frac{1}{1-\alpha}}.
			\end{align*}
			Thus $p^{(1/\alpha)}\in \mathscr{E}^{(\alpha)}$ and so $p^{(1/\alpha)} = p_\theta$ for some $\theta\in\Theta$. It is now easy to show that $p_\theta^{(\alpha)} = p$. Thus for any $p\in \mathbb{M}^{(\tiny{1/\alpha})}$ characterized by $h,f$ and $w$, there exists $p_\theta\in \mathscr{E}^{(\alpha)}$ characterized by the same functions
			such that $p^{(\alpha)}_\theta = p$. Hence the mapping is onto.
		\end{proof}
We now find	the $\alpha$-scaled Student distributions which form an
$\mathscr{E}^{(1/\alpha)}$-family, in view of Lemma \ref{1lem:equivalence_of_E_alpha_and_I_alpha}.
\vspace{0.2cm}

\begin{example}[{\bf{Cauchy distributions}}]
	\label{1expl:cauchy_distribution}
	
\emph{	Let us consider the $d$-dimensional Student distributions $p_\theta$
	as in 
	(\ref{1eqn:student_distribution}). The
	$\alpha$-scaled measure of $p_\theta$ is given by
	\begin{equation}
	\label{1eqn:escort_of_student_t}
	q_\eta({\bf{x}}) := p_\theta^{(\alpha)}({\bf{x}}) =
	\widetilde{N}_{\eta,\alpha} [ 1+ b_\alpha({\bf{x}} - \boldsymbol{\mu})^\top \boldsymbol{\Sigma}^{-1} ({\bf{x}} - \boldsymbol{\mu})]_+^{\frac{\alpha}{\alpha-1}},
	\end{equation}
	where
	\begin{equation*}
	\widetilde{N}_{\eta,\alpha} := 1\Big/\int [ 1+ b_\alpha({\bf{x}} - \boldsymbol{\mu})^\top \boldsymbol{\Sigma}^{-1} ({\bf{x}} - \boldsymbol{\mu})]^{\frac{\alpha}{\alpha-1}} d {\bf{x}}
	\end{equation*}
	is the normalizing factor and $\eta = [\mu_i,\sigma_{ij}]^\top_{i,j\in\{1,\ldots,d\}, i\leq j}$. Observe that $q_\eta$ is a valid density function for $\alpha\in(-\infty, \min\{0,(d-2)/d\})\cup (d/(d+2),1)\cup(1,\infty)$ and it has full support for $\alpha\in(d/(d+2),1)$. Notice that $q_\eta$ in (\ref{1eqn:escort_of_student_t}) can be re-written, for
	$\bf{x}\in\mathbb{S}$, as
	\begin{align*}
	q_\eta ({\bf{x}})
	& =\widetilde{N}_{\eta,\alpha} \big[ 1 + b_\alpha \lbrace \rm{vec}^\top (\boldsymbol{\Sigma}^{-1}) \rm{vec}({\bf{x}}{\bf{x}}^\top) - 2 (\boldsymbol{\Sigma} ^{-1}\boldsymbol{\mu})^\top {\bf{x}}
	+\boldsymbol{\mu}^\top \boldsymbol{\Sigma} ^{-1}\boldsymbol{\mu}\rbrace\big]^{\frac{\alpha}{\alpha-1}}\\
	& =  
	S(\eta)^{\frac{\alpha}{\alpha -1}} \widetilde{N}_{\eta,\alpha}
	\big[ 1 + b_\alpha S(\eta)^{-1}\lbrace 
	\rm{vec} ^\top(\boldsymbol{\Sigma}^{-1}) \rm{vec}({\bf{x}}{\bf{x}}^\top) - 2 (\boldsymbol{\Sigma}^{-1}\boldsymbol{\mu})^\top {\bf{x}}
	\rbrace\big]^{\frac{1}{1-({1}/{\alpha})}},
	\end{align*}
	where $S(\eta):= 1 + b_\alpha \boldsymbol{\mu}^\top \boldsymbol{\Sigma}^{-1}\boldsymbol{\mu}$.
	Using the notations
	$\beta = 1/\alpha$, $c_\beta = b_{1/\beta}$ and
	$M_{\eta,\beta} = \widetilde{N}_{\eta, {1/\beta}}$,
	for ${\bf{x}}\in \mathbb{S}$, we have
	\begin{equation}
	\label{1eqn:cauchy_distribution_as_E_alpha}
	q_\eta ({\bf{x}}) =
	S(\eta)^{\frac{1}{1 - \beta}} M_{\eta,\beta}
	\big[ 1 + c_\beta S(\eta)^{-1}\lbrace 
	\rm{vec}^\top (\boldsymbol{\Sigma}^{-1}) \rm{vec}({\bf{x}}{\bf{x}}^\top)  - 2 (\boldsymbol{\Sigma} ^{-1}\boldsymbol{\mu})^\top {\bf{x}}
	\rbrace\big]^{\frac{1}{1 - \beta}},
	\end{equation}
	where $\beta\in (d/(d-2),0)\cup(0,1)\cup (1,{(d+2)}\big/{d})$ for $d\leq 2$ and $\beta\in (-\infty,0)\cup(0,1)\cup (1,{(d+2)}\big/{d})$ for $d\ge 3$.
	Comparing 
	(\ref{1eqn:cauchy_distribution_as_E_alpha}) with (\ref{1eqn:form_of_general_E_alpha_family}), we see that $q_\eta$'s form a $d(d+3)/2$-parameter $\mathscr{E}^{(\beta)}$-family with
	\begin{align*}
	\begin{array}{ll}
	\eta = [\mu_i,\sigma_{ij}]^\top_{i,j\in\{1,\ldots,d\}, i\le j},\quad
	Z(\eta) = S(\eta)^{\frac{1}{1-\beta}}M_{\eta,\beta},\quad h({\bf{x}})\equiv 1,\\
	w(\eta)=\big[w^{(1)}(\eta),w^{(2)}(\eta)\big]^\top,\quad
	f({\bf{x}}) = \big[f^{(1)}({\bf{x}}),f^{(2)}({\bf{x}})\big]^\top,
	\end{array}
	\end{align*}
	where
	\begin{align*}
	w^{(1)}(\eta) = - 2 c_\beta S(\eta)^{-1}\cdot \boldsymbol{\Sigma} ^{-1}\boldsymbol{\mu},\quad 
	f^{(1)}({\bf{x}}) = {\bf{x}},\quad
	w^{(2)}(\eta) =  c_\beta S(\eta)^{-1}\cdot\rm{vec} (\boldsymbol{\Sigma} ^{-1}),\quad
	f^{(2)}({\bf{x}}) =	\rm{vec}({\bf{x}}{\bf{x}}^\top).
	\end{align*}
	Some special cases of (\ref{1eqn:cauchy_distribution_as_E_alpha}) include the following:	
	\begin{itemize}
		\item[(a)] The usual $d$-dimensional Cauchy distributions 
		correspond to $\beta = (d + 3)/(d + 1)$.
		\item[(b)] The generalized Cauchy distributions studied in 
		\cite{Rider57J} correspond to $\beta = (1 + \omega)/\omega$ and
		$\beta\in(1,3)$.
		\item[(c)] The multivariate truncated generalized Cauchy distributions studied in \cite[Eq. (2.3)]{AteyaM13J} correspond to
		$\beta = 1 + 2/(2\kappa + d)$ where $\kappa$ equals to
		the $\alpha$ in their paper and $\beta\in(1,(d + 2)/d)$.
	\end{itemize}	
	While studying the diffusion problem under L\'evy distributions, Prato and Tsallis \cite[Eq. (10)-(11)]{PratoT99J} found (\ref{1eqn:cauchy_distribution_as_E_alpha}) as the
	maximizer of R\'enyi (or Tsallis) entropy subject
	to linear constraints on the $\alpha$-scaled
	measure of the distribution. In
	\cite{VignatP07J, OharaW10J, GhoshdastidarDB12ISIT, Matsuzoe17J}, these distributions were studied
	as $q$-Gaussian distributions. However, we shall call them simply Cauchy distributions with location parameter $\boldsymbol{\mu}$ and scale parameter $\boldsymbol{\Sigma}$.	}
\end{example}

Observe that the functions $w$ and $f$ in Cauchy distribution as in (\ref{1eqn:cauchy_distribution_as_E_alpha}) are the same as 
the ones in Student distribution (\ref{1eqn:student_distribution_as_M_alpha_1}). Thus
by a similar argument as described in Corollary
\ref{1cor:Student_regular_M_alpha}, we can show that 
Cauchy distributions form a $d(d + 3)/2$-parameter regular $\mathscr{E}
^{(\beta)}$-family for $\beta\in(1,{(d+2)}\big/{d})$. Note that for $\beta\notin(1,{(d+2)}\big/{d})$, they do not define a
regular family because, in this case, the support depends on the unknown parameters.
\vspace{0.2cm}

\begin{example}
	\label{1expl:student_t_E_alpha}
\emph{	Consider the Student distributions as in 
	(\ref{1eqn:student_distribution_as_M_alpha_1}). In view
	of Remark \ref{1rem:M_alpha_E_alpha}(a), these form an
	$\mathscr{E}^{(\alpha')}$-family, where $\alpha'=2-\alpha$ (that is,
	$\alpha'\in (-\infty,1)\cup (1,(d+2)/d)\cup ((d+2)/d,\infty)$ when $d\leq 2$, and $\alpha'\in (-\infty,1)\cup (1,(d+2)/d)\cup (2,\infty)$ otherwise), characterized by the same functions as in 
	(\ref{1eqn:charecterizing_entities_of_student_as_M}). Student distributions as $\mathscr{E}^{(\alpha')}$ is studied in the literature, for example, in \cite{OharaW10J, Matsuzoe17J}. Observe that, when $\alpha'\in (1, (d+2)/d)$, they indeed form a regular $\mathscr{E}^{(\alpha')}$ as this corresponds to $\alpha\in((d-2)/d,1)$ in their $\mathbb{M}^{(\alpha)}$ form. }   
\end{example}
\vspace{0.2cm}

\begin{remark}
	\label{1rem:analogy_between_alpha_families_exponential_family}
\emph{Consider an exponential family of probability distributions where, for ${\bf{x}}\in\mathbb{S}$, 
	\begin{align}
	p_\theta({\bf{x}}) 
	&= Z(\theta)  \exp \big[h({\bf{x}}) +w(\theta)^\top f({\bf{x}})\big].
	\label{1eqn:exponential_family_as_E_alpha}
	\end{align}
	Then each member of this family can also be re-written in any of the following equivalent forms: 
	\begin{align}
	p_\theta({\bf{x}})^{-1} 
	&= Z(\theta)^{-1} \exp \big[-h({\bf{x}}) - w(\theta)^\top f({\bf{x}})\big]\quad \text{or}
	\label{1eqn:exponential_family_as_M_alpha}\\
	p_\theta({\bf{x}})^{-1}
	&=  \exp \big[-h({\bf{x}})-Z'(\theta) -w(\theta)^\top f({\bf{x}})\big],
	\label{1eqn:exponential_family_as_B_alpha}
	\end{align}
	where $Z'(\theta) = \ln Z(\theta)$. Analogous to (\ref{1eqn:exponential_family_as_E_alpha}),
	(\ref{1eqn:exponential_family_as_M_alpha}) and
	(\ref{1eqn:exponential_family_as_B_alpha}), respectively, the
	probability distributions in
	$\mathscr{E}^{(\alpha)}$, $\mathbb{M}^{(\alpha)}$ and
	$\mathbb{B}^{(\alpha)}$-families can be expressed, for 
	${\bf{x}}\in\mathbb{S}$, as
	\begin{align*}
	p_\theta({\bf{x}}) &= Z(\theta)  e_\alpha \big[h({\bf{x}}) +w(\theta)^\top f({\bf{x}})\big],\\
	p_\theta({\bf{x}})^{-1}& = Z(\theta)^{-1}~  e_\alpha \big[
	-h({\bf{x}})-w(\theta)^\top f({\bf{x}})\big],\\
	p_\theta({\bf{x}})^{-1}& =  e_\alpha \big[-h({\bf{x}})
	-Z'(\theta)-w(\theta)^\top f({\bf{x}})\big],
	\end{align*}
	where the $\alpha$-exponential function $e_\alpha:[-\infty,\infty]\to
	(-\infty,\infty]$ is defined as
	\begin{eqnarray*}
		{e_\alpha(r)} = \left\{
		\begin{array}{ll}
			{[\max\{1+(1-\alpha)r,0\}]^{1/(1-\alpha)},} &\hbox{} 
			\alpha\neq 1,\\
			{\exp(r),} &\hbox{}\alpha=1.
		\end{array}
		\right.
	\end{eqnarray*}
	The $\alpha$-exponential function coincides with the usual exponential function as $\alpha\to 1$. Hence the families $\mathscr{E}^{(\alpha)}$, $\mathbb{M}^{(\alpha)}$ and
	$\mathbb{B}^{(\alpha)}$ coincide with the usual exponential family as $\alpha\to 1$. Thus these three power-law families can be seen as generalizations of the exponential family. These power-law families are sometimes known as deformed exponential families (see, for example, \cite{Matsuzoe17J}).}
\end{remark}
\section{Projection theorems for general power-law families}
\label{1sec:estimating_equation_for_the_general_family}
In this section, we extend the projection theorems of $B_\alpha$, $\mathscr{I}_\alpha$ and $D_\alpha$ divergences to the general power-law families by directly solving the associated estimating equations. We also find conditions under which the new projection theorems reduce to the ones as in the canonical case.
We shall begin by recalling the projection theorems known in the literature. In the following, assume that the families are canonical and regular with support $\mathbb{S}$ being finite and the parameter space $\Theta$ being the natural parameter space. Let $p_n$ denote the empirical distribution of  sample ${X}_1,\ldots,{X}_n$.

\begin{enumerate}
	
	\item[(a)] {\em Projection theorem for $B_\alpha$-divergence}:
	Consider a $\mathbb{B}^{(\alpha)}$-family characterized by $f$ and $h=q^{\alpha-1}$, where $q$ is a probability distribution with support $\mathbb{S}$. The reverse $B_\alpha$-projection of $p_n$ on $\mathbb{B}^{(\alpha)}$ satisfies
	\begin{eqnarray}
	\label{1eqn:estimating_equation_for_regular_B_alpha}
	\mathbb{E}_\theta[f(X)] = \bar{f},
	\end{eqnarray}
	where $\bar{f} := [\bar{f_1},\ldots, \bar{f_k}]^\top $, $\bar{f_i} := \frac{1}{n}\sum_{j=1}^{n}f_i(X_j)$ for $i\in\{1,\ldots,k\}$ and $\mathbb{E}_\theta[\cdots]$ denotes expected value with respect to $p_\theta$. (See Theorem \ref{1thm:orthogonality2} and Remark \ref{1rem:projection_B_alpha_alpha_greater_1}). Ohara and Wada \cite[Prop. 3]{OharaW10J} established (\ref{1eqn:estimating_equation_for_regular_B_alpha}) for the continuous case with $h$ being identically a constant. Csisz\'ar and Mat\'u\v{s} \cite{CsiszarM12J} studied this for the general Bregman divergences.
	
	\vspace*{0.2cm}
	
	\item[(b)] {\em Projection theorem of $\mathscr{I}_\alpha$-divergence}:
	Consider an $\mathbb{M}^{(\alpha)}$-family characterized by $f$ and $h=q^{\alpha-1}$, where $q$ is a probability distribution with support $\mathbb{S}$. The reverse $\mathscr{I}_\alpha$-projection of $p_n$ on $\mathbb{M}^{(\alpha)}$ satisfies
	\begin{eqnarray}
	\label{1eqn:estimating_equation_for_regular_M_alpha}
	\dfrac{\mathbb{E}_\theta\big[f(X)\big]}{\mathbb{E}_\theta[h(X)]} = \dfrac{\bar{f}}{\overline{h}},	\end{eqnarray}
	where $\overline{h} := \frac{1}{n}\sum_{j=1}^n h(X_j)$. This is due to \cite[Theorem 18 and Theorem 21]{KumarS15J2}.
	
	\vspace*{0.2cm}
	
	\item[(c)] {\em Projection theorem of $D_\alpha$-divergence}:
	Consider an $\mathscr{E}^{(\alpha)}$-family characterized by $f$ and  $h=q^{1-\alpha}$, where $q$ is a probability distribution with support $\mathbb{S}$. The reverse $D_\alpha$-projection of $p_n$ on $\mathscr{E}^{(\alpha)}$ satisfies
	\begin{eqnarray}
	\label{1eqn:estimating_equation_for_regular_E_alpha}
	\dfrac{\mathbb{E}_{\theta^{(\alpha)}}\big[f(X)\big]}{\mathbb{E}_{\theta^{(\alpha)}}[h(X)]} = 
	\dfrac{\overline{f}^{(\alpha)}}{\overline{h}^{(\alpha)}},
	\end{eqnarray}
	where $\mathbb{E}_{\theta^{(\alpha)}}[\cdots]$ denotes expectation with respect to $p_{\theta}^{(\alpha)}$, $\overline{h}^{(\alpha)}$ and $\overline{f_i}^{(\alpha)}$ are respectively averages of $h$ and $f_i$ with respect to $p_n^{(\alpha)}$. This is due to \cite[Theorem 6]{KumarS16J}.
\end{enumerate}
Before we turn to the main results of this section we prove the following lemma and corollary that establish a connection
between the generalized Hellinger and the Jones et al. estimating equations.
\vspace{0.2cm}

\begin{lemma}
	\label{1lem:connection_of_estimating_equation_and_projection_equation}
	The estimating equations (\ref{1eqn:score_equation_for_D_alpha_discrete_case}) and (\ref{1eqn:score_equation_I_alpha}) are the same up-to the transformation
	$p\mapsto p^{(\alpha)}$ when $\mathbb{S}$ is discrete. In the
	continuous case the same is true between (\ref{1eqn:score_equation_for_D_alpha}) and (\ref{1eqn:score_equation_I_alpha}) provided the empirical distribution $p_n$ is replaced by a continuous estimate $\widetilde{p}_n$ and $\int \nabla[p_\theta ({\bf{x}})] d{\bf{x}} = 0$.
\end{lemma}
\vspace{0.2cm}

\begin{proof}
	We present the proof for the discrete case. The proof for the continuous case follows by a similar argument if we replace $p_n$ by $\widetilde{p}_n$ throughout in the proof.
	
	The generalized Hellinger estimating equation (\ref{1eqn:score_equation_for_D_alpha_discrete_case}) can be re-written as
	
	\begin{eqnarray*}
	\label{1eqn:modified_estimating_equation_for_D_alpha}
	\dfrac{\sum\limits_{{\bf{x}}\in\mathbb{S}}p_n({\bf{x}})^\alpha p_\theta({\bf{x}})^{1-\alpha}s({\bf{x}};\theta)}{\sum\limits_{{\bf{x}}\in\mathbb{S}}p_n({\bf{x}})^\alpha p_\theta({\bf{x}})^{1-\alpha}} =
	{\sum\limits_{{\bf{{\bf{x}}}}\in\mathbb{S}}p_\theta({\bf{x}})s({\bf{x}};\theta)},
	\end{eqnarray*}
	since $\sum\limits_{{\bf{x}}\in\mathbb{S}}p_\theta({\bf{x}})s({\bf{x}};\theta) = 0$. This can further be re-written as
	\begin{align}
	\label{1eqn:estimating _equation_of_D_alpha_after_transformation}
	\dfrac{\sum\limits_{{\bf{x}}\in\mathbb{S}}p_n^{(\alpha)}({\bf{x}}) \big[ p^{(\alpha)}_\theta({\bf{x}})\big]
		^{\tfrac{1}{\alpha}-1}s({\bf{x}};\theta)}{\sum\limits_{{\bf{x}}\in\mathbb{S}} p_n^{(\alpha)} ({\bf{x}})
		\big[p_\theta^{(\alpha)}({\bf{x}})\big] ^{\tfrac{1}{\alpha}-1}}
	=
	\dfrac{\sum\limits_{{\bf{x}}\in\mathbb{S}}\big[ p^{(\alpha)}_\theta({\bf{x}})\big]^{\tfrac{1}{\alpha}} s({\bf{x}};\theta)}
	{\sum\limits_{{\bf{x}}\in\mathbb{S}}\big[p^{(\alpha)}_\theta({\bf{x}})\big]^{\tfrac{1}{\alpha}}}. 
	\end{align}
	Observe that
	\begin{align*}
	s^{(\alpha)}({\bf{x}};\theta) := \nabla\ln p^{(\alpha)}_\theta({\bf{x}}) = 
	\nabla\ln \frac{p_\theta({\bf{x}})^\alpha}{\|p_\theta\|^\alpha}
	= \nabla \big[ \ln p_\theta ({\bf{x}})^\alpha -\ln \|p_\theta\|^\alpha\big]
	= \alpha \big[s({\bf{x}};\theta) - \nabla\ln \|p_\theta\| \big].
	\end{align*}
	Hence
	\begin{eqnarray}
	\label{1eqn:s_in_terms_of_s_alpha}
	s({\bf{x}};\theta) = \tfrac{1}{\alpha} s^{(\alpha)}({\bf{x}};\theta) + A(\theta),
	\end{eqnarray}
	where $A(\theta) = \nabla\ln \|p_\theta\|$. Plugging (\ref{1eqn:s_in_terms_of_s_alpha}) in (\ref{1eqn:estimating _equation_of_D_alpha_after_transformation}),
	\begin{align*}
	\label{1eqn:estimating _equation_of_D_alpha_after_transformation_final}
	\dfrac{\sum\limits_{{\bf{x}}\in\mathbb{S}}p_n^{(\alpha)}({\bf{x}}) \big[ p^{(\alpha)}_\theta({\bf{x}})\big]
		^{\tfrac{1}{\alpha}-1}s^{(\alpha)}({\bf{x}};\theta)}{\sum\limits_{{\bf{x}}\in\mathbb{S}}p_n^{(\alpha)} ({\bf{x}})
		\big[p_\theta^{(\alpha)}({\bf{x}})\big] ^{\tfrac{1}{\alpha}-1}}  
	=
	\dfrac{\sum\limits_{{\bf{x}}\in\mathbb{S}}\big[ p^{(\alpha)}_\theta({\bf{x}})\big]^{\tfrac{1}{\alpha}} s^{(\alpha)}({\bf{x}};\theta)}
	{\sum\limits_{{\bf{x}}\in\mathbb{S}}\big[ p^{(\alpha)}_\theta({\bf{x}})\big]^{\tfrac{1}{\alpha}}}. 
	\end{align*}
	This is same as the Jones et al. estimating equation (\ref{1eqn:score_equation_I_alpha}) with $p_n$, $p_\theta$ and $\alpha$, respectively, replaced by $p_n^{(\alpha)}$, $p_\theta^{(\alpha)}$ and $1/\alpha$.
\end{proof}

This, together with Lemma \ref{1lem:equivalence_of_E_alpha_and_I_alpha},
establishes the following equivalence between the Jones et al. estimation and generalized Hellinger estimation.
\vspace{0.2cm}

\begin{corollary}
	\label{1cor:equivalence_of_projection_problem_E_alpha_M_alpha}
	Suppose that $\mathscr{E}^{(\alpha)}$ is an $\alpha$-exponential family characterized by $h,w,f,\Theta$ where all the distributions have a common support $\mathbb{S}$. Then, under the assumptions of Lemma \ref{1lem:connection_of_estimating_equation_and_projection_equation}, solving the generalized Hellinger estimation problem under $\mathscr{E}^{(\alpha)}$-family is equivalent to solving the Jones et al. estimation problem under the $\mathbb{M}^{(1/\alpha)}$-family characterized by the same entities.
\end{corollary}
\vspace{0.2cm}

The following result extends the already known projection theorems of the divergences $B_\alpha$, $\mathscr{I}_\alpha$ and $D_\alpha$ to the general power-law families as defined in Section \ref{1sec:the_general_form_of_the_alpha_families}.
\vspace{0.2cm}

\begin{theorem}
	\label{1thm:estimating_equations_for_general_families} 
	Let ${\bf{X}}_1,\ldots,{\bf{X}}_n$ be $n$ i.i.d. samples. Let $\Pi$ be one of the families $\mathbb{B}^{(\alpha)}$, $\mathbb{M}^{(\alpha)}$, or $\mathscr{E}^{(\alpha)}$ and assume that support of $\Pi$ does not depend on the parameter space $\Theta$. In (c), assume also that $\int {\partial_r}[p_\theta ({\bf{x}})] d{\bf{x}} = 0$ for $r\in\{1,\ldots,k\}$. Then the following hold.
	
	\begin{itemize}
		\item[(a)] Basu et al. estimator under 
		$\mathbb{B}^{(\alpha)}$ must satisfy
		\begin{equation}
		\label{1eqn:estimating_equation_for_general_exponential_family}
		\partial _r [w (\theta)]^\top \mathbb{E}_\theta [f(\textbf{X})] = 
		\partial _r [w (\theta)]^\top \bar{f},\quad\text{for}~r\in\{1,\ldots,k\}.
		\end{equation}
		\item[(b)] Jones et al. estimator under $\mathbb{M}^{(\alpha)}$ must satisfy
		\begin{equation}
		\label{1eqn:estimating_equation_for_general_M_alpha_family}
		\frac{ \partial_r [w(\theta)]^\top \mathbb{E}_\theta [f(\textbf{X})]}
		{\mathbb{E}_\theta [h(\textbf{X}) +  
			w(\theta)^\top f(\textbf{X}) ]} = 
		\frac{ \partial_r [w(\theta)]^\top \bar{f}}
		{\bar{h} +   w(\theta)^\top \bar{f}},
		\quad\text{for}~r\in\{1,\ldots,k\}.
		\end{equation}
		\item[(c)] Generalized Hellinger estimator under $\mathscr{E}^{(\alpha)}$
		must satisfy
		\begin{eqnarray}
		\label{1eqn:estimating_equation_for_general_E_alpha_family}
		\frac{\partial_r [w(\theta)]^\top \mathbb{E}_{\theta^{(\alpha)}} [f(\textbf{X})]}{\mathbb{E}_{\theta^{(\alpha)}} [h(\textbf{X}) +  w(\theta)^\top f(\textbf{X})]} = 
		\frac{\partial_r [w(\theta)]^\top \overline{f}^{(\alpha)}}
		{\overline{h}^{(\alpha)} + w(\theta)^\top \overline{f}^{(\alpha)}},\quad\text{for}~r\in\{1,\ldots,k\}.
		\end{eqnarray}
	\end{itemize}
	Here $\partial _r [w (\theta)]: = \big[ \tfrac{\partial}{\partial \theta_r} [w_1(\theta)], \ldots, \tfrac{\partial}{\partial \theta_r} [w_s(\theta
	)]\big]^\top$ for $r\in\{1,\ldots,k\}$. In (c), $\overline{h}^{(\alpha)}: = \mathbb{E}_{\widetilde{p}_n^{(\alpha)}}
	[h(\textbf{\textit{X}})]$ and
	$\overline{f}^{(\alpha)} :=\mathbb{E}_{\widetilde{p}_n^{(\alpha)}} [f(\textbf{\textit{X}})]$, where $\widetilde{p}_n$ is  the
	empirical distribution $p_n$ in the discrete case; a suitable continuous estimate of ${p}_n$ in the continuous case.
\end{theorem}
\vspace{0.2cm}

\begin{proof}
(a) If $p_\theta\in \mathbb{B}^{(\alpha)}$ then from
		Definition \ref{1defn:general_B_alpha_family}, for
		${\bf{x}}\in\mathbb{S}$, 
		\begin{equation*}
		p_\theta({\bf{x}}) ^{\alpha - 1}
		= h({\bf{x}}) + F(\theta) + w(\theta)^\top f({\bf{x}}).
		\end{equation*}
		Taking derivative with respect to $\theta_r$ for
		$r\in\{1,\ldots,k\}$,
		\begin{align}
		\label{1eqn:derivative_of_general_B_alpha}
		(\alpha -1)p_\theta({\bf{x}})^{\alpha -2}\partial_r [p_\theta ({\bf{x}})] &= \partial_r[F(\theta)] + \partial_r[w(\theta)]^\top
		f({\bf{x}}).
		\end{align}
		The Basu et al. estimating equation
		(\ref{1eqn:score_equation_B_alpha}) can be re-written as
		\begin{equation}
		\label{1eqn:estimating_equation_B_alpha_without_nabla}
		\frac{1}{n} \sum\limits_{j=1}^n p_\theta(\textbf{X}_j)^{\alpha -2}
		\partial_r [p_\theta (\textbf{X}_j)] =
		\int p_\theta({\bf{x}})^{\alpha -1} \partial_r [p_\theta ({\bf{x}})] d{\bf{x}}.
		\end{equation}
		Substituting (\ref{1eqn:derivative_of_general_B_alpha}) in
		(\ref{1eqn:estimating_equation_B_alpha_without_nabla}), we get
		(\ref{1eqn:estimating_equation_for_general_exponential_family}).\\
		
		(b) If $p_\theta\in \mathbb{M}^{(\alpha)}$, using
		Definition \ref{1defn:general_M_alpha_family}, for
		${\bf{x}}\in\mathbb{S}$,
		\begin{equation*}
		p_\theta({\bf{x}})^{\alpha -1} = Z(\theta)^{\alpha -1}[h({\bf{x}})
		+w(\theta)^\top f({\bf{x}})].
		\end{equation*}	
		Taking derivative with respect to $\theta_r$ for
		$r\in\{1,\ldots,k\}$, we get
		\begin{align*}
		(\alpha -1)p_\theta({\bf{x}})^{\alpha -2} 	
		\partial_r [p_\theta ({\bf{x}})]&= \partial_r[Z(\theta)^{\alpha - 1}][h({\bf{x}})
		+ w(\theta)^\top f({\bf{x}})] + Z(\theta)^{\alpha - 1}\partial_r[w(\theta)]^\top f({\bf{x}}).
		\end{align*}
		Substituting this in the Jones et al. estimating equation (\ref{1eqn:score_equation_I_alpha}),
		\begin{equation*}
		\frac{ \partial_r [w(\theta)]^\top \mathbb{E}_\theta [f(\textbf{X})]}
		{\mathbb{E}_\theta [h(\textbf{X}) +  
			w(\theta)^\top f(\textbf{X}) ]} = 
		\frac{ \partial_r [w(\theta)]^\top \bar{f}}
		{\bar{h} +   w(\theta)^\top \bar{f}},
		\quad r\in\{1,\ldots,k\}.
		\end{equation*}
		\noindent
		
(c) This follows from (b) and Corollary \ref{1cor:equivalence_of_projection_problem_E_alpha_M_alpha}.
\end{proof}
\vspace{0.2cm}

\begin{remark}
	\begin{itemize}
\item[(a)] Jones et al. and generalized Hellinger estimation under $\mathbb{B}^{(\alpha)}$:\emph{ Recall that a $\mathbb{B}^{(\alpha)}$-family can
		be expressed as an $\mathbb{M}^{(\alpha)}$-family
		as in (\ref{1eqn:B_alpha_to_M_alpha_1}) or
		(\ref{1eqn:B_alpha_to_M_alpha_2}). This implies that
		the Jones et al.
		estimator under $\mathbb{B}^{(\alpha)}$-family satisfies
		(\ref{1eqn:estimating_equation_for_general_M_alpha_family}) with
		$w$, $f$ replaced by $\widetilde{w},\widetilde{f}$ as defined in
		(\ref{1eqn:B_alpha_to_M_alpha_1}) or
		(\ref{1eqn:B_alpha_to_M_alpha_2}).
		Further, in view of Remark \ref{1rem:M_alpha_E_alpha}(a),
		(\ref{1eqn:B_alpha_to_M_alpha_1}) or
		(\ref{1eqn:B_alpha_to_M_alpha_2})
		is also an
		$\mathscr{E}^{(\alpha')}$-family where $\alpha'=2-\alpha$. Thus
		the $\alpha'$-generalized Hellinger estimator
		under $\mathbb{B}^{(\alpha)}$-family must satisfy
		(\ref{1eqn:estimating_equation_for_general_E_alpha_family})
		with $w$, $f$ and $\alpha$ replaced, respectively, by $\widetilde{w},\widetilde{f}$ and $\alpha'$.}
 \item[(b)] Basu et al. and generalized Hellinger estimation under $\mathbb{M}^{(\alpha)}$:\emph{ An $\mathbb{M}^{(\alpha)}$-family
		can be expressed as a $\mathbb{B}^{(\alpha)}$-family as in
		(\ref{1eqn:M_alpha_to_B_alpha}). Thus the Basu et al. estimator
		under an $\mathbb{M}^{(\alpha)}$-family must satisfy
		(\ref{1eqn:estimating_equation_for_general_exponential_family})
		with $w$ and $f$ replaced by $\widetilde{w},\widetilde{f}$ as defined in
		(\ref{1eqn:M_alpha_to_B_alpha}). Further, in view of Remark \ref{1rem:M_alpha_E_alpha}(a), the $\alpha'$-generalized Hellinger estimator
		under $\mathbb{M}^{(\alpha)}$-family must satisfy
		(\ref{1eqn:estimating_equation_for_general_E_alpha_family})
		with $\alpha$ replaced by $\alpha'$.}
		\item[(c)] Basu et al. and Jones et al. estimation under $\mathscr{E}^{(\alpha)}$:\emph{ An $\mathscr{E}^{(\alpha)}$-family
		can be expressed as an $\mathbb{M}^{(\alpha')}$-family as in
		(\ref{1eqn:E_alpha_as_M_alpha}), and hence can be expressed
		as a $\mathbb{B}^{(\alpha')}$-family as in
		(\ref{1eqn:M_alpha_to_B_alpha}) with $\alpha$ replaced by $\alpha'$.
		Thus the $\alpha'$-Jones et al. estimator under
		$\mathscr{E}^{(\alpha)}$-family satisfies
		(\ref{1eqn:estimating_equation_for_general_M_alpha_family}).
		Similarly the $\alpha'$-Basu et al. estimator under
		$\mathscr{E}^{(\alpha)}$-family satisfies
		(\ref{1eqn:estimating_equation_for_general_exponential_family})
		with $w$ and $f$ replaced by $\widetilde{w},\widetilde{f}$ as in
		(\ref{1eqn:M_alpha_to_B_alpha}).
		}
	\end{itemize}
\end{remark}
\vspace{0.2cm}

We now show that, when the families are regular, the projection equations in Theorem \ref{1thm:estimating_equations_for_general_families} reduce to the one as in the canonical case.
\vspace{0.2cm}

\begin{corollary}
	\label{1cor:similarity_of_estimating_equations}
	The estimating equations (\ref{1eqn:estimating_equation_for_general_exponential_family}), (\ref{1eqn:estimating_equation_for_general_M_alpha_family}) and (\ref{1eqn:estimating_equation_for_general_E_alpha_family}), respectively, reduce to (\ref{1eqn:estimating_equation_for_regular_B_alpha}), (\ref{1eqn:estimating_equation_for_regular_M_alpha}) and (\ref{1eqn:estimating_equation_for_regular_E_alpha}) if the underlying families are regular. 
\end{corollary}
\vspace{0.2cm}

\begin{proof}
	Let us first observe that for a regular family the
	matrix $[\partial_i(w_j(\theta))]_{k\times k}$ is non-singular
	for $\theta\in \Theta$.
	To see this, let
	\begin{equation*}
	c_1 \partial_r[w_1(\theta)] +\cdots + c_k\partial_r[w_k(\theta)] =0
	\end{equation*}
	for some scalars $c_1,\ldots,c_k$ and for each $r\in\{1,\dots,k\}$. Then
	\begin{equation*}
	c_1w_1(\theta) +\cdots +c_k w_k(\theta) = c,
	\end{equation*}
	for some constant $c$. Now linear independence of
	$1$, $w_1,\dots,w_k$ implies that $c=c_1=\cdots =c_k
	=0$.
	Consider a regular $\mathbb{B}^{(\alpha)}$-family. Then from (\ref{1eqn:estimating_equation_for_general_exponential_family}), we have
	\begin{equation}
	\label{1eqn:reg_b_alpha_est}
	\partial _r [w (\theta)]^\top \big(\mathbb{E}_\theta [f(\textbf{X})] - \bar{f} \big) = 0~\text{for}~r\in\{1,\dots,k\}.
	\end{equation}
	Since $[\partial_i(w_j(\theta))]_{k\times k}$ is non-singular, (\ref{1eqn:reg_b_alpha_est}) reduces to $\mathbb{E}_\theta [f(\textbf{X})] = \bar{f}$.
	Again (\ref{1eqn:estimating_equation_for_general_M_alpha_family})
	can be re-written as
	\begin{equation*}
	\partial_r [w(\theta)]^\top \Bigg[\frac{ \mathbb{E}_\theta [f(\textbf{X})]}
	{\mathbb{E}_\theta [h(\textbf{X}) + 
		w(\theta)^\top f(\textbf{X}) ]} -
	\frac{\bar{f}}
	{\overline{h} +   w(\theta)^\top \bar{f}}\Bigg]=0.
	\end{equation*}
	For a regular $\mathbb{M}^{(\alpha)}$-family, this reduces to
	\begin{equation}
	\label{1eqn:estimating_equation_regular_M_alpha_1}
	\frac{ \mathbb{E}_\theta [f(\textbf{X})]}
	{\mathbb{E}_\theta [h(\textbf{X}) + 
		w(\theta)^\top f(\textbf{X}) ]} -
	\frac{\bar{f}}
	{\overline{h} + w(\theta)^\top \bar{f}}=0.
	\end{equation}
	This implies
	\begin{eqnarray}
	\overline{h} +  w(\theta)^\top \bar{f}
	=
	\overline{h} + 
	\frac{\overline{h} +  w(\theta)^\top \bar{f}}
	{\mathbb{E}_\theta [h(\textbf{X}) +  
		w(\theta)^\top f(\textbf{X}) ]} 
	w(\theta)^\top  \mathbb{E}_\theta [f(\textbf{X})].\nonumber
	\end{eqnarray}
	That is,
	\begin{eqnarray*}
		\big\lbrace\overline{h} +   w(\theta)^\top \bar{f}\big\rbrace\big\lbrace\mathbb{E}_\theta [h(\textbf{X}) + 
		w(\theta)^\top f(\textbf{X}) ]\big\rbrace
		=
		\overline{h}~\mathbb{E}_\theta [h(\textbf{X}) +  
		w(\theta)^\top f(\textbf{X}) ] + 
		[\overline{h} + w(\theta)^\top \bar{f}]
		w(\theta)^\top  \mathbb{E}_\theta [f(\textbf{X})].\nonumber
	\end{eqnarray*}
	Hence
	\begin{equation*}
	\big\lbrace\overline{h} +  w(\theta)^\top \bar{f}\big\rbrace\mathbb{E}_\theta [h(\textbf{X})]
	=
	\overline{h}~\mathbb{E}_\theta [h(\textbf{X}) + 
	w(\theta)^\top f(\textbf{X}) ].
	\end{equation*}
	Substituting this back
	in (\ref{1eqn:estimating_equation_regular_M_alpha_1}),
	\begin{equation*}
	\frac{\mathbb{E}_\theta [f(\textbf{X})]}
	{\mathbb{E}_\theta [h(\textbf{X}) ]} = 
	\frac{ \bar{f}}
	{\overline{h}}.
	\end{equation*}
	In a similar fashion, the result for regular $\mathscr{E}^{(\alpha)}$-family can be shown.
\end{proof}

Theorem \ref{1thm:estimating_equations_for_general_families} fails if the support of the underlying family depends on the parameters. We show this by an example in Section \ref{1sec:Jones_et_al_on_student_t_alpha_bigger_1}.

Basu et al. estimating equation (\ref{1eqn:score_equation_B_alpha}) differs from the Jones et al.
estimating equation (\ref{1eqn:score_equation_I_alpha}) in which the weights
are normalized. Much research has been done to compare these two
methods (for example, see \cite{JonesHHB01J}). We saw in Section
\ref{1sec:the_general_form_of_the_alpha_families} that
a regular $\mathbb{B}^{(\alpha)}$-family can be viewed as a
regular $\mathbb{M}^{(\alpha)}$-family under some conditions. In the following, we show that the two estimations coincide on a regular
$\mathbb{B}^{(\alpha)}$-family with $h$ being a non-zero constant
(or on a regular
$\mathbb{M}^{(\alpha)}$-family with $h$ being a non-zero constant).
\vspace{0.2cm}

\begin{theorem}
	\label{1thm:same_estimation_under_basu_e_al_and_Jones_et_al}
	For a regular $\mathbb{B}^{(\alpha)}$-family with $h$ being
	identically a non-zero constant,
	Basu et al. estimating equation (\ref{1eqn:estimating_equation_for_general_exponential_family}) and Jones et al. estimating
	equation (\ref{1eqn:estimating_equation_for_general_M_alpha_family}) are the same.
\end{theorem}
\vspace{0.2cm}

\begin{proof}
	Consider the $\mathbb{B}^{(\alpha)}$-family as in 
	(\ref{1eqn:B_alpha_q_1}). If it is
	regular, from Corollary \ref{1cor:similarity_of_estimating_equations}, the Basu {et al.} estimating equation 
	is given by
	\begin{equation}
	\label{1eqn:Basu_et_al_for_B_alpah_q_equals_1}
	\mathbb{E}_\theta[f(\textbf{X})] = \bar{f}.
	\end{equation}
	We now show that the Jones {et al.} estimating equation (\ref{1eqn:estimating_equation_for_general_M_alpha_family}) for this family is also
	the same.
	Recall that (\ref{1eqn:B_alpha_q_1}) can be written as an
	$\mathbb{M}^{(\alpha)}$-family as in (\ref{1eqn:B_alpha_to_M_alpha_q_1}).
	The proof is divided into
	two parts.
	\begin{itemize}
		\item[(i)] Suppose that $F(\theta)$ is linearly independent with
		$1,w_i(\theta)$'s. Then
		(\ref{1eqn:B_alpha_to_M_alpha_q_1})
		forms a regular $\mathbb{M}^{(\alpha)}$-family by Proposition
		\ref{1pro:regular_B_alpha_and_regular_M_alpha}. Therefore using
		Corollary \ref{1cor:similarity_of_estimating_equations}, we see that
		the Jones {et al.} estimating equation (\ref{1eqn:estimating_equation_for_general_M_alpha_family}) for
		(\ref{1eqn:B_alpha_to_M_alpha_q_1}) is same as
		(\ref{1eqn:Basu_et_al_for_B_alpah_q_equals_1}),
		since $h$ is identically a constant.
		\item[(ii)] Next let us suppose that $F(\theta)$ is 
		linearly dependent with
		$1,w_i(\theta)$'s. Then there exists scalars
		$c_0, c_1,\ldots,c_k$ (not all zero) such that
		\begin{equation}
		\label{1eqn:value_of_F(theta)}
		F(\theta) = c_0 + c_1 w_1(\theta) +\cdots + c_k w_k(\theta).
		\end{equation}
		Then
		\begin{equation}
		\label{1eqn:derivative_of_F(theta)}
		\partial_r[F(\theta)] = c_1 \partial _r [w_1(\theta)] +\cdots+
		c_k \partial_r [w_k(\theta)].
		\end{equation}
		Using
		Theorem \ref{1thm:estimating_equations_for_general_families}(b),
		the Jones {et al.} estimating equation (\ref{1eqn:estimating_equation_for_general_M_alpha_family})
		for (\ref{1eqn:B_alpha_to_M_alpha_q_1}) is given, for $r\in\{1,\ldots,k\}$, by
		\begin{equation*}
		\frac{ \partial_r [w(\theta)/S(\theta)]^\top \mathbb{E}_\theta [f(\textbf{X})]}
		{\mathbb{E}_\theta [h +  [ w(\theta)/S(\theta)]^\top f(\textbf{X})]} = 
		\frac{ \partial_r [w(\theta)/S(\theta)]^\top \bar{f}}
		{h +  [ w(\theta)/S(\theta)]^\top\bar{ f}}.
		\end{equation*}
		Substituting the value of $S(\theta)$, an easy calculation yields, for $r\in\{1,\ldots,k\}$,
		\begin{align}
		\label{1eqn:Jones_et_al_estimating_equation_for_B_alpha_q_1}
		\frac{ \partial_r[F(\theta)] + \partial_r [w(\theta)]^\top \mathbb{E}_\theta [f(\textbf{X})]}
		{\mathbb{E}_\theta [h +  F(\theta) + w(\theta)^\top f(\textbf{X})]} = 
		\frac{ \partial_r[F(\theta)] + \partial_r [w(\theta)]^\top \bar{f}}
		{h +  F(\theta) + w(\theta)^\top\bar{ f}}.
		\end{align}
		Using (\ref{1eqn:value_of_F(theta)}) and (\ref{1eqn:derivative_of_F(theta)}) 
		in 
		(\ref{1eqn:Jones_et_al_estimating_equation_for_B_alpha_q_1}),
		\begin{align}
		\label{1eqn:Jones_et_al_estimating_equation_for_B_alpha_q_1_(i)}
		\partial_r[w(\theta)]^\top\Big[
		\frac{ \mathbb{E}_\theta [f(\textbf{X}) + c]}
		{h +  c_0 + w(\theta)^\top 
			\mathbb{E}_\theta [f(\textbf{X}) + c]} - 
		\frac{ \bar{f} + c}
		{h +  c_0 + w(\theta)^\top[\bar{ f} + c]}\Big] = 0,
		\end{align}
		where $c = [c_1,\ldots, c_k]^\top$. Since $1,w_1,\ldots,
		w_k$ are linearly independent, 
		$[\partial_i(w_j(\theta))]_{k\times k}$ is non-singular. Using this, (\ref{1eqn:Jones_et_al_estimating_equation_for_B_alpha_q_1_(i)})
		becomes
		\begin{equation*}
		\label{1eqn:Jones_et_al_estimating_equation_for_B_alpha_q_1_(ii)}
		\frac{ \mathbb{E}_\theta [f(\textbf{X}) + c]}
		{h +  c_0 + w(\theta)^\top 
			\mathbb{E}_\theta [f(\textbf{X}) + c]} =
		\frac{ \bar{f} + c}
		{h +  c_0 + w(\theta)^\top[\bar{ f} + c]}.
		\end{equation*}
		Proceeding as in Corollary \ref{1cor:similarity_of_estimating_equations}, 
		\begin{equation*}
		(h +  c_0 + w(\theta)^\top 
		\mathbb{E}_\theta [f(\textbf{X}) + c])(h +  c_0)
		= 
		(h +  c_0 + w(\theta)^\top[\bar{ f} + c])
		(h + c_0).	
		\end{equation*}
		That is,
		\begin{equation*}
		\frac{h +  c_0 + w(\theta)^\top 
			\mathbb{E}_\theta [f(\textbf{X}) + c]}
		{h +  c_0 + w(\theta)^\top[\bar{ f} + c]} = 1.
		\end{equation*}
Hence\hspace{3cm} $\mathbb{E}_\theta [f(\textbf{X}) + c] = \bar{f} + c,~\text{and thus}~ \mathbb{E}_\theta [f(\textbf{X})] = \bar{f}$.
	\end{itemize}
\end{proof}

\begin{corollary}
	For a regular $\mathbb{M}^{(\alpha)}$-family as in (\ref{1eqn:form_of_general_M_alpha_family}) with $h$ being identically a non-zero constant,
	the Basu et al. and the Jones et al. estimating
	equations are the same if  $Z^{1-\alpha}$ is linearly independent
	with $w_i$'s.
\end{corollary}
\section{Applications: Generalized estimation under Student and Cauchy distributions}
\label{1sec:generalized_estimation_on_student_t_and_cauchy}

In this section we find
Jones et al. estimators \cite{JonesHHB01J} for the parameters of Student distribution for $\nu\in(0,\infty)$ and generalized Hellinger estimators \cite{BasuSP11B} of Cauchy distribution for $\beta\in (1,(d+2)/2)$. For the estimation of Cauchy distributions we use the kernel density estimate for the empirical measure. We also find a robust estimator of the mean parameter of Student distribution for the case when $\nu\notin (0,\infty)$.

\subsection{Basu et al. \cite{BasuHHJ98J} and Jones et al. \cite{JonesHHB01J} estimation under Student distributions}
\label{1sec:Basu_et_al_and_Jones_et_al_estimation_on_student_t}

In Theorem \ref{1thm:student_dist_regular_B_alpha} we saw that for $\alpha\in\big((d-2)/d,1\big)$ (that is, for $\nu\in(0,\infty)$)
Student distributions form a $d(d+3)/2$-parameter regular $\mathbb{B}^{(\alpha)}$-family with $f^{(1)}({\bf{x}}) = {\bf{x}}$ and $f^{(2)}({\bf{x}})= \rm{vec}({\bf{x}}{\bf{x}}^\top)$. Hence to find the Basu {et al.} estimators of the parameters, its mean and variance should be finite. However, as we saw in Example
\ref{1expl:example_of_B_alpha}, (\ref{1eqn:student_t_as_B_alpha}) does not have finite mean and variance for $\alpha\in\big((d-2)/d,d/(d+2)\big]$.
Hence we restrict ourselves to Student distributions for $\alpha\in\big(d/(d+2),1\big)$. The mean and the covariance of a Student distribution for $\alpha\in\big(d/(d+2),1\big)$ are given by
$\boldsymbol{\mu}$ and $\boldsymbol{K}:=(k_{ij})_{d\times d} = [\nu/(\nu-2)]\cdot \boldsymbol{\Sigma}$ respectively.
Let $\textbf{X}_1,\ldots,\textbf{X}_n$ be an i.i.d. sample where each $\textbf{X}_i=[X_{1i},\ldots,X_{di}]^\top$ for $i\in\{1,\ldots,n\}$. Suppose also that
the true distribution $p$ is a Student distribution as in (\ref{1eqn:student_t_as_B_alpha}). Using Corollary \ref{1cor:similarity_of_estimating_equations}, the Basu {et al.} estimators of $\boldsymbol{\mu}$ and $\boldsymbol{K}$ are given, for $i,j\in\{1,\ldots,d\}$ and $i\leq j$, by
\begin{eqnarray}
\label{1eqn:Basu_et_al_estimators_for_student_t}
\widehat{\boldsymbol{\mu}} = \overline{\textbf{X}},\quad
\widehat{k_{ij}} = \frac{1}{n}\sum\limits_{r=1}^n X_{ir}X_{jr} - \widehat{\mu_i}\widehat{\mu_j},
\end{eqnarray}
where $\overline{\textbf{X}}=\frac{1}{n}\sum_{i=1}^n
\textbf{X}_i$.

Next consider the
Student distributions as in (\ref{1eqn:student_distribution_as_M_alpha_1})
with $\alpha\in\big(d/(d+2),1\big)$. We saw that it forms a
$d(d+3)/2$-parameter regular $\mathbb{M}^{(\alpha)}$-family
with $h\equiv 1$. Hence, from Theorem
\ref{1thm:same_estimation_under_basu_e_al_and_Jones_et_al},
the Jones {et al.} estimators for $\boldsymbol{\mu}$ and $\boldsymbol{K}$ are the
same as the Basu {et al.} estimators
as in (\ref{1eqn:Basu_et_al_estimators_for_student_t}). In \cite{EguchiKK11J} and \cite[Theorem 5]{GayenK18ISIT} this was solved directly from the estimating equation (\ref{1eqn:estimating_equation_for_general_M_alpha_family}). 
We summarize the above results in the following.

\begin{theorem}
	\label{1thm:estimators_for_student_t}
	For $\alpha\in(d/(d+2),1)$, the Basu {et al.} and the Jones {et al.} estimators of
	mean and covariance parameters of a $d$-dimensional Student distribution as in (\ref{1eqn:student_distribution}) are the same and are given by (\ref{1eqn:Basu_et_al_estimators_for_student_t}).
\end{theorem}

\begin{remark}
	\label{1rem:equivalence_of_student_t_and_normal_distribution}
\emph{	It can be shown that, as $\nu\to +\infty$, Student distributions coincide
	with a normal distribution with mean $\boldsymbol{\mu}$ and covariance matrix $\boldsymbol{K}$ \cite{JohnsonV07J}. Similarly, as $\alpha\to 1^-$, Basu {et al.} estimating equation
	or Jones {et al.} estimating equation becomes ML estimating equation (\ref{1eqn:score_equation_mle_in_terms_of_sample}).
	Thus there is a continuity of the generalized estimators of
	mean and covariance parameters as $\nu\to +\infty$. This suggests that when the samples are from a Student distribution with sufficiently large $\nu$, MLE of its parameters can be approximated by a generalized estimator of the respective parameters. (Note that the MLE of Student distributions do not have closed-from solution and numerical methods must be resorted to solve it \cite{Barnett66J1, DempsterLR77J}.) Simulations suggest that, for $\nu\geq 40$, generalized estimators (\ref{1eqn:Basu_et_al_estimators_for_student_t}) are close to MLE even for small sample size.}
\end{remark}
\vspace{0.2cm}

For $\alpha\notin ((d-2)/d,1)$, the support of Student distributions depend on the parameters. Thus Theorem \ref{1thm:estimating_equations_for_general_families} can not be used to find the estimators. However, in this case, one can find the estimators by maximizing the respective likelihood function as described in the following.

\subsection{Jones et al. estimation \cite{JonesHHB01J} under Student distributions for $\nu\notin (0,\infty)$}
\label{1sec:Jones_et_al_on_student_t_alpha_bigger_1}

For simplicity we deal only the one-dimensional case. Suppose that $X_1,\ldots,X_n$ is an i.i.d. sample where $X_1\leq \cdots \leq X_n$. Suppose also that
the true distribution $p$ is
a Student distribution with some known $\alpha>1$ (that is, $\nu<-1$) and variance, say 
$\sigma^2=1$:
\begin{equation}
\label{1eqn:student_t_for_alpha_greater_1}
p_\mu (x) = N_{\alpha} [1 +b_\alpha(x -\mu)^2]^{\frac{1}{\alpha -1}}_+ ,
\end{equation}
where
$N_\alpha$ is the normalizing factor. The support of $p_\mu$ is given by $\mathbb{S} = \lbrace x : \mu - c_\alpha \leq x \leq \mu +c_\alpha\rbrace$,
where $c_\alpha:=\sqrt{-1\big/ b_\alpha}$. 
(Recall that $b_\alpha<0$ for $\alpha>1$).
Observe that (\ref{1eqn:student_t_for_alpha_greater_1}) defines
an $\mathbb{M}^{(\alpha)}$-family
whose support depends on the unknown parameter.
We now show that the Jones {et al.} estimator of $\mu$
could be different from $\overline{X}$.
Since the support of $p_\mu$
depends on $\mu$, we cannot apply Theorem \ref{1thm:estimating_equations_for_general_families}. Hence we resort to the
maximization of the associated likelihood function:
\begin{equation*}
L_3^{(\alpha)}(\theta) = \dfrac{\alpha}{\alpha-1}\ln\Big[\dfrac{1}{n}\sum\limits_{i=1}^n p_\theta({X}_i)^{\alpha-1}\Big] - \ln\Big[\int p_\theta({x})^\alpha dx
\Big] .
\end{equation*}
The likelihood function, for
(\ref{1eqn:student_t_for_alpha_greater_1}), becomes
\begin{align}
\label{1eqn:likelihood_function_of_student_t_for_alpha_2}
\lefteqn{L_3^{(\alpha)}(\mu\mid X_1,\ldots,X_n)}\nonumber\\
&=
\frac{\alpha}{\alpha-1}\ln \Big[ \tfrac{1}{n} \sum\limits_{i=1}^n 
N_\alpha ^{\alpha -1}~[1 +b_\alpha(X_i -\mu)^2]~ {\bf 1}(\mu -c_\alpha
\leq X_i \leq
\mu +c_\alpha)\Big] -\ln\big( \mathbb{E}_\mu \big[N_{\alpha}^{\alpha-1}\big\lbrace 1 +b_\alpha{(X -\mu)^2}\big\rbrace\big]\big)\nonumber\\
&=  \frac{\alpha}{\alpha-1}\ln \Big[ \frac{N_\alpha^{\alpha -1}}{n} \sum\limits_{i=1}^n [1 +b_\alpha(X_i -\mu)^2]~ {\bf 1}(X_i -c_\alpha\leq \mu \leq
X_i +c_\alpha)\Big] - \ln [ N_\alpha^{\alpha-1}(1+b_\alpha)]\nonumber\\
&= \frac{\alpha}{\alpha-1}\ln \Big[ \frac{N_\alpha^{\alpha-1}}{n} \sum\limits_{i=1}^n [1 +b_\alpha(X_i -\mu)^2]~ {\bf 1}(\mu\in I_i)\Big] - \ln [ N_\alpha^{\alpha-1}(1+b_\alpha)],\nonumber
\end{align}
where ${\bf 1}(\cdot)$ denotes the indicator function
and $I_i=[X_i -c_\alpha, X_i +c_\alpha]$, $i\in\{1,\ldots,n\}$. Observe that the maximizer of 
$L_3^{(\alpha)}(\mu)$ is same as the maximizer of
\begin{equation}
\label{1eqn:alternative_likelihood_function_of_student_t_for_alpha_2}
\ell^{(\alpha)}(\mu):=
\sum\limits_{i=1}^n [1 +b_\alpha(X_i -\mu)^2]~ {\bf 1}(\mu\in I_i).
\end{equation}
(Note that, Basu et al. likelihood function (\ref{1eqn:likelihood_function_for_B_alpha}) for the model in (\ref{1eqn:student_t_for_alpha_greater_1}) also reduces to (\ref{1eqn:alternative_likelihood_function_of_student_t_for_alpha_2}); hence both the estimators are the same in this case.)
It
is clear from (\ref{1eqn:alternative_likelihood_function_of_student_t_for_alpha_2})
that $\ell^{(\alpha)}(\mu)$ is positive if and only if $\mu$ lies at least in one $I_i$ for $i\in\{1,\ldots,n\}$.
Thus, to find the maximizer $\widehat{\mu}$ of $\ell^{(\alpha)}(\mu)$,
we only need to consider the cases when $\mu$ lies in one of the
$I_i$'s.

Consider $I_1$. If $I_1$ is disjoint from all other
$I_j$ for $j\neq 1$, then $\ell^{(\alpha)}(\mu)$ equals to 
$[1 +b_\alpha(X_1 -\mu)^2]$
for $\mu\in I_1$. Similarly, if $I_1\cap I_2\neq\emptyset$ but all
other $I_j$'s for $j\notin \{1,2\}$, are disjoint from $I_1$, then the value
of $\ell^{(\alpha)}(\mu)$ in $I_1$ is given by:
\begin{eqnarray*}
	{\ell^{(\alpha)}(\mu)} = \left\{
	\begin{array}{ll}
		{[1 +b_\alpha(X_1 -\mu)^2],} &\hbox{~} 
		\mu\in I_1\setminus I_2,\\
		{[2 +b_\alpha(X_1 -\mu)^2 +b_\alpha(X_2 -\mu)^2],} &\hbox{~} 
		\mu\in I_1\cap I_2.
	\end{array}
	\right.
\end{eqnarray*}
In general, if $I_2,I_3,\ldots,I_k$ for some $k\in\{2,\dots,n\}$ satisfy
$\cap_{i=1}^{k}I_i\neq\emptyset$                                                     
and all other $I_j$'s for $j\notin \{1,\dots,k\}$ are disjoint from
$I_1$, then 
$I_1$ can be divided into $k$ disjoint sub-intervals and 
in each of these sub-intervals
the value of $\ell^{(\alpha)}(\mu)$  
is given, for $\mu \in \big(\cap_{i=1}^j I_i\big)\setminus \big(\cup_{i=j+1}^k I_i\big),~j\in\{1,\ldots,k\}$, where $\cup_{k+1}^k I_i:=\emptyset$,  by
\begin{equation*}
\ell^{(\alpha)}(\mu) =
\sum\limits_{i=1}^j [1 +b_\alpha(X_i -\mu)^2].
\end{equation*}
Let us define ${\bf{1}}(\mu\in\emptyset) =0$. Then for $\mu\in I_1$, we can write
\begin{eqnarray*}
\label{1eqn:likelihood_in_I_1_student}
\lefteqn{\ell^{(\alpha)}(\mu)}
&&\hspace*{0.5cm} =
[1 +b_\alpha(X_1 -\mu)^2]~ {\bf{1}}\big(\mu\in I_1\setminus \cup_{r=2}^n I_r\big)\nonumber +
\Big\{\sum\limits_{i=1}^2 [1 +b_\alpha(X_i -\mu)^2]\Big\}~{\bf{1}}\big(\mu \in \big(\cap_{r=1}^2 I_r\big)\setminus \big(\cup_{r=3}^n I_r\big)\big)\nonumber\\
&&\hspace*{1cm} +\cdots+
\Big\{\sum\limits_{i=1}^{k} [1 +b_\alpha(X_i -\mu)^2]\Big\}~{\bf{1}}\big(\mu \in \big(\cap_{r=1}^k I_r\big)\setminus \big(\cup_{r=k+1}^n I_r
\big)\big)\nonumber\\
&&\hspace*{1cm} +\cdots +
\Big\{\sum\limits_{i=1}^n [1 +b_\alpha(X_i -\mu)^2]\Big\}~{\bf{1}}\big(\mu \in \big(\cap_{r=1}^n I_r\big)\big)\nonumber\\
&&\hspace*{0.3cm} =\sum\limits_{j=1}^n\Big\{\Big[\sum\limits_{i=1}^j 
[1 +b_\alpha(X_i -\mu)^2]\Big]
~{\bf{1}}\big(\mu \in \big(\cap_{r=1}^j I_r\big)\setminus \big(\cup_{r=j+1}^n I_r\big)\big)\Big\}.
\end{eqnarray*}
Let $I_2':=I_2\setminus I_1$. Then proceeding as above,
for $\mu\in I_2'$, we have
\begin{eqnarray*}
\label{1eqn:likelihood_in_I_2_student}
\lefteqn{\ell^{(\alpha)}(\mu)} &&\hspace*{0.5cm} =
[1 +b_\alpha(X_2 -\mu)^2]~ {\bf{1}}\big(\mu\in I_2'\setminus \cup_{r=3}^n I_r\big) +
\Big\{\sum\limits_{i=2}^3 [1 +b_\alpha(X_i -\mu)^2]\Big\}~{\bf{1}}\big(\mu \in \big(
I_2'\cap I_3\big)\setminus \big(\cup_{r=4}^n I_r\big)\big) \nonumber\\
&&\hspace*{1cm}
+\cdots+\Big\{\sum\limits_{i=2}^{k} [1 +b_\alpha(X_i -\mu)^2]\Big\}~{\bf{1}}\big(\mu \in \big(I_2'\cap (\cap_{r=3}^k I_r)\big)\setminus \big(\cup_{r=k+1}^n I_r\big)\big)\nonumber\\
&&\hspace*{1cm} +\cdots +
\Big\{\sum\limits_{i=2}^n [1 +b_\alpha(X_i -\mu)^2]\Big\}~{\bf{1}}\big(\mu \in \big(I_2'\cap(\cap_{r=3}^n I_r)\big)\big)\nonumber\\
&&\hspace*{0.3cm} = \sum\limits_{j=2}^n\Big\{\Big[\sum\limits_{i=2}^j 
[1 +b_\alpha(X_i -\mu)^2]\Big]
~{\bf{1}}\big(\mu \in \big(I_2'\cap (\cap_{r=2}^j I_r)\big)\setminus \big(\cup_{r=j+1}^n I_r\big)\big)\Big\}.
\end{eqnarray*}
In general, let $I_k':=I_k\setminus \big(\cup_{i=1}^{k-1} I_i\big)$, $k\in\{1,\ldots,n\}$. Then for $\mu\in I_k'$, 
\begin{eqnarray}
\label{1eqn:likelihood_for_any_k_sample_student}
\ell^{(\alpha)}(\mu)
=\sum\limits_{j=k}^n\Big\{\Big[\sum\limits_{i=k}^j 
[1 +b_\alpha(X_i -\mu)^2]\Big]
~{\bf{1}}\big(\mu \in \big(I_k'\cap (\cap_{r=k}^j I_r)\big)\setminus \big(\cup_{r=j+1}^n I_r\big)\big)\Big\}.
\end{eqnarray}

\begin{figure}
	\centering
	\includegraphics[width=14cm, height=1.5cm]{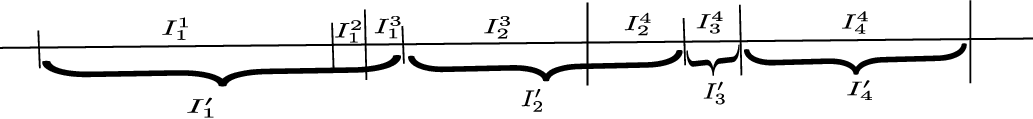}
	\caption{An example of the intervals $I_1'$, $I_2'$, $I_3'$ and $I_4'$ and its subintervals with four sample points $X_1,X_2,X_3$ and $X_4$, where $I_1'=[X_1-c_\alpha,X_1+c_\alpha]$, $I_2'= [X_1+c_\alpha, X_2+c_\alpha]$, $I_3'= [X_2+c_\alpha, X_3+c_\alpha]$ and $I_4'= [X_3+c_\alpha, X_4+c_\alpha]$.}
	\label{1fig:subintervals}
\end{figure}

Hence for $\mu\in I_k'$, $k \in \{1,\dots,n\}$,
we can divide $I_k'$ into at most $(n-k+1)$ sub-intervals
$I_k^j:=\big[\big(I_k'\cap (\cap_{i=k}^j I_i)\big)\setminus \big(\cup_{i=j+1}^n I_i\big)\big]$, $j\in\{k,\dots,n\}$, such
that the indicator functions in 
(\ref{1eqn:likelihood_for_any_k_sample_student}) will be positive for 
$\mu$ in either of these sub-intervals.
For example, in Figure \ref{1fig:subintervals} we considered a case
where $I_1'$ can be divided into three disjoint sub-intervals,
namely $I_1^1$, $I_1^2$ and $I_1^3$, and $I_2'$ can be divided
into two disjoint sub-intervals $I_2^3$ and $I_2^4$, and so on.
The maximizer of $\ell^{(\alpha)}(\mu)$ for $\mu$ in each of these
sub-intervals $I_k^j$ can be found in the following way.

Let $k$ and $j$ be such that $I_k^j\neq \emptyset$. 
Then the following possible cases appear:
\begin{itemize}
	\item[(i)] $j=k$ and $I_k^{j+1}\neq \emptyset$:
	\begin{eqnarray*}
		{I_k^j} = \left\{
		\begin{array}{ll}
			{[X_{k-1}+c_\alpha,X_{k+1}-c_\alpha],} &\hbox{~if~} 
			I_k'\neq I_k,\\
			{[X_k-c_\alpha,X_{k+1}-c_\alpha],} &\hbox{~if~} 
			I_k'=I_k,
		\end{array}
		\right.
	\end{eqnarray*}
	\item[(ii)] $j=k$ and $I_k^{j+1}= \emptyset$:
	\begin{eqnarray*}
		{I_k^j} = \left\{
		\begin{array}{ll}
			{[X_{k-1}+c_\alpha,X_{k}+c_\alpha],} &\hbox{~if~} I_k'\neq I_k,\\
			{[X_k-c_\alpha,X_{k}+c_\alpha],} &\hbox{~if~} 
			I_k'=I_k,
		\end{array}
		\right.
	\end{eqnarray*}
	\item[(iii)] $j>k$ and $I_k^{j+1}\neq \emptyset$:
	\begin{eqnarray*}
		{I_k^j} = \left\{
		\begin{array}{ll}
			{[\max \{X_{j}-c_\alpha, X_{k-1}+c_\alpha\},X_{j+1}-c_\alpha],} &\hbox{~if~} 
			I_k'\neq I_k,\\
			{[X_{j}-c_\alpha,X_{j+1}-c_\alpha],} &\hbox{~if~} I_k'=I_k,
		\end{array}
		\right.
	\end{eqnarray*}
	\item[(iv)]$j>k$ and $I_k^{j+1}= \emptyset$:
	\begin{eqnarray*}
		{I_k^j} = \left\{
		\begin{array}{ll}
			{[\max \{X_{j}-c_\alpha, X_{k-1}+c_\alpha\}, X_{k}+c_\alpha],} &\hbox{~if~} 
			I_k'\neq I_k,\\
			{[X_{j}-c_\alpha,X_{k}+c_\alpha],} &\hbox{~if~} 
			I_k'=I_k.
		\end{array}
		\right.
	\end{eqnarray*}
\end{itemize}
Also from (\ref{1eqn:likelihood_for_any_k_sample_student}), for $\mu\in I_k^j$,
$\ell^{(\alpha)}(\mu)=\sum\limits_{i=k}^j [1+b_\alpha(X_i-\mu)^2]$.
Since
\begin{eqnarray*}
\label{1eqn:deriavtive_of_l_2}
\frac{\partial \ell^{(\alpha)}(\mu)}{\partial \mu} =
\sum\limits_{i=k}^{j}\frac{\partial}{\partial \mu} [1+b_\alpha(X_i-\mu)^2]= 
-2b_\alpha\sum\limits_{i=k}^{j} (X_{i}-\mu),
\end{eqnarray*}
$\ell^{(\alpha)}(\mu)$ is monotone increasing for $\mu \leq \frac{1}{j-k+1}
\sum_{i=k}^{j}X_{i}$, and monotone decreasing otherwise. 
Thus the local maximizer
of $\ell^{(\alpha)}(\mu)$ in any non-empty $I_k^j$ for $j\in\{k,k+1,\dots,n\}$ 
is given by 
\begin{align}
\label{1eqn:formula_for_estimator_student_alpha_2}
\text{the median of}~\Big\lbrace I_L, I_R, \tfrac{1}{j-k+1}\sum_{i=k}^{j}X_{i}\Big\rbrace,
\end{align}
where $I_L$ and $I_R$ are respectively the left and right ends of
the sub-interval $I_k^j$ as in (i)-(iv).
\begin{figure}
	\centering
	\hspace*{-4cm}
	\begin{minipage}{0.24\textwidth}
		\includegraphics[width=7cm, height=5cm]{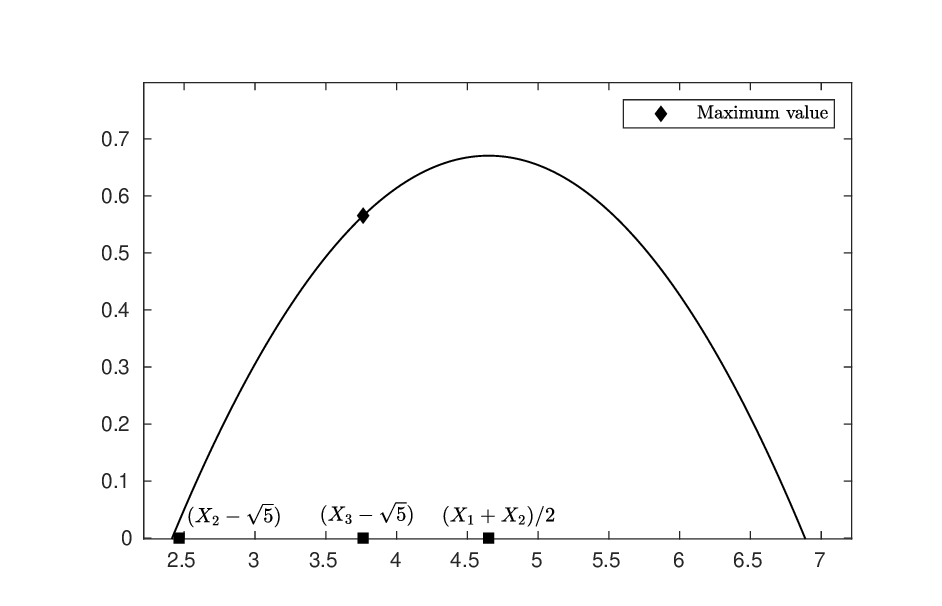}
	\end{minipage}
	\hspace*{3.3cm}
	\begin {minipage}{0.24\textwidth}
	\includegraphics[width=6.8cm, height=4.6cm]{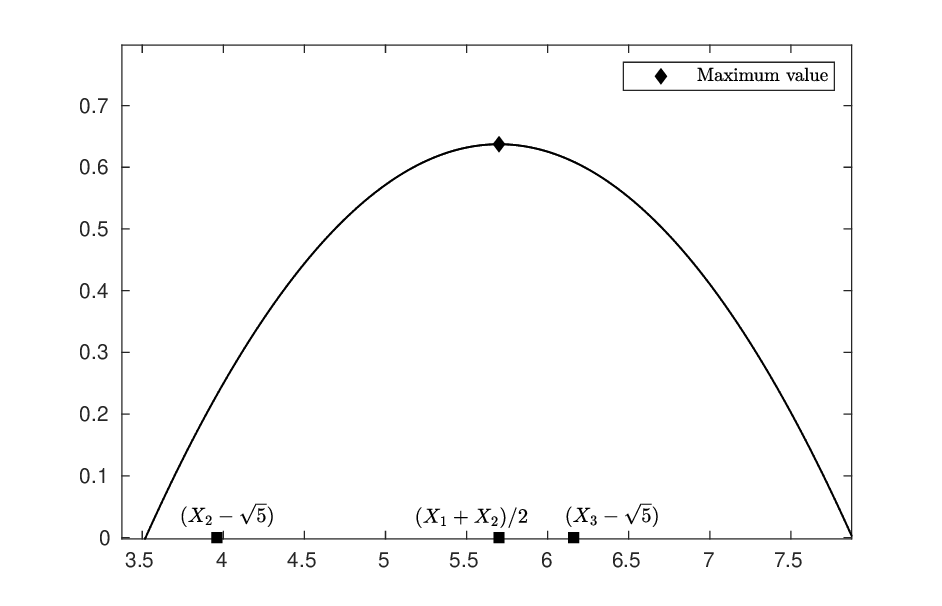}
\end{minipage}
\caption{An example of two different cases of maximizers in $[X_2-\sqrt{5},X_3-\sqrt{5}]$. In the first picture, the maximizer is at the right end of the interval, whereas, in the second, it is the average of $X_1$ and $X_2$.}
\label{1fig:different_maximizers}
\end{figure}
Figure \ref{1fig:different_maximizers} shows two different cases
of maximizer of $\ell^{(\alpha)}(\mu)$ for $\alpha=2$ ($c_\alpha=\sqrt{5} $ here) in
the sub-interval $[X_2-\sqrt{5},X_3-\sqrt{5}]$.

Observe that, in this process
we divide $\cup_{i=1}^n I_i$ into a {\em finite} number of
non-empty disjoint sub-intervals 
such that $\ell^{(\alpha)}(\mu)$ is positive if and only if
$\mu$ lies in one of these sub-intervals. In each of these sub-intervals,
$\ell^{(\alpha)}(\mu)$ has a unique maximizer.
Thus we have a finite number of local maximizers 
of $\ell^{(\alpha)}(\mu)$ in $\cup_{i=1}^n I_i$, and hence
the global maximizer is one among these local maximizers. This implies that
the global maximizer can be different from the sample mean. \vspace{0.2cm}

\begin{remark}
\emph{	Observe that $c_\alpha\to \infty$  when $\alpha\to 1$. Thus
	for $\alpha\to 1$, the length
	of the intervals $I_i=[X_i-c_\alpha,X_i+c_\alpha]$ for $i\in \{1,\dots,n\}$ increases, 
	and hence all the indicator
	functions in (\ref{1eqn:alternative_likelihood_function_of_student_t_for_alpha_2}) become positive for any $\mu\in \mathbb{R}$. This implies that the
	maximizer of $\ell^{(\alpha)}(\mu)$ is the usual sample mean
	$\overline{X}$. 
	This complies with the case $\alpha=1$ as
	the MLE of the mean parameter of a normal
	distribution is $\overline{X}$. Recall that
	the likelihood function $L_3^{(\alpha)}$ coincides with the
	usual $\log$ likelihood function and the Student distribution
	coincides with the normal distribution as $\alpha\to 1$.}
\end{remark}

To demonstrate the algorithm, we generated the following random sample from
the mixture $0.8p+0.2\mathcal{N}(10,1)$, where
$p$ is the Student distribution with $\alpha=2$, $\mu=0$ and $\sigma=1$:
\begin{equation}
\label{1eqn:sample}
   \begin{bmatrix} 
   -1.7287& -1.1761& -1.0597& -0.3236& -0.2340\\
   0.4706& 0.4712& 0.5435& 0.6309& 0.7533 \\
   0.8020& 0.9237& 1.1394& 1.4373& 1.5351  \\
   1.6941& 8.6501& 10.7254& 12.7694& 13.0349.\\
   \end{bmatrix} 
\end{equation}
Consider $\mu\in I'_1=[-3.9648, 0.5074]$. Then 
$I_1^j\neq\emptyset$ for all $j\in\{1,\dots,16\}$. Using the formula 
(\ref{1eqn:formula_for_estimator_student_alpha_2}), we have the following 
local maximizers of $\ell^{(\alpha)}(\mu)$ in each of these sub-intervals $I_1^j$, $j\in\{1,\dots,16\}$,  of $I_1'$:
\begin{equation*}
   \begin{bmatrix} 
   -3.4122& -3.2958& -2.5597& -2.4701& -1.7655& -1.7649& -1.6926& -1.6052\\
    -1.4828& -1.4341& -1.3124& -1.0967& -0.7988& -0.7010& -0.5420& 0.3674. \\
   \end{bmatrix} 
\end{equation*}
Next consider $\mu\in I_2'=[0.5074, 1.06]$. 
Then the indicator functions in (\ref{1eqn:likelihood_in_I_2_student})
are positive only when $j = 16$, that is, $I_2^j\neq \emptyset$ only for
$j=16$. Using
(\ref{1eqn:formula_for_estimator_student_alpha_2}), we get the maximizer of $\ell^{(2)}(\mu)$ in $I_2^{16}(=I_2')$ is $0.5074$. Similarly, we have only one maximizer in each $I_i'$ for $i\in\{3,\dots,16\}$
and they respectively are:
\begin{equation*}
   \begin{bmatrix} 
   1.0600& 1.1764& 1.9125& 2.0021& 2.7067& 2.7073& 2.7796\\
    2.8670& 2.9894& 3.0389& 3.1598& 3.3755& 3.6734& 3.7712. \\
   \end{bmatrix} 
\end{equation*}
Next consider $\mu\in I_{17}'=[6.4140, 10.8862]$. Then 
$I_{17}^j\neq\emptyset$ for
$j\in\{17,\dots,20\}$. We then have four local maximizers of
$\ell^{(2)}(\mu)$ in four sub-intervals of $I_{17}'$, namely $I_{17}^j$ for
$j\in\{17,\dots,20\}$,
and they respectively are $8.4893$, $9.6878$,
$10.7150, 10.8862$. Similarly for $\mu\in I_i', i\in\{18, 19, 20\}$, 
$\ell^{(2)}(\mu)$ has only one local
maximizer in each $I_i'$ and they respectively are $12.1766, 12.9615, 15.0055$.

Comparing the values of $\ell^{(2)}(\mu)$ at each of the local maximizers, we get
$0.3674$ as the global maximizer of $\ell^{(2)}(\mu)$. Hence the Jones et al. (or Basu et al.) estimator for $\mu$ is
$\widehat{\mu}= 0.3674$, which is different from $\overline{X} = 1.2940$. Also MLE and MLE after deleting the outliers are, respectively, 2.173 and 0.1721. We repeated this exercise with $p$ being a Student distribution with mean zero and $\nu = -5$ ($\alpha = 1.5$). The results are shown in Table \ref{1tab:comp_esti}. We observe that Jones et al. (or Basu et al.) estimator is close to the true parameter (also to the ML-estimator without outliers) as $\alpha$ gets close to $1$ from right.

\begin{table}
    \centering
    \caption{Comparison of Jones et al. estimator of mean parameter with its ML-estimator for the mixture $0.8p+0.2\mathcal{N}(10,1)$, where $p$ is the Student distribution with $\alpha=2$, $\mu=0$ and $\sigma=1$.}
    \begin{tabular}{|c| c| c| c|}
    \hline
        Degrees of &  \multicolumn{3}{c|}{Estimators by different methods}\\
       \cline{2-4}
         freedom &  & MLE & Jones et al.\\
            of the model    &  MLE   & (without outliers) & estimator\\
        \hline
        &&&\\
        $\nu =-3$ & ~~2.1733 & 0.1721 & 0.3674\\
        ($\alpha = 2$)&&&\\ 
			&&&\\
			\hline
				&&&\\
		$\nu =-5$ & ~~2.2730 & 0.2770 & 0.3186\\
        ($\alpha = 1.5$)&&& \\ 
        	&&&\\
			\hline
    \end{tabular}
    \label{1tab:comp_esti}
\end{table}

\subsection{Generalized Hellinger estimation under Cauchy distributions}

Consider the Cauchy distributions as in (\ref{1eqn:cauchy_distribution_as_E_alpha}). Here we find the generalized Hellinger estimators for their location and scale parameters using a kernel density estimate for the sample empirical measure.
Let $\textbf{X}_i = [X_{1i},\dots,X_{di}]^\top$, $i\in\{1,\dots,n\}$ be an
i.i.d. sample. Let $\widetilde{p}_n$ be a suitable continuous density estimator of the empirical measure $p_n$. When $\beta \in\big(1, (d + 2)/d\big)$, we saw that Cauchy distributions form a regular $\mathscr{E}^{(\beta)}$-family. Thus in this case we use Corollary \ref{1cor:similarity_of_estimating_equations} to estimate its parameters.
But for $\beta\notin \big(1,(d+2)/d\big)$, the support of this distribution depends on the unknown parameters. In this case one can estimate
the parameters by maximizing the associated likelihood function as we did for Student distributions in Section \ref{1sec:Jones_et_al_on_student_t_alpha_bigger_1}.

Let $\beta \in\big(1, (d + 2)/d\big)$. The characterizing entities of Cauchy distributions as a regular $\mathscr{E}^{(\beta)}$ are respectively $h({\bf{x}})\equiv 1$, $f^{(1)}({\bf{x}}) = {\bf{x}}$ and $f^{(2)}({\bf{x}}) = \text{vec}({\bf{x}}
{\bf{x}}^T)$ (see Example \ref{1expl:cauchy_distribution}). Using (\ref{1eqn:estimating_equation_for_regular_E_alpha}), we therefore have the following estimating equations
\begin{align}
\label{1eqn:estimating_equations_for_cauchy_distribution_as_D_alpha}
\mathbb{E}_{\eta^{(\beta)}}[\textbf{X}] = \mathbb{E}_{\widetilde{p}_n^{(\beta)}}[{\bf{X}}]
~\text{and}~~ \mathbb{E}_{\eta^{(\beta)}}[\text{vec}(\textbf{X}\textbf{X}^T)]
=\mathbb{E}_{\widetilde{p}_n^{(\beta)}}[\text{vec}(\textbf{X}\textbf{X}^T)].
\end{align}

Let us find $\mathbb{E}_{\eta^{(\beta)}}[\textbf{X}]$ and $\mathbb{E}_{\eta^{(\beta)}}[\text{vec}(\textbf{X}\textbf{X}^T)]$.
In Example \ref{1expl:cauchy_distribution}, we saw that
$p_\theta^{(\alpha)} = q_\eta,~\text{where}~\alpha=1/\beta$.
Hence $q_\eta^{(\beta)} =p_\theta$. 
Thus
\begin{equation}
\mathbb{E}_{\eta^{(\beta)}} [\textbf{X}] = \mathbb{E}_\theta[\textbf{X}]
= \mu~\text{and}~
\mathbb{E}_{\eta^{(\beta)}}[X_iX_j] = \mathbb{E}_\theta[ X_iX_j]
= k_{ij} + \mu_i\mu_j,
\end{equation}
where ${\bf{X}}=[X_1,\dots,X_d]^T$ and $k_{ij}=[\nu/(\nu-2)]\cdot \sigma_{ij}$. Using this in (\ref{1eqn:estimating_equations_for_cauchy_distribution_as_D_alpha}),
we get 
\begin{eqnarray}
\label{1eqn:estimators_for_cauchy_distribution}
{\mu} = \mathbb{E}_{\widetilde{p}_n^{(\beta)}}[{\bf{X}}],~~
{k_{ij}} = \mathbb{E}_{\widetilde{p}_n^{(\beta)}}[X_iX_j] - \widehat{\mu_i}\widehat{\mu_j},
\end{eqnarray}
for $i,j\in\{1,\dots,d\}$, $i\leq j$. Thus generalized Hellinger estimators of the location and scale parameters are 
\begin{eqnarray}
\label{1eqn:estimators_for_cauchy_distribution1}
\widehat{\mu} = \mathbb{E}_{\widetilde{p}_n^{(\beta)}}[{\bf{X}}],~~
\widehat{\sigma_{ij}} = \Big(\mathbb{E}_{\widetilde{p}_n^{(\beta)}}[X_iX_j] - \widehat{\mu_i}\widehat{\mu_j}\Big)\Big/\Big(\nu/[\nu-2]\Big)
\end{eqnarray}
for $i,j\in\{1,\dots,d\}$ and $i\leq j$. We summarize these in the following theorem.
\vspace*{0.2cm}

\begin{theorem}
\label{1thm:estimator_for_cauchy_distribution}
Let $\beta\in\big(1,(d+2)/d\big)$ and $\widetilde{p}_n$ be a suitable continuous estimate of the sample empirical measure. Then the generalized Hellinger estimators of the location and scale parameters of a $d$-dimensional Cauchy distribution are given by (\ref{1eqn:estimators_for_cauchy_distribution1}).
\end{theorem}

\begin{remark}
	In Example \ref{1expl:student_t_E_alpha}, We saw that Student distributions form a regular $\mathscr{E}^{(\alpha')}$-family for $\alpha'\in\big(1,(d+2)/d\big)$. Thus one can do the $\alpha'$-generalized Hellinger estimation on Student distributions as well. 
\end{remark}
\vspace*{0.2cm}

Notice that the estimators in (\ref{1eqn:estimators_for_cauchy_distribution1}) involve a continuous estimate $\widetilde{p}_n$ of
$p_n$. In the following we present examples
where we use `kernel density estimation' to find such $\widetilde{p}_n$ and
use it to find the estimators. In the
literature the commonly used kernel to estimate the $d$-dimensional
empirical measure is of the following form (see \cite{TamuraB86J}):
\begin{equation}
\label{1eqn:general_form_of_kernel}
\widetilde{p}_n({\bf{x}}) = \frac{1}{nh_n^d}\sum\limits_{i=1}^n \xi\Big(
\frac{{\bf{x}}-\textbf{X}_i}{h_n}\Big),
\end{equation}
where $\xi$ is a symmetric distribution on $\mathbb{R}^d$ and 
$\lbrace h_n\rbrace$ is a sequence of real numbers with suitable properties,
called bandwidth.
These properties of the kernel $\xi$ and bandwidth $h_n$ influence the performance of the estimators greatly. There is no general theory in the literature
to choose the right continuous estimate for a given problem.
However authors like Beran
\cite{Beran77J}, Tamura and Boos \cite{TamuraB86J}, and Simpson \cite{Simpson89J} imposed conditions on $\xi$ and $h_n$ so that the estimators perform better
in their specific setting. We shall use the following two kernels, namely {\em uniform kernel} and {\em Epanechnikov kernel}, with bandwidth $h_n = n^{-1/2d}$ for $n\geq 1$ to find the estimators. Both the kernels and $h_n$ satisfy the following conditions which guarantee the $L_1$ convergence 
of $\widetilde{p}_n$ to the true density \cite[Lem. 3.1]{TamuraB86J} (see also \cite[Sec. 3.3]{BasuSP11B} and the references therein):

\begin{itemize}
	\item[(i)] $\xi$ is symmetric about $0$ and has compact support.
	\item[(ii)] $\lim\limits_{n\to\infty} h_n = 0$ and 
	$\lim\limits_{n\to\infty} [h_n + (nh_n^d)^{-1}] = 0$.
\end{itemize}
\vspace*{0.2cm}

We first use the $d$-dimensional uniform kernel which is defined as follows.
\begin{displaymath}
\xi_u({\bf{x}}) = \left\{
\begin{array}{ll}
{ \frac{1}{2^d}\quad \text{if}~ {\bf{x}}\in [-1,1]^d }  
\\0~ \quad\text{otherwise}
.
\end{array}
\right
.
\end{displaymath}

Let us denote
$[X_{1i}-n^{-1/2d},X_{1i}+n^{-1/2d}]\times \dots \times [X_{di}-n^{-1/2d},X_{di}+n^{-1/2d}]$ by 
$\big[\textbf{X}_i-n^{-1/2d},\textbf{X}_i+n^{-1/2d}\big]$ for
$i\in\{1,\dots,n\}$ and call them rectangles. Assume that all these
rectangles are disjoint (these are actually disjoint for sufficiently large 
$n$).
Then from (\ref{1eqn:general_form_of_kernel}) we have
\begin{equation*}
\widetilde{p}_n({\bf{x}}) = n^{-1/2}\sum\limits_{i=1}^n \xi_u\big(n^{1/2d}[{\bf{x}} - \textbf{X}_i]\big).
\end{equation*}
That is,
\begin{displaymath}
\widetilde{p}_n({\bf{x}}) = \left\{
\begin{array}{ll}
{ n^{-1/2}2^{-d}~~ \text{if}~ {\bf{x}}\in 
[{\bf{X}}_i-n^{-1/2d},{\bf{X}}_i+ n^{-1/2d}]~\text{for~} i\in\{1,\dots,n\},}  
\\
0 \hspace*{1.2cm}\text{otherwise}
.
\end{array}
\right.
\end{displaymath}
Thus $\widetilde{p}_n$ is the uniform distribution on 
$\bigcup_{i=1}^n\big[\textbf{X}_i-n^{-1/2d},\textbf{X}_i+n^{-1/2d}\big].$
This implies that the $\beta$-scaled distribution $\widetilde{p}_n^{(\beta)}$ is the
same as $\widetilde{p}_n$.
Therefore, we have
\begin{equation*}
\mathbb{E}_{\widetilde{p}_n^{(\beta)}}[\textbf{X}] =
\mathbb{E}_{\widetilde{p}_n}[\textbf{X}] = \frac{1}{n}\sum_{i=1}^n \textbf{X}_i,
\end{equation*}
\begin{align*}
\mathbb{E}_{\widetilde{p}_n^{(\beta)}}[{X_j}^2] =
\mathbb{E}_{\widetilde{p}_n}[{X}_j^2] &=\frac{1}{n^{1/2}2^d}\sum_{l=1}^n
\tfrac{(2n^{-1/2d})^{d-1}[(X_{jl}+n^{-1/2d})^3 -(X_{jl}-n^{-1/2d})^3]}{3}\\
&= \frac{1}{n^{1/2}2^d}\sum_{l=1}^n
\frac{(2n^{-1/2d})^{d}[(3X_{jl}^2+n^{-2/2d})]}{3}\\
& =\frac{1}{n}\sum_{l=1}^n X_{jl}^2 +\frac{1}{3n^{1/d}},~\text{for}~j\in\{1,\dots,d\},
\end{align*}
and
\begin{align*}
\mathbb{E}_{\widetilde{p}_n^{(\beta)}}[{X_i X_j}] =
\mathbb{E}_{\widetilde{p}_n}[X_i {X}_j] &=\frac{1}{n^{1/2}2^d}\sum_{l=1}^n
\frac{(2n^{-1/2d})^{d-2}[4n^{-1/2d}X_{il}][4n^{-1/2d}X_{jl}]}{4}\\
&=\frac{1}{n}\sum_{l=1}^n X_{il} X_{jl}~
\text{for}~ i,j\in\{1,\dots,d\}~ \text{and}~ i< j.
\end{align*}
Hence the estimators for the parameters are given by 
\begin{eqnarray}
\label{1eqn:estimator_of_cauchy_distribution_with_uniform_kernel}
\widehat{\mu} =  \overline{\textbf{X}},
~~\widehat{k_{ij}}=[\nu/(\nu-2)]\cdot \widehat{\sigma_{ij}}  =
\frac{1}{n}\sum\limits_{l=1}^n X_{il}X_{jl} -
\widehat{\mu_i}\widehat{\mu_j} + \epsilon_n,
\end{eqnarray}
for $i,j=1,\dots,d$ and $i\leq j$, where $\epsilon_n=\frac{1}{3n^{1/d}}$ if $i=j$ and equals to
zero otherwise.
\vspace*{0.4cm}

Next we find the estimators using the following $d$-dimensional Epanechnikov kernel 
\begin{displaymath}
\xi_e({\bf{x}}) = \left\{
\begin{array}{ll}
{ \frac{d+2}{2^d}(1-\|{\bf{x}}\|^2)\quad \text{if}~ {\bf{x}}\in [-1,1]^d }  
\\0~ \hspace*{2cm}\text{otherwise}
,
\end{array}
\right
.
\end{displaymath}
where $\|\cdot\|$ denotes the Euclidean norm.
We thus have
\begin{displaymath}
\widetilde{p}_n({\bf{x}}) = \left\{
\begin{array}{ll}
{  \frac{d+2}{2^d}(1-\|n^{1/2d}({\bf{x}} -{\bf{X}}_i)\|^2)~~ \text{if}~ {\bf{x}}\in 
	[{\bf{X}}_i-n^{-1/2d},{\bf{X}}_i+ n^{-1/2d}]~\text{for~} i\in\{1,\dots,n\},}  
\\
0 \hspace*{3.5cm}\text{otherwise}
.
\end{array}
\right.
\end{displaymath}
This yields the same estimators for $\mu$ and $k_{ij}$ as in (\ref{1eqn:estimator_of_cauchy_distribution_with_uniform_kernel}) except the correction term $\epsilon_n$ which differs only up-to a scale factor. For example when $\beta=2$, $\epsilon_n$ changes to $\frac{1}{7n}$ for $d=1$, to $\frac{17}{91n^{1/2}}$ for $d=2$, to $\frac{25}{63n^{1/3}}$ for $d=3$, and so on. 

Observe that, for $\beta\in\big[(d+4)/(d+2), (d+2)/d\big)$ (that is, $\nu<2$), Cauchy distributions do not have finite mean and variance. Hence the estimates (\ref{1eqn:estimator_of_cauchy_distribution_with_uniform_kernel}) oscillate as one increases the sample size $n$. In this case some other smoothing techniques or changing the bandwidth in Theorem \ref{1thm:estimator_for_cauchy_distribution} may produce a better estimator. A general theory for this is part of our future work. However for $\beta\in\big(1,(d+4)/(d+2)\big)$, these estimators are close to the true parameters. In Table \ref{1tab:compare_hellinger_est_cauchy}, we summarize the results of a simulation study where we find the estimators by taking average of 25 different sets of random samples of size $50$ drawn from a standard Cauchy distribution using both the kernels for different degrees of freedom. We observe that, for $\nu>2$, the estimators are bounded, and as $\nu$ increases, their performances get better.

\begin{table}
	\centering
	\caption{$\beta$-Hellinger estimators of the location and scale parameters obtained by kernel density estimation method using uniform and Epanechnikov kernels with bandwidth $h_n=1/\sqrt{n}$ and sample size $n=50$.}
	\begin{tabular}{|*{5}{c|}}  
		\hline
	\multicolumn{1}{|c|}{} & \multicolumn{2}{|c|}{} & \multicolumn{2}{|c|}{} 	
	\\
	\multicolumn{1}{|c|}{Parameters} & \multicolumn{2}{|c|}{Estimators} & \multicolumn{2}{|c|}{Estimators using} 
	\\
	\multicolumn{1}{|c|}{($\mu_T=0,~\sigma_T=1$)} & \multicolumn{2}{|c|}{using uniform} & \multicolumn{2}{|c|}{Epanechnikov}
	\\ 
	\multicolumn{1}{|c|}{} & \multicolumn{2}{|c|}{kernel} & \multicolumn{2}{|c|}{kernel}
	\\\hline	
	\multicolumn{1}{|c|}{} & \multicolumn{1}{|c|}{$\widehat{\mu}$} & 
	\multicolumn{1}{|c|}{$\widehat{\sigma}$} & \multicolumn{1}{|c|}{$\widehat{\mu}$} & \multicolumn{1}{|c|}{$\widehat{\sigma}$}
	\\\hline
	\multicolumn{1}{|c|}{$\nu=1$} & \multicolumn{1}{|c|}{} & 
	\multicolumn{1}{|c|}{} & \multicolumn{1}{|c|}{} & \multicolumn{1}{|c|}{}\\
	
	\multicolumn{1}{|c|}{($\beta=2$)}& \multicolumn{1}{|c|}{0.7689} & 
	\multicolumn{1}{|c|}{21.1981} & \multicolumn{1}{|c|}{0.7689} & \multicolumn{1}{|c|}{21.3493} 
	\\\hline	
	\multicolumn{1}{|c|}{$\nu=2.1$} & \multicolumn{1}{|c|}{} & 
	\multicolumn{1}{|c|}{} & \multicolumn{1}{|c|}{} & \multicolumn{1}{|c|}{}\\
	
	\multicolumn{1}{|c|}{($\beta=51/31$)}& \multicolumn{1}{|c|}{0.1535} & 
	\multicolumn{1}{|c|}{1.6649} & \multicolumn{1}{|c|}{0.1535} & \multicolumn{1}{|c|}{1.6844} 
	\\\hline
	\multicolumn{1}{|c|}{$\nu=3$} & \multicolumn{1}{|c|}{} & 
	\multicolumn{1}{|c|}{} & \multicolumn{1}{|c|}{} & \multicolumn{1}{|c|}{}\\
	
	\multicolumn{1}{|c|}{($\beta=3/2$)}& \multicolumn{1}{|c|}{0.0480} & 
	\multicolumn{1}{|c|}{1.4575} & \multicolumn{1}{|c|}{0.0480} & \multicolumn{1}{|c|}{1.4567} 
	\\\hline
	\multicolumn{1}{|c|}{$\nu=4$} & \multicolumn{1}{|c|}{} & 
	\multicolumn{1}{|c|}{} & \multicolumn{1}{|c|}{} & \multicolumn{1}{|c|}{}\\
	
	\multicolumn{1}{|c|}{($\beta=7/5$)}& \multicolumn{1}{|c|}{0.0585} & 
	\multicolumn{1}{|c|}{1.3061} & \multicolumn{1}{|c|}{0.0585} & \multicolumn{1}{|c|}{1.3049} 
	\\\hline
	\multicolumn{1}{|c|}{$\nu=5$} & \multicolumn{1}{|c|}{} & 
	\multicolumn{1}{|c|}{} & \multicolumn{1}{|c|}{} & \multicolumn{1}{|c|}{}\\
	
	\multicolumn{1}{|c|}{($\beta=4/3$)}& \multicolumn{1}{|c|}{0.0119} & 
	\multicolumn{1}{|c|}{1.2707} & \multicolumn{1}{|c|}{0.0119} & \multicolumn{1}{|c|}{1.2736} 
	\\\hline
	\multicolumn{1}{|c|}{$\nu=7$} & \multicolumn{1}{|c|}{} & 
	\multicolumn{1}{|c|}{} & \multicolumn{1}{|c|}{} & \multicolumn{1}{|c|}{}\\
	
	\multicolumn{1}{|c|}{($\beta=5/4$)}& \multicolumn{1}{|c|}{-0.0383} & 
	\multicolumn{1}{|c|}{1.1307} & \multicolumn{1}{|c|}{-0.0383} & \multicolumn{1}{|c|}{1.1312} 
	\\\hline
	\multicolumn{1}{|c|}{$\nu=10$} & \multicolumn{1}{|c|}{} & 
	\multicolumn{1}{|c|}{} & \multicolumn{1}{|c|}{} & \multicolumn{1}{|c|}{}\\
	
	\multicolumn{1}{|c|}{($\beta=13/11$)}& \multicolumn{1}{|c|}{0.0087} & 
	\multicolumn{1}{|c|}{1.0961} & \multicolumn{1}{|c|}{0.0087} & \multicolumn{1}{|c|}{1.0919} 
	\\\hline
	\multicolumn{1}{|c|}{$\nu=15$} & \multicolumn{1}{|c|}{} & 
	\multicolumn{1}{|c|}{} & \multicolumn{1}{|c|}{} & \multicolumn{1}{|c|}{}\\
	
	\multicolumn{1}{|c|}{($\beta=9/8$)}& \multicolumn{1}{|c|}{0.0045} & 
	\multicolumn{1}{|c|}{1.0760} & \multicolumn{1}{|c|}{0.0045} & \multicolumn{1}{|c|}{1.0737} 
	\\\hline							
	\end{tabular}

	\label{1tab:compare_hellinger_est_cauchy}
\end{table}	
\section{Summary and concluding remarks}
\label{1sec:summary_remarks}
Projection theorems of Jones et al. ($\mathscr{I}_\alpha$) and Hellinger ($D_\alpha$) divergences tell us that the reverse projection, respectively, on the power-law families $\mathbb{M}^{(\alpha)}$ and $\mathscr{E}^{(\alpha)}$ turns out to be a
forward projection on a ``simpler'' (linear or $\alpha$-linear) family which, in turn, reduces to a linear problem on the underlying probability distribution.
The applicability of these projection theorems known in the literature were limited as they dealt only discrete and canonical models. In this work, we first generalized the associated power-law families to a more general set-up including the continuous case. 
We observed that these two families are related through an escort transformation, apart from the $\alpha\leftrightarrow (2-\alpha)$ transformation studied in \cite{Tsallis09B}.
We then introduced the notion of regularity for these power-law families analogous to the concept of regular exponential family (or full rank family). This makes these families unique from similar families studied in the literature, namely,  the $\phi$-exponential class defined in \cite{Naudts04J}, the $\mathcal{F}_{[\beta h]}$ class defined in \cite{CsiszarM12J} and so on.
We then extended the projection theorems of $\mathscr{I}_{\alpha}$ and $D_{\alpha}$ to these general form of the power-law families by solving the respective estimating equations. We observed that, for regular families, the new estimating equations coincide with the respective projection equations, similar to the ones in the canonical case. We also observed that both the estimating equations were characterized by some specific statistics of the samples.
Such a characterization is well-known in the literature for the pair MLE  and regular exponential family (See, for example, \cite[pp. 149--150]{Brown86B}).
We finally showed that the Student and Cauchy distributions respectively form a regular $\mathbb{M}^{(\alpha)}$ and $\mathscr{E}^{(\alpha)}$, and they are the escort distributions of each other. We then applied the above projection theorems to find generalized estimators for their parameters.
Interestingly, both the  Basu  et al.
and Jones  et al. estimators for the
mean and covariance parameters of a $d$-dimensional Student distribution 
for $\nu >2$ are the same as the
MLE of the respective parameters of a normal distribution, where $\nu$ is the degrees of freedom. We also found a more general class of distributions that includes Student distributions for which both the estimations are the same (Theorem \ref{1thm:same_estimation_under_basu_e_al_and_Jones_et_al}). In \cite[Eq. (38)]{EguchiKK11J}, it was shown that sample mean and sample variance are still the generalized estimators (Jones et al. or Basu et al.) for compactly supported Student distributions (that is, $\nu <-d$), but with the assumption that all samples are from true distribution. We showed that, in the presence of outliers, generalized estimator for the mean parameter might differ from the sample mean (Section \ref{1sec:Jones_et_al_on_student_t_alpha_bigger_1}).
Next we found a class of generalized Hellinger estimators
for the location and scale parameters of Cauchy distributions that involve a continuous estimate for the sample empirical measure. In particular, we found the estimators by using the uniform and Epanechnikov kernel density estimates for the empirical measure. We summarized this by a simulation study where we observed that these estimators are close to the true value of the parameters of Cauchy distributions when their mean and variance are finite. 
It is well-known that
the MLE for Student or Cauchy
distributions do not have a closed form solution.
To overcome this, standard iterative methods such as Newton-Raphson,
Gauss-Newton, EM are used in the literature \cite{Barnett66J1, DempsterLR77J}. However, the sequence of estimators in these
iterative methods may converge to a local maximum  and the rate of convergence is also slow \cite{Barnett66J1, LiuR95J}.
Later some generalized
iterative methods such as ECM, ECME were
proposed, for example, in \cite{LiuR95J}, where the rate of
convergence was made faster than the previous methods. But
again,  they converge only
to a local maximum. These difficulties can be overcome by some of the generalized estimators that
we studied in this paper as they are not only robust but also have closed form solution.

\section*{Acknowledgements}
Atin Gayen is supported by an INSPIRE fellowship of the Department of Science and Technology, Govt. of India. Part of this work was carried out when the authors were with the Indian Institute of Technology Indore. The authors would like to thank Professor Arup Bose for his constructive comments. The authors would also like to thank Professor Michel Broniatowski for the discussions they had with him during his visit to India through the VAJRA programme of Govt. of India. The authors also thank the Editor, Associate Editor and the referees for their valuable comments that improved the presentation of the paper.

\section*{Appendix: Projection theorem of density power divergence}
 The Projection theorem and the Pythagorean property of the more general class of Bregman divergences were established by Csisz\'ar and Mat\'u\v{s} \cite{CsiszarM12J} using tools from convex analysis. The density power divergence $B_\alpha$ is a subclass of the Bregman divergences. However, it is not easy to extract the results for the $B_\alpha$-divergence from \cite{CsiszarM12J}. Ohara and Wada \cite{OharaW10J} also studied this by considering a specific form of the associated parametric family. 
In this section we derive the projection results for the $B_\alpha$-divergence in the discrete case using some elementary tools.
We must point out that the geometry of $B_\alpha$-divergence is quite a natural extension of that of $I$-divergence. Let $\mathbb{S}$ be a finite alphabet set and
$\mathcal{P} := \mathcal{P}(\mathbb{S})$ be the space of all probability distributions on $\mathbb{S}$. 
Then for any $p,q\in \mathcal{P}$, from (\ref{1defn:B_alpha_divergence}), the $B_\alpha$-divergence in the discrete case can be written as
\begin{equation}
\label{1eqn:B_alpha_divergence_in_discrete_form}
B_\alpha (p, q) = 
\frac{\alpha}{1 - \alpha}\sum\limits_{x\in\mathbb{S}}
p (x)q(x)^{\alpha -1}
- \frac{1}{1 - \alpha}\sum\limits_{x\in\mathbb{S}}p(x)^{\alpha} 
+ \sum\limits_{x\in\mathbb{S}} q(x)^{\alpha}.
\end{equation} 
Let us also recall the definitions of reverse and forward projections given in (\ref{1eqn:reverse_projection}) and (\ref{1eqn:forward_projection}). 
For $p\in \mathcal{P}$, we shall denote the support of $p$ as $\text{Supp}(p)$.
For $\mathbb{C}\subset \mathcal{P}$, $\text{Supp}(\mathbb{C})$
is defined as the union of support of members of $\mathbb{C}$. 

We now show the Pythagorean inequality of $B_\alpha$-divergence
in connection with the forward projection on a non-empty closed, convex set (hence compact, since $\mathbb{S}$ is finite). Thus the existence of forward projection always guaranteed, since $B_\alpha$ is lower semi-continuous \cite[Lem. 2.12]{CsiszarM12J}. 
In the following we assume that $\text{Supp}(q) = \mathbb{S}$.
\vspace*{0.2cm}

\begin{theorem}
	\label{1thm:pythagorean_theorem}
	Let $p^*$ be the forward $B_\alpha$-projection of $q$ on a closed and convex
	set $\mathbb{C}$. Then
	\begin{equation}
	\label{1eqn:Pythagorean_inequality}
	B_\alpha (p,q) \geq B_\alpha (p,p^*)+B_\alpha(p^*,q) \quad \forall p\in \mathbb{C}.
	\end{equation}
	Further if $\alpha <1$, $\text{Supp}(\mathbb{C}) = \text{Supp}(p^*)$.
\end{theorem}
\vspace*{0.3cm}

\begin{proof}
	Let $p\in\mathbb{C}$ and define, for $t\in [0,1]$ 
	and $x\in \mathbb{S}$,
	\begin{equation}
	\label{convex_combination_of_p}
	p_t(x) = (1-t) p^*(x) + t p(x).
	\end{equation}
	Since $\mathbb{C}$ is convex, $p_t\in \mathbb{C}$. By mean-value theorem, for each $t\in (0,1)$,
	\begin{eqnarray}
	\label{1eqn:derivative}
	0 & \leq & \frac{1}{t}\big[B_\alpha (p_t, q) -B_\alpha (p^*, q)\big]\nonumber\\
	& = & \frac{1}{t}\big[B_\alpha (p_t, q) -B_\alpha (p_0, q)\big]\nonumber\\
	& = & \frac{d}{dt}B_\alpha (p_t, q)\big |_{t=\tilde{t}} \quad \text{for some } \tilde{t}\in (0,t).
	\end{eqnarray}
	From (\ref{1eqn:B_alpha_divergence_in_discrete_form}), 
	we have
	\begin{eqnarray*}
		\dfrac{d}{dt} B_\alpha (p_t, q)=\dfrac{\alpha}{\alpha-1}\sum\limits_{x\in \mathbb{S}} \big[p(x)-p^*(x)\big]\big[p_t(x)^{\alpha-1}-q(x)^{\alpha-1}\big].
	\end{eqnarray*}
	Therefore (\ref{1eqn:derivative}) implies
	\begin{equation}
	\label{1eqn:derivative_calculation}
	\dfrac {\alpha}{\alpha-1}\sum\limits_{x\in \mathbb{S}} \big[p(x)-p^*(x)\big]\big[p_{\tilde{t}} (x)^{\alpha-1}-q(x)^{\alpha-1}\big] \geq 0.
	\end{equation}
	Hence, as $t\downarrow 0$, we have
	\begin{equation}
	\label{1eqn:pythagorean_equivalent}
	\dfrac {\alpha}{\alpha-1}\sum\limits_{x\in \mathbb{S}} \big[p(x)-p^*(x)\big]\big[p^*(x)^{\alpha-1}-q(x)^{\alpha-1}\big] \geq 0,
	\end{equation}
	which implies
	\begin{equation*}
	B_\alpha (p,q) \geq B_\alpha (p,p^*)+B_\alpha(p^*,q).
	\end{equation*}
	
	If $\text{Supp}(p^*)\neq \text{Supp}(\mathbb{C})$, then there exists
	$p\in\mathbb{C}$ and $x\in\mathbb{S}$ such that $p(x)>0$ but
	$p^*(x) = 0$. Hence if $\alpha <1$, then the left-hand side of (\ref{1eqn:derivative_calculation}) goes to $-\infty$ as $t\downarrow 0$, which contradicts (\ref{1eqn:derivative_calculation}). This proves the claim.
\end{proof}
\begin{remark}
	\label{1rem:support_of_q_notequal_support_of_projection}
	If $\alpha>1$, in general, $\text{Supp}(p^*)\neq \text{Supp}(\mathbb{C})$.
	\cite[Ex.~2]{KumarS15J2} serves as a counterexample here as well.
	It follows from the following fact. Since $\mathbb{S}$ is finite, the $B_\alpha$-divergence can be
	written as
	\begin{equation*}
	B_\alpha(p,u) = \frac{1}{1-\alpha} |\mathbb{S}|^{1-\alpha}-\frac{1}{1-\alpha}
	\sum\limits_{x\in\mathbb{S}} p(x)^\alpha,
	\end{equation*}
	where $u$ is the uniform distribution on $\mathbb{S}$ and
	$|~\mathbb{S}~|$ denotes the cardinality of $\mathbb{S}$. This
	implies
	\begin{equation*}
	\displaystyle\arg\min_{p\in\mathbb{C}}B_\alpha(p,u) = \displaystyle\arg\max_{p\in\mathbb{C}} H_\alpha(p),
	\end{equation*}
	where $H_\alpha(p):= \frac{1}{1-\alpha}\ln \sum_{x\in\mathbb{S}} p(x)^\alpha$, the R\'enyi entropy of $p$ of order $\alpha$. That is,
	forward $B_\alpha$-projection of the uniform distribution
	on $\mathbb{C}$ is same as the maximizer of R\'enyi entropy
	on $\mathbb{C}$. The same is true when $B_\alpha$ is replaced by 
	$\mathscr{I}_\alpha$ or $D_\alpha$.
\end{remark}

We will now present a situation when the equality holds in (\ref{1eqn:Pythagorean_inequality}).

\begin{definition}
	\label{1defn:linear_family}
	The {\em linear family}, determined by $k$ real valued functions $f_i, i\in\{1,\ldots,k\}$ on $\mathbb{S}$ and $k$ real numbers $a_i, i\in\{1,\ldots,k\}$, is defined as
	\begin{equation}
	\label{1eqn:linear_family}
	\mathbb{L} := \Big\{ p\in \mathcal{P} : \sum\limits_{x\in \mathbb{S}} p(x)f_i(x) = a_i,\quad r\in\{1,\ldots,k\} \Big\}.
	\end{equation} 
\end{definition} 
\begin{theorem}
	\label{1thm:pythagorean_equality}
	Let $p^*$ be the forward $B_\alpha$-projection of $q$ on $\mathbb{L}$. The following hold.
	\begin{enumerate}
		\item[(a)] If $\alpha<1$ then the Pythagorean equality holds, that is,
		\begin{equation}
		\label{1eqn:pythagorean_equality}
		B_\alpha (p,q) = B_\alpha (p,p^*)+B_\alpha(p^*,q) \quad\forall p\in \mathbb{L}.
		\end{equation}
		
		\item[(b)] If $\alpha>1$ and $\text{Supp}(p^*) = \text{Supp}(\mathbb{L})$ then the Pythagorean equality (\ref{1eqn:pythagorean_equality}) holds.
	\end{enumerate}
\end{theorem}
\vspace*{0.2cm}

\begin{proof}
	(a) Let $p_t$ be as in (\ref{convex_combination_of_p}). Since $\text{Supp}(p^*) = \text{Supp}(\mathbb{L})$, there exists $t' < 0$ such that $p_t = (1 -t)p^* + tp\in \mathbb{L}$ for $t\in (t',0)$. Hence,
	proceeding as in Theorem \ref{1thm:pythagorean_theorem},
	for every $t\in (t',0)$, there exists $\tilde{t}\in (t,0)$ such that
	\begin{equation*}
	\frac {\alpha}{\alpha-1}\small\sum\limits_{x\in \mathbb{S}} \big[p(x)-p^*(x)\big]\big[p_{\tilde{t}} (x)^{\alpha-1}-q(x)^{\alpha-1}\big] \le 0.
	\end{equation*}
	Hence we get (\ref{1eqn:pythagorean_equivalent}) with a reversed inequality. Thus we have equality in (\ref{1eqn:pythagorean_equivalent}). Hence we have  (\ref{1eqn:pythagorean_equality}).

	(b) Similar to (a).
\end{proof}
\vspace*{0.2cm}

When $\alpha >1$, equality in (\ref{1eqn:pythagorean_equality}) does not hold in general. In the following we present an example where the equality in (\ref{1eqn:pythagorean_equality}) does not hold.
\vspace*{0.2cm}

\begin{example}
	\label{1expl:pyth_ineq_strict}
	Let
	$\alpha=2$, $\mathbb{S}=\{1,2,3,4\}$ and
	$$\mathbb{L}:= \{ p\in \mathcal{P}:p(1)-3p(2)-5p(3)-6p(4)=0\}.$$
	In view of Remark \ref{1rem:support_of_q_notequal_support_of_projection}
	and \cite[Ex.~2]{KumarS15J2},
	we see that $p^*=[3/4,1/4,0,0]^top$ is the forward
	$B_\alpha$-projection of the uniform distribution $u$ on $\mathbb{L}$. 
	However there exists $p\in\mathbb{L}$, say $p = [157/200,97/600,1/50,1/30]^top$, that satisfies only
	the strict inequality in (\ref{1eqn:Pythagorean_inequality}). The
	issue here is that $\text{Supp}(p^*)$ $\subsetneq\text{Supp}(\mathbb{L})$.
\end{example}
\vspace*{0.2cm}

We now find an explicit expression of the forward $B_\alpha$-projection in both the cases $\alpha<1$ and $\alpha >1$ separately.
\vspace*{0.2cm}

\begin{theorem}
	\label{1thm:forward_projection}
	Let $q\in \mathcal{P}$ and let $\mathbb{L}$ be a linear family of probability 
	distributions as in (\ref{1eqn:linear_family}).
	\begin{enumerate}
		\item[(a)] If $\alpha<1$, the forward $B_\alpha$-projection $p^*$ of $q$ on $\mathbb{L}$ satisfies
		\begin{eqnarray}
		\label{1eqn:forward_projection_alpha<1}
		p^*(x) = \big[q(x)^{\alpha-1}+ F + \theta^top  f(x)\big]^{\frac{1}{\alpha-1}} \quad\forall x\in \text{Supp}(\mathbb{L}),
		\end{eqnarray}
		with $\theta := [\theta_1,\dots,\theta_k]^top$, $f:= [f_1,\dots,f_k]^top$ where
		$\theta_1,\dots,\theta_k$ are some scalars and $F$ is a constant.
		\vspace*{0.2cm}
		
		\item[(b)] If $\alpha>1$, the forward $B_\alpha$-projection $p^*$ of $q$ on $\mathbb{L}$ satisfies
		\begin{eqnarray}
		\label{1eqn:forward_projection_alpha>1}
		p^*(x) = \big[q(x)^{\alpha-1}+ F +\theta^top f(x)
		\big]_+^{\frac{1}{\alpha-1}} \quad\forall x\in \mathbb{S},
		\end{eqnarray}
		where $\theta$, $f$ and $F$ are as in (a).
	\end{enumerate}
\end{theorem}
\vspace*{0.2cm}

\begin{proof}
	\begin{itemize}
		\item[(a)] The proof is similar to that for
		$I$-divergence in \cite[Th. 3.2]{CsiszarS04B}.
		The linear family in Definition \ref{1defn:linear_family} can be re-written as
		\begin{equation}
		\label{1eqn:linear_family_equivalent}
		\mathbb{L}:= \Big\{p\in\mathcal{P} : \sum\limits_{x\in \text{Supp}(\mathbb{L})} p(x)f_i(x) = a_i, \quad i\in\{1,\dots,k\}\Big\}.
		\end{equation}
		Let $\mathbb{H}$ be the subspace of $\mathbb{R}^{|\text{Supp}(\mathbb{L})|}$ spanned by the $k$ vectors $f_1(\cdot)-a_1, \dots, f_k(\cdot)-a_k$. Then every $p\in \mathbb{L}$ can be thought of a $|\text{Supp}(\mathbb{L})|$-dimensional vector in $\mathbb{H}^{\perp}$. Hence $\mathbb{H}^{\perp}$ is a subspace of $\mathbb{R}^{|\text{Supp}(\mathbb{L})|}$ that contains a vector whose components are strictly positive since $p^*\in\mathbb{L}$ and $\text{Supp}(p^*)=\text{Supp}
		(\mathbb{L})$. It follows that $\mathbb{H}^{\perp}$ is spanned by its	probability vectors. From (\ref{1eqn:pythagorean_equivalent}) we see that (\ref{1eqn:pythagorean_equality}) is equivalent to
		\begin{equation}
		\label{1eqn:pythagorean_equality_equivalent}
		\sum\limits_{x\in \mathbb{S}} \big[p(x)-p^*(x)\big]\big[p^*(x)^{\alpha-1}-q(x)^{\alpha-1}\big] = 0\quad\forall p\in\mathbb{L}.
		\end{equation}
		This implies that the vector
		\begin{equation*}
		p^*(\cdot)^{\alpha-1} - q(\cdot)^{\alpha-1} - \sum\limits_x p^*(x)\big[p^*(x)^{\alpha-1}-q(x)^{\alpha-1}
		\big] \in (\mathbb{H}^\perp)^\perp=\mathbb{H}.
		\end{equation*}	
		Hence
		\begin{eqnarray*}
			p^*(x)^{\alpha-1} - q(x)^{\alpha-1} - \sum\limits_x p^*(x)\big[p^*(x)^{\alpha-1} - q(x)^{\alpha-1}\big]	= \sum\limits_{i=1}^k c_i\big[f_i(x)-a_i\big]
			\quad\forall x\in \text{Supp}(\mathbb{L})
		\end{eqnarray*}
		for some scalars $c_1,\ldots,c_k$. This implies (\ref{1eqn:forward_projection_alpha<1}) for appropriate choices of $F$ and $\theta_1,\ldots,\theta_k$.
		\vspace{0.1cm}
		
		\item[(b)] The proof of this case is similar to that of $\mathscr{I}_\alpha$-divergence \cite[Th. 14(b)]{KumarS15J2}. The optimization problem concerning the forward $B_\alpha$-projection is
		\begin{align}
		\min_p \, & B_{\alpha}(p,q)\label{1eqn:min}\\
		\mbox{subject to } &\sum\limits_x p(x)f_i(x) = a_i, \quad i\in\{1,\ldots,k\},\label{1eqn:linear_constraints}\\
		&\sum\limits_x p(x)       = 1, \label{1eqn:probability_constraint}\\
		&p(x)                    \ge 0 \quad\forall x \in \mathbb{S}. \label{1eqn:positivity_constraints}
		\end{align}
		Hence by \cite[Prop.~3.3.7]{Bertsekas03B}, there exists Lagrange multipliers $\lambda_1,\dots,\lambda_k$, 
		$\nu$ and $(\mu(x),$ $x\in \mathbb{S})$, respectively, associated with the above constraints such that, for $x\in \mathbb{S}$,
		\begin{eqnarray}
		\label{1eqn:lagrange}
		\dfrac{\partial}{\partial p(x)}B_\alpha(p,q)\Big|_{p=p^*} & = & - \sum\limits_{i=1}^k \lambda_i [f_i(x) - a_i] + \mu(x) - \nu,\\
		\label{1eqn:feasibility_condition}
		\mu(x) & \ge & 0,\\
		\label{1eqn:slackness_condition}
		\mu(x)p^*(x) & = & 0.
		\end{eqnarray}
		Since
		\begin{equation}
		\label{1eqn:partial_derivative}
		\dfrac{\partial}{\partial p(x)}B_\alpha(p,q)=\dfrac{\alpha}{\alpha-1}
		\big[p(x)^{\alpha-1}-q(x)^{\alpha-1}\big],
		\end{equation}
		(\ref{1eqn:lagrange}) can be re-written as
		\begin{equation}
		\label{1eqn:lagrange1}
		\dfrac{\alpha}{\alpha-1}\big[p^*(x)^{\alpha-1} - q(x)^{\alpha-1}\big] =         - \sum\limits_{i=1}^k\lambda_i\big[f_i(x)-a_i\big] + \mu(x) - \nu\quad \text{ for } x\in \mathbb{S}.
		\end{equation}
		Multiplying both sides by $p^*(x)$ and summing over all $x\in \mathbb{S}$, we get
		\begin{equation*}
		\nu = \dfrac{\alpha}{\alpha-1}\sum\limits_{x\in \mathbb{S}} p^*(x)\big[ q(x)^{\alpha-1}-p^*(x)^{\alpha-1}\big].
		\end{equation*}
		For $x\in \text{Supp}(p^*)$, from (\ref{1eqn:slackness_condition}),
		we must have 
		$\mu(x)=0$. Then, from (\ref{1eqn:lagrange1}), we have
		\begin{align}
		\label{1eqn:projection_alpha>1_equivalent1}
		p^*(x)^{\alpha-1} = q(x)^{\alpha-1} - \dfrac{\alpha-1}{\alpha} \sum\limits_{i=1}^k\lambda_i
		\big[f_i(x)-a_i\big] - \dfrac{\alpha-1}{\alpha}\nu.
		\end{align}
		If $p^*(x)=0$, from (\ref{1eqn:lagrange1}) we get
		\begin{eqnarray}
		\label{1eqn:projection_alpha>1_equivalent2}
		q(x)^{\alpha-1} - \dfrac{\alpha-1}{\alpha} \sum\limits_{i=1}^k\lambda_i\big[f_i(x)-a_i\big] - \dfrac{\alpha-1}{\alpha}\nu = -
		\dfrac{\alpha-1}{\alpha}\mu(x)\leq 0.
		\end{eqnarray}
		Combining (\ref{1eqn:projection_alpha>1_equivalent1}) and (\ref{1eqn:projection_alpha>1_equivalent2}) we get (\ref{1eqn:forward_projection_alpha>1}).
	\end{itemize}
\end{proof}

Theorem \ref{1thm:forward_projection} suggests us to define a parametric family of probability distributions that is a generalization of the usual exponential family. We call it a $\mathbb{B}^{(\alpha)}$-family. First we formally define this family and then show an orthogonality relationship between this family and the linear family. As a consequence we will also show that the reverse $B_{\alpha}$-projection on a 
$\mathbb{B}^{(\alpha)}$-family is same as the forward projection on a linear family.
\vspace*{0.2cm}

\begin{definition}
	\label{1defn:twin_alpha_powerlaw_family}
	Let $q\in\mathcal{P}$ where $\text{Supp}(q)=\mathbb{S}$ for $\alpha>1$ and $f=[f_1,\dots,f_k]^top$ where $f_i$ for $i\in\{1,\dots,k\}$ be real valued function on $\mathbb{S}$.
	The $k$-parameter canonical $\mathbb{B}^{(\alpha)} := \mathbb{B}^{(\alpha)}(q, f)$ family of probability distributions characterized by $q$ and $f$ is 
	defined by $\mathbb{B}^{(\alpha)} = \{p_\theta : \theta\in\Theta\}\subset\mathcal{P}$ where
	\begin{eqnarray}
	\label{1eqn:B_alpha_family_formula}
	p_{\theta}(x) = \big[q(x)^{\alpha - 1} + F(\theta) + \theta^top f(x)\big]^{\frac{1}{\alpha-1}}>0 \quad\text{for } x\in\mathbb{S},
	\end{eqnarray}
	for some $F:\Theta\to\mathbb{R}$ and $\Theta$ is the subset of $\mathbb{R}^k$ for which $p_\theta \in\mathcal{P}$.
\end{definition}
\vspace*{0.2cm}

\begin{remark}
	\label{1rem:B_alpha_family}
	\begin{itemize}
		
		\item[(a)] Observe that $\mathbb{B}^{(\alpha)}$-family is a special case of the family $\mathcal{F}_{[\beta h]}$ in \cite[Eq.~(28)]{CsiszarM12J} with
		$h= q$ and $\beta(\cdot, t) = \frac{1}{\alpha -1}[t^{\alpha}-\alpha t +\alpha -t]$.
		\vspace*{0.2cm}
		
		\item[(b)] The family
		depends on the reference measure $q$  only in a loose manner in the sense that any other member of the family can play the role of $q$. The change of reference measure only corresponds to a translation of the parameter space.
		(This fact is true for the $\mathbb{M}^{(\alpha)}$-family \cite[Prop. 22]{KumarS15J2}.)
	\end{itemize}
\end{remark}
\vspace*{0.2cm}

The following theorem and its corollary together establish an ``orthogonality" relationship between the 
$\mathbb{B}^{(\alpha)}$-family and the associated linear family.
\vspace*{0.2cm}

\begin{theorem}
	\label{1thm:orthogonality1}
	Let $\alpha\in (0,1)$. Consider a $\mathbb{B}^{(\alpha)}$-family as in Definition \ref{1defn:twin_alpha_powerlaw_family} and let $\mathbb{L}$ be the corresponding linear family determined by the same functions $f_i, i\in\{1,\ldots,k\}$ and some constants $a_i,i\in\{1,\ldots,k\}$ as in (\ref{1eqn:linear_family}). If
	$p^*$ is the forward $B_\alpha$-projection of $q$ on $\mathbb{L}$ then we have the following:
	\begin{enumerate}
		\item[(a)] $\mathbb{L}\cap \text{cl}(\mathbb{B}^{(\alpha)}) = \{ p^*\}$ and
		\begin{eqnarray*}
			\label{1eqn:pythagorean_equality1}
			B_\alpha (p,q) = B_\alpha (p,p^*) + B_\alpha(p^*,q)\quad \forall p\in \mathbb{L}.
		\end{eqnarray*}
		\item[(b)] Further, if $\text{Supp}(\mathbb{L}) = \mathbb{S}$, then $\mathbb{L}\cap\mathbb{B}^{(\alpha)}=\{ p^*\}$.
	\end{enumerate}
\end{theorem}
\vspace*{0.2cm}

\begin{proof}
	By Theorem \ref{1thm:forward_projection}, the forward $B_\alpha$-projection $p^*$ of $q$ on $\mathbb{L}$ is in $\mathbb{B}^{(\alpha)}$. This implies that $p^*\in \mathbb{L}\cap \mathbb{B}^{(\alpha)}$. Hence it
	suffices to prove the following:
	\begin{enumerate}
		\item[(i)] Every $\tilde{p}\in \mathbb{L}\cap \text{cl}(\mathbb{B}^{(\alpha)})$ satisfies (\ref{1eqn:pythagorean_equality}) with $\tilde{p}$ in place of $p^*$.
		\item[(ii)] $\mathbb{L}\cap \text{cl}(\mathbb{B}^{(\alpha)})$ is non-empty.
	\end{enumerate}
	
	We now proceed to prove both (i) and (ii).
	\vspace*{0.2cm}
	
	(i) Let $\tilde{p}\in \mathbb{L}\cap \text{cl}(\mathbb{B}^{(\alpha)})$. As $\tilde{p}\in\text{cl}(\mathbb{B}^{(\alpha)})$, this implies that there exists a sequence $\{p_n\}\subset\mathbb{B}^{(\alpha)}$ such that $p_n\rightarrow \tilde{p}$ as $n\to\infty$. Since $p_n\in \mathbb{B}^{(\alpha)}$, we can write
	\begin{equation}
	\label{1eqn:P_n}
	p_n(x)^{\alpha-1} = q(x)^{\alpha -1} + F_n + \theta_{n}^{T} f(x)\quad\forall x\in \mathbb{S}
	\end{equation}
	for some constants $\theta_{n} := [\theta_{n}^{(1)},\dots,\theta_{n}^{(k)}]^top\in \mathbb{R}^k$
	and $F_n$. Now for any $p\in \mathbb{L}$ we have, from the definition of linear family, $\sum_{x\in \mathbb{S}}p(x)f_i(x) = a_i, i\in\{1,\ldots,k\}$. Since $\tilde{p}\in \mathbb{L}$, we also have
	$\sum\limits_{x\in \mathbb{S}}\tilde{p} (x)f_i(x) = a_i, i\in\{1,\ldots,k\}$. Multiplying both sides of (\ref{1eqn:P_n}) by $p$ and $\tilde{p}$ separately, we get
	\begin{equation*}
	\sum\limits_{x\in \mathbb{S}} p(x) p_n(x)^{\alpha-1} = \sum\limits_{x\in \mathbb{S}}p(x)q(x)^{\alpha -1} + F_n + \sum\limits_{i=1}^k\theta_n^{(i)}a_i
	\end{equation*}
	and
	\begin{equation*}
	\sum\limits_{x\in \mathbb{S}} \tilde{p}(x) p_n(x)^{\alpha-1} = \sum\limits_{x\in \mathbb{S}}\tilde{p} (x)q(x)^{\alpha -1} + F_n + \sum\limits_{i=1}^k\theta_n^{(i)}a_i.
	\end{equation*}
	Combining the above two equations, we get
	\begin{equation*}
	\sum\limits_{x\in \mathbb{S}}\big[p(x) -\tilde{p}(x)\big]\big[p_n(x)^{\alpha-1} -q(x)^{\alpha-1}
	\big] = 0.
	\end{equation*}
	As $n\to\infty$, the above becomes
	\begin{equation*}
	\sum\limits_{x\in \mathbb{S}}\big[p(x) -\tilde{p}(x)\big]\big[\tilde{p}(x)^{\alpha-1} - q(x)^{\alpha-1}\big] = 0,
	\end{equation*}
	which is equivalent to (\ref{1eqn:pythagorean_equality}).
	\vspace*{0.2cm}
	
	(ii) Let $p_n^*$ be the forward $B_\alpha$-projection of $q$ on the linear family
	\begin{equation*}
	\mathbb{L}_n :=
	\Big\{p:\sum\limits_{x\in \mathbb{S}} p(x)f_i(x) = \Big(1-\frac{1}{n}\Big)a_i+\frac{1}{n}\sum\limits
	_{x\in \mathbb{S}} q(x)f_i(x), \quad i\in\{1,\ldots,k\}\Big\}
	\end{equation*}
	(see Figure \ref{1fig:B_alpha_projection}).
	\begin{figure}
		\centering
		\includegraphics[width=6cm, height=5.5cm]{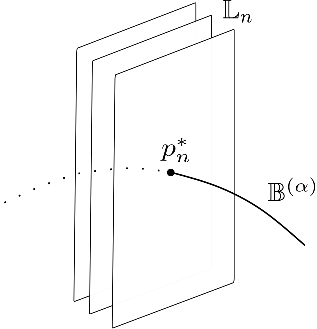}
		\caption{Orthogonality between $\mathbb{B}^{(\alpha)}$-family and
			the linear family $\mathbb{L}_n$.}
		\label{1fig:B_alpha_projection}
	\end{figure}
	By construction $\left(1-\frac{1}{n}\right) p+\frac{1}{n}q\in \mathbb{L}_n$ for any $p\in \mathbb{L}$. Hence, since $\text{Supp}(q) = \mathbb{S}$, we have $\text{Supp}(\mathbb{L}_n) = \mathbb{S}$. Since 
	$\mathbb{L}_n$ is also characterized by the same functions $f_i, i\in\{1,\ldots,k\}$, we have $p_n^*\in \mathbb{B}^{(\alpha)}$ for every $n\in \mathbb{N}$. Hence limit of any convergent sub-sequence of $\{ p_n^*\}$ belongs to $\text{cl}(\mathbb{B}^{(\alpha)})\cap \mathbb{L}$. Thus $\text{cl}(\mathbb{B}^{(\alpha)})\cap \mathbb{L}$ is non-empty. This completes the proof.
\end{proof}

\begin{corollary}
	\label{1cor:orthogonality}
	Let $\alpha\in (0,1)$. Let $\mathbb{L}$ and $\mathbb{B}^{(\alpha)}$ be characterized by the same functions $f_i,i\in\{1,\ldots,k\}$. Then $\mathbb{L}\cap \text{cl}(\mathbb{B}^{(\alpha)}) = \{p^*\}$ and
	\begin{eqnarray}
	\label{1eqn:orthogonality1}
	B_\alpha (p,q) = B_\alpha (p,p^*)+B_\alpha(p^*,q)\quad\forall p\in \mathbb{L},\quad\forall q\in \text{cl}(\mathbb{B}^{(\alpha)}).
	\end{eqnarray}
\end{corollary}
\vspace*{0.2cm}

\begin{proof}
	By Theorem \ref{1thm:orthogonality1}, we have $\mathbb{L}\cap \text{cl}(\mathbb{B}^{(\alpha)}) = \{p^*\}$. In view of Remark \ref{1rem:B_alpha_family}(b), notice that every member of $\mathbb{B}^{(\alpha)}$ has the same projection on $\mathbb{L}$, namely $p^*$. Hence (\ref{1eqn:pythagorean_equality}) holds for every $q\in\mathbb{B}^{(\alpha)}$. Thus we only need to prove (\ref{1eqn:pythagorean_equality}) for every $q\in\text{cl}(\mathbb{B}^{(\alpha)})\setminus\mathbb{B}^{(\alpha)}$. Let 
	$q\in\text{cl}(\mathbb{B}^{(\alpha)})\setminus\mathbb{B}^{(\alpha)}$. There exists $\{q_n\}\subset\mathbb{B}^{(\alpha)}$ such that $q_n\rightarrow q$. Hence for any $p\in \mathbb{L}$, we have
	\begin{equation}
	\label{1eqn:pythagorean_equality_in_sequence_of_p_n}
	B_\alpha (p,q_n) = B_\alpha (p,p^*) + B_\alpha (p^*,q_n)\quad \forall n\in\mathbb{N}.
	\end{equation}
	Since for a fixed $p$, $q\mapsto B_\alpha (p, q)$ is continuous as a function from $\mathcal{P}$ to $[0, \infty]$, taking limit as $n\to\infty$
	on both sides of (\ref{1eqn:pythagorean_equality_in_sequence_of_p_n}), we have 
	\begin{eqnarray*}
		\label{1eqn:orthogonality}
		B_\alpha (p,q) = B_\alpha (p,p^*)+B_\alpha(p^*,q)\quad\forall p\in \mathbb{L},\quad\forall q\in \text{cl}(\mathbb{B}^{(\alpha)}).
	\end{eqnarray*}
	This completes the proof.
\end{proof}
\vspace*{0.2cm}

Theorem \ref{1thm:orthogonality1} does not hold, in general, for $\alpha>1$ as shown in the following example.
\vspace*{0.2cm}

\begin{example}
	\label{1expl:B_alpha_does_not_intersect_L}
	Let $\alpha,\mathbb{S},\mathbb{L}$ and $u$ be as in Example	\ref{1expl:pyth_ineq_strict}. Then the associated
	$\mathbb{B}^{(\alpha)}$-family is given by
	\begin{equation*}
	\mathbb{B}^{(\alpha)} = \{p_\theta : p_\theta(x) = u(x)+ F(\theta) + \theta
	f(x), \quad \forall x\in \mathbb{S}\},
	\end{equation*}
	where $f= [1,-3,-5,-6]^top$, $F(\theta)= \frac{13\theta}{4}$ and $\theta\in (-\frac{1}{17},\frac{1}{11})$. Then we have
	\begin{eqnarray*}
		\mathbb{B}^{(\alpha)} =\big \{p_\theta:\theta\in (-\tfrac{1}{17},\tfrac{1}{11})
		\big\},\\
		\text{cl}(\mathbb{B}^{(\alpha)}) = \big\{p_\theta:\theta\in 
		[-\tfrac{1}{17},\tfrac{1}{11}]\big\},
	\end{eqnarray*}
	where $p_\theta = \big[(\frac{1}{4}+\frac{17\theta}{4}),
	(\frac{1}{4}+\frac{\theta}{4}), (\frac{1}{4}-\frac{7\theta}{4}),
	(\frac{1}{4}-\frac{11\theta}{4})\big]^top$. If $p_\theta\in 
	\text{cl}(\mathbb{B}^{(\alpha)}) \cap \mathbb{L}$ then $\sum\limits_{x\in 
		\mathbb{S}} p_\theta(x) f(x)$ $=0$. This implies $\theta = \frac{13}{115}$, which
	is outside the range of $\theta$. Hence 
	$\text{cl}(\mathbb{B}^{(\alpha)}) \cap \mathbb{L}= \emptyset$.
\end{example}
\vspace*{0.2cm}

The following theorem tells us that a reverse $B_\alpha$-projection on a $\mathbb{B}^{(\alpha)}$-family can be turned into a forward $B_\alpha$-projection on the associated linear family. We shall refer this as the {\em projection theorem} for the $B_\alpha$-divergence. This theorem is analogous to the one for $I$-divergence
\cite[Th. 3.3]{CsiszarS04B}, $\mathscr{I}_\alpha$-divergence
\cite[Th. 18]{KumarS15J2} and $D_\alpha$-divergence \cite[Th. 6]{KumarS16J}.
\vspace*{0.2cm}

\begin{theorem}
	\label{1thm:orthogonality2}
	Let $\alpha\in (0,1)$. Let $X_1^n:=
	(X_1,\dots, X_n)\in \mathbb{S}^n$. Let 
	$p_n$ be the empirical probability measure of $X_1^n$ and let
	\begin{equation}
	\label{1eqn:empirical_linear_family}
	\widehat{\mathbb{L}}_n := \Big\{p\in\mathcal{P} : \sum\limits_{x\in \mathbb{S}} p(x)f_i(x) = \bar{f_i}, \quad i\in\{1,\ldots,k\}\Big\},
	\end{equation}
	where $\bar{f_i} = \frac{1}{n}\sum_{j=1}^n f_i(X_j), i\in\{1,\ldots,k\}$. Let $p^*$ be the forward $B_\alpha$-projection of $q$ on $\widehat{\mathbb{L}}_n$. Then the following hold.
	\begin{enumerate}
		\item[(i)] If $p^*\in \mathbb{B}^{(\alpha)}$, then $p^*$ is the	reverse $B_\alpha$-projection of $p_n$ on $\mathbb{B}^{(\alpha)}$.
		\item[(ii)] If $p^*\notin \mathbb{B}^{(\alpha)}$, then $p_n$ does not have a reverse $B_\alpha$-projection on $\mathbb{B}^{(\alpha)}$. However, $p^*$ is the reverse $B_\alpha$-projection of $p_n$ on \text{cl}$(\mathbb{B}^{(\alpha)})$.
	\end{enumerate}
\end{theorem}

\begin{proof}
	Let us first observe that $\widehat{\mathbb{L}}_n$ is constructed so that $p_n\in\widehat{\mathbb{L}}_n$. Since the families $\widehat{\mathbb{L}}_n$ and $\mathbb{B}^{(\alpha)}$ are defined by the same functions $f_i$, $i\in\{1,\dots,k\}$, by Corollary \ref{1cor:orthogonality}, we have $\widehat{\mathbb{L}}\cap \text{cl}(\mathbb{B}^{(\alpha)}) = \{p^*\}$ and 
	\begin{equation*}
	\label{1eqn:pythagorean_equality2}
	B_\alpha(p_n,q) = B_\alpha (p_n,p^*) + B_\alpha(p^*,q) \quad\forall q\in \text{cl}(\mathbb{B}^{(\alpha)}).
	\end{equation*}
	Hence it is clear that the minimizer of $B_\alpha(p_n, q)$ over $q\in\text{cl}(\mathbb{B}^{(\alpha)})$ is same as the minimizer of $B_\alpha (p^* , q)$ over 
	$q\in\text{cl}(\mathbb{B}^{(\alpha)})$ (Notice that this statement is also true with $\text{cl}(\mathbb{B}^{(\alpha)})$ replaced by $\mathbb{B}^{(\alpha)}$). But $B_\alpha (p^* , q)$ over $q\in\text{cl}(\mathbb{B}^{(\alpha)})$ is uniquely minimized by $q = p^*$. Hence if $p^*\notin \mathbb{B}^{(\alpha)}$, since minimum value of $B_\alpha(p_n, q)$ over $q\in\text{cl}(\mathbb{B}^{(\alpha)})$ is same as that of $B_\alpha(p_n, q)$ over $q\in\mathbb{B}^{(\alpha)}$, the later is not attained on $\mathbb{B}^{(\alpha)}$.
\end{proof}

\begin{remark}
	\label{1rem:projection_B_alpha_alpha_greater_1}
	Theorems \ref{1thm:orthogonality1}, \ref{1thm:orthogonality2}, and Corollary \ref{1cor:orthogonality} continue to hold for $\alpha >1$ as well if attention is restricted to probability measures with strictly positive components and the existence of $p^*$ is guaranteed.
\end{remark}

\bibliographystyle{IEEEtranS}
\bibliography{myjmva}
\end{document}